\documentclass[11pt, reqno]{amsart}
\usepackage{amssymb, amsthm, amsmath, amsfonts,mathtools}
\usepackage{array, epsfig}
\usepackage{bbm}

\usepackage{hyperref}
\usepackage[numbers,square]{natbib}
\allowdisplaybreaks

\numberwithin{equation}{section}

  \evensidemargin
\oddsidemargin
\makeatletter
\def\@tocline#1#2#3#4#5#6#7{\relax
  \ifnum #1>\c@tocdepth 
  \else
    \par \addpenalty\@secpenalty\addvspace{#2}%
    \begingroup \hyphenpenalty\@M
    \@ifempty{#4}{%
      \@tempdima\csname r@tocindent\number#1\endcsname\relax
    }{%
      \@tempdima#4\relax
    }%
    \parindent\z@ \leftskip#3\relax \advance\leftskip\@tempdima\relax
    \rightskip\@pnumwidth plus4em \parfillskip-\@pnumwidth
    #5\leavevmode\hskip-\@tempdima
      \ifcase #1
       \or\or \hskip 1em \or \hskip 2em \else \hskip 3em \fi%
      #6\nobreak\relax
    \dotfill\hbox to\@pnumwidth{\@tocpagenum{#7}}\par
    \nobreak
    \endgroup
  \fi}
\makeatother

\newcommand{\E}{\mathbb E}

\newcommand{\R}{\mathbb{R}}
\newcommand{\N}{\mathbb{N}}

\newcommand{\C}{\mathbb{C}}

\renewcommand{\P}{\mathbb{P}}

\renewcommand{\Re}{\operatorname{Re}}

\newcommand{\Vol}{\mathop{\mathrm{Vol}}\nolimits}

\newcommand{\eps}{\varepsilon}

\newcommand{\toweak}{\overset{w}{\underset{n\to\infty}\longrightarrow}}

\newcommand{\ind}{\mathbbm{1}}

\theoremstyle{plain}
\newtheorem{theorem}{Theorem}[section]
\newtheorem{lemma}[theorem]{Lemma}

\newtheorem{proposition}[theorem]{Proposition}
\newtheorem{conjecture}[theorem]{Conjecture}

\theoremstyle{definition}

\newtheorem{example}[theorem]{Example}

\newtheorem{remark}[theorem]{Remark}

\theoremstyle{remark}

\begin{document}

\author{Zakhar Kabluchko}
\address{Zakhar Kabluchko: Institut f\"ur Mathematische Stochastik,
Universit\"at M\"unster,
Orl\'eans-Ring 10,
48149 M\"unster, Germany}
\email{zakhar.kabluchko@uni-muenster.de}

\author{Alexander Marynych}
\address{Alexander Marynych: Faculty of Computer Science and Cybernetics, Taras Shevchenko National University of Kyiv, Volodymyrska 60, 01601 Kyiv, Ukraine; Department of Mathematics, Kyiv School of Economics, Mykoly Shpaka 3, 03113 Kyiv, Ukraine}
\email{marynych@knu.ua, o.marynych@kse.org.ua}

\author{Joscha Prochno}
\address{Joscha Prochno: Faculty of Computer Science and Mathematics, University of Passau, Dr.-Hans-Kapfinger-Str. 30, 94032 Passau, Germany}
\email{joscha.prochno@uni-passau.de}

\title[Intrinsic volumes of \(\ell_p\)-balls]{Intrinsic volumes of \(\ell_p\)-balls and a continuum of Maxwell--Poincar\'e--Borel laws for their curvature measures}



\begin{abstract}
For $p>1$, we derive explicit formulas for the intrinsic volumes $V_0(\mathbb B_p^n),\dots,V_{n-1}(\mathbb B_p^n)$
of the $n$-dimensional \(\ell_p\)-balls
\[
\mathbb B_p^n
=
\{x\in\mathbb R^n:\ |x_1|^p+\ldots+|x_n|^p\le 1\}
\]
and, more generally, of their coordinate-weighted analogues. The formula is given in terms of a one-dimensional integral involving the special function
\[
\mathcal F_p(t;\nu)
=
\int_{\mathbb R}|u|^\nu e^{-|u|^p-t|u|^{2p-2}}\,du.
\]
Previously known formulas for the intrinsic volumes of ellipsoids, weighted crosspolytopes, and rectangular boxes arise as special or limiting cases. We also obtain asymptotic formulas for \(V_{j(n)}(\mathbb B_p^n)\) in the high-dimensional regime \(n\to\infty\), where the index \(j(n)\) is allowed to depend on \(n\).

We further investigate the curvature measures of $\mathbb B_p^n$. These are finite measures
$$
\Phi_0(\mathbb B_p^n,\cdot),\dots,\Phi_{n-1}(\mathbb B_p^n,\cdot)
$$
on \(\partial\mathbb B_p^n\) that localize the intrinsic volumes. We prove a Maxwell--Poincar\'{e}--Borel type limit theorem: if \(X_n\) is a random boundary point of \(\mathbb B_p^n\) distributed according to the normalized curvature measure $\Phi_{j(n)}(\mathbb B_p^n,\cdot)/V_{j(n)}(\mathbb B_p^n)$,
where \(j(n)/n\to\alpha\in[0,1]\) as \(n\to\infty\), then for every fixed \(r\in\mathbb N\), the joint distribution of the first \(r\) coordinates of \(n^{1/p}X_n\) converges weakly to the product measure \(\nu_{p,\alpha}^{\otimes r}\). Here \(\nu_{p,\alpha}\) is an explicit probability measure on \(\mathbb R\) depending on \(p>1\) and \(\alpha\in[0,1]\). The classical Maxwell--Poincar\'{e}--Borel theorem is recovered in the special case \(p=2\) and \(\alpha=1\).

The main tool underlying these results is an explicit characterization of the curvature measures of coordinate-weighted \(\ell_p\)-balls, and in particular an explicit formula for their mixed moments.
\end{abstract}
\maketitle

\tableofcontents

\section{Introduction}\label{sec:introduction}

The study of high-dimensional convex bodies plays an important role in modern mathematics, as it reveals fundamental properties of geometric structures with far-reaching applications in areas such as optimization, data analysis, and theoretical computer science. Arguably, the most important class of such bodies is formed by the $\ell_p$-balls, and in this paper we analyze their intrinsic volumes, along with those of their coordinate-weighted analogues, and investigate their curvature measures, establishing an explicit characterization of these measures and proving a continuum of Maxwell--Poincar\'e--Borel limit theorems; for a comprehensive overview of probabilistic methods in the study of $\ell_p$-balls we refer the reader to \cite{P2026, ProchnoThaleTurchi2018}.

Let $K\subset \mathbb R^n$ be a convex body, that is, a compact convex set with a non-empty interior. The intrinsic volumes are a sequence of geometric quantities associated with a convex body that capture its size and shape across different dimensions. In particular, they interpolate between classical notions such as volume, surface area, and mean width. We shall denote the intrinsic volumes of $K$ by $V_0(K), \dots, V_n(K)$. A standard way to introduce these quantities is via the Steiner formula
\[
\Vol_n(K+t\mathbb B_2^n)
 = \sum_{j=0}^{n}\kappa_{n-j}\,V_j(K)\,t^{\,n-j},
\qquad t\ge 0,
\]
where \(\mathbb B_2^m\) is the Euclidean unit ball in \(\mathbb R^m\), and
\[
\kappa_m:=\Vol_m(\mathbb B_2^m)=\frac{\pi^{m/2}}{\Gamma(1+\frac m2)}
\]
denotes its \(m\)-dimensional Lebesgue measure or volume. In particular, \(V_n(K)=\Vol_n(K)\) is the volume of \(K\), \(2V_{n-1}(K) = \mathcal H^{n-1}(\partial K)\) is the \((n-1)\)-dimensional surface area of \(K\), and \(V_0(K)=1\) for every non-empty convex body \(K\). Here $\mathcal H^{n-1}(\cdot)$ denotes the $(n-1)$-dimensional Hausdorff measure on $\R^{n}$.

The intrinsic volumes may also be interpreted as suitably normalized mean projection volumes. More precisely, let \(E_{n,j}\) be a random \(j\)-dimensional linear subspace of \(\mathbb R^n\), distributed according to the rotation-invariant probability measure on the Grassmannian manifold, and let \(\Pi_{n,j}K\) denote the orthogonal projection of \(K\) onto \(E_{n,j}\). Then Kubota's formula states that
\[
V_j(K)
=
\binom{n}{j}\frac{\kappa_n}{\kappa_j\kappa_{n-j}}
\mathbb E\,\Vol_j(\Pi_{n,j}K),
\qquad j=1,\dots,n.
\]
Thus, \(V_j(K)\) may be interpreted as a normalized average \(j\)-dimensional volume of the orthogonal projections of \(K\). For additional background on this topic, we refer the reader to~\cite{SchneiderBook,schneider_weil_book}.

Despite the fundamental role of intrinsic volumes in convex geometry, where they form the basic building blocks of valuation theory and play a central role in integral geometry or random geometric models, explicit formulas for the full sequence \(V_0(K),\dots,V_n(K)\) are available only for a rather small collection of convex bodies. Elementary examples include Euclidean balls and rectangular boxes. Among the classical polyhedral examples are regular simplices~\cite{hadwiger,ruben} and regular crosspolytopes~\cite{betke_henk}; for these and related examples, see also~\cite[Section~15.2.3]{HenkRichterGebertZiegler2017}. More recent additions include ellipsoids~\cite{GusakovaSpodarevZaporozhets2025}, orthocentric simplices~\cite{KabluchkoSchange2025,KabluchkoZaporozhets2019} and certain weighted analogues of crosspolytopes~\cite{HenkHernandezCifre2008a,KabluchkoZaporozhets2019}. In the infinite-dimensional Hilbert-space setting, explicit formulas are likewise available only for a few examples; see~\cite{Gao2003,GV01,kabluchko_zaporozhets_sobolev}.

\vspace*{2mm}
The paper is divided into two parts. In Section~\ref{sec:main_results_intrinsic_volumes}, we derive an explicit formula for the intrinsic volumes of the \(\ell_p\)-balls in $\R^n$ and, more generally, of their coordinate-weighted analogues. We also obtain asymptotic formulas for these intrinsic volumes in the high-dimensional regime \(n\to\infty\). In Section~\ref{sec:main_results_curvature_measures}, we investigate the curvature measures of \(\ell_p\)-balls and, in particular, prove Maxwell--Poincar\'{e}--Borel type limit theorems for these measures.

\section{Main results on intrinsic volumes}\label{sec:main_results_intrinsic_volumes}

\subsection{Exact results on intrinsic volumes}
We begin with the standard unit \(\ell_p\)-ball in \(\mathbb R^n\),
\[
\mathbb B_p^n:= \Bigl\{x\in \mathbb R^n:\ |x_1|^p+\cdots+|x_n|^p\le 1\Bigr\},
\qquad 1\le p<\infty,\quad n\ge 2.
\]
Its \(n\)-dimensional volume is given by the classical formula that goes back to Dirichlet,
\begin{equation}\label{eq:volume_l_p_ball_exact}
V_n(\mathbb B_p^n)=\Vol_n(\mathbb B_p^n)
=\frac{\bigl(2\,\Gamma(1+\tfrac1p)\bigr)^n}{\Gamma(1+\tfrac np)},
\end{equation}
see~\cite[Section~5.4.1]{BorweinBailey2004} and~\cite{Gao2013,KempkaVybiral2017,Wang2005}.
Our first theorem provides an exact formula for all remaining intrinsic volumes
\(V_0(\mathbb B_p^n),\dots,V_{n-1}(\mathbb B_p^n)\).

\begin{theorem}[Exact formula for $V_j(\mathbb B_p^n)$]\label{theo:V_j_ell_p_ball_exact}
Fix $1 < p < \infty$ and $n\geq 2$. Then for all  $j\in\{0,\dots,n-1\}$, the $j$-th intrinsic volume of the unit $\ell_p$-ball in $\R^n$ is given by
\begin{equation}\label{eq:V_j_ell_p_ball_exact}
V_j(\mathbb B_p^n)
=
\frac{p\,(p-1)^{n-j-1}\binom{n}{j}}
{\kappa_{n-j}\,\Gamma(1+\frac jp)\,\Gamma(\frac {n-j}2)}
\int_0^\infty
\theta^{\frac{n-j}2-1}\,
\mathcal I_p(\theta)^j\,
\mathcal J_p(\theta)^{n-j-1}\,
\mathcal  K_p(\theta)\,d\theta,
\end{equation}
where
\begin{align}
\mathcal I_p(\theta)
&:=
\int_{\mathbb R} e^{-|x|^p-\theta |x|^{2p-2}}\,dx, \qquad \theta\geq 0,  \label{eq:def_I_p}
\\
\mathcal J_p(\theta)
&:=\int_{\mathbb R} |x|^{p-2}e^{-|x|^p-\theta |x|^{2p-2}}\,dx, \qquad \theta\geq 0,  \label{eq:def_J_p}
\\
\mathcal  K_p(\theta)
&:=
\int_{\mathbb R} |x|^{2p-2}e^{-|x|^p-\theta |x|^{2p-2}}\,dx = -\mathcal I_p'(\theta), \qquad \theta\geq 0.   \label{eq:def_K_p}
\end{align}
\end{theorem}
\begin{remark}[On special functions]
Let  \(1<p<\infty\).
The functions \(\mathcal I_p\), \(\mathcal J_p\), and \(\mathcal K_p\) introduced above are special cases of the more general family
\[
\mathcal F_p(t;\nu)
:=
\int_{\mathbb R} |y|^\nu e^{-|y|^p-t|y|^{2p-2}}\,dy,
\qquad t>0,\quad \nu\in\mathbb C,\quad \Re \nu>-1.
\]
More precisely,
\[
\mathcal I_p(\theta)=\mathcal F_p(\theta;0),\qquad
\mathcal J_p(\theta)=\mathcal F_p(\theta;p-2),\qquad
\mathcal K_p(\theta)=\mathcal F_p(\theta;2p-2).
\]
For completeness, \(\mathcal F_p\) can be expressed  in terms of the Fox--Wright function
$
{}_1\Psi_0 [(a,A);z]
:=
\sum_{k=0}^\infty \Gamma(a+Ak)\,z^k/k!.
$
The exact formula depends on whether \(1<p<2\), \(p=2\), or \(p>2\). As this representation will not be needed in the sequel, we do not record it here.
\end{remark}


The following examples discuss the cases $p=2,1$ and $p\to \infty$, corresponding to the Euclidean ball, the crosspolytope,  and the cube, respectively. These are precisely the instances in which explicit formulas for the full sequence of intrinsic volumes of $\mathbb B_p^n$ were previously known.

\begin{example}[$p=2$: Euclidean balls]
For the Euclidean unit ball \(\mathbb B_2^n\), the intrinsic volumes are given by
\begin{equation}\label{eq:V_j_balls}
V_j(\mathbb B_2^n)
=
\binom{n}{j}\frac{\kappa_n}{\kappa_{n-j}},
\qquad j=0,1,\dots,n.
\end{equation}
This classical formula follows immediately from~\eqref{eq:V_j_ell_p_ball_exact}, since
\begin{equation}\label{eq:I_J_K_for_p_equal_2}
\mathcal I_2(\theta)
=
\mathcal J_2(\theta)
=
\int_{\mathbb R} e^{-(1+\theta)t^2}\,dt
=
\sqrt{\pi}\,(1+\theta)^{-1/2},
\qquad
\mathcal K_2(\theta)
=
-\mathcal I_2'(\theta)
=
\frac{\sqrt{\pi}}{2}(1+\theta)^{-3/2}.
\end{equation}
Substituting these expressions into~\eqref{eq:V_j_ell_p_ball_exact} yields~\eqref{eq:V_j_balls}.
\end{example}

\begin{example}[$p=1$: crosspolytopes]
The case \(p=1\) corresponds to the regular crosspolytope
\[
\mathbb B_1^n
=
\Bigl\{x\in\mathbb R^n:\ |x_1|+\cdots+|x_n|\le 1\Bigr\}
=
\operatorname{conv}\{\pm e_1,\dots,\pm e_n\},
\]
where $\{e_1,\dots,e_n\}$ is the standard basis of $\R^n$. Its intrinsic volumes are known explicitly; see~\cite{betke_henk} and also~\cite[Section~15.2.3]{HenkRichterGebertZiegler2017}. For \(j=0,\dots,n-1\), one has
\begin{equation}\label{eq:V_j_crosspolytope}
V_j(\mathbb B_1^n)
=
2^{j+1}\binom{n}{j+1}\frac{j+1}{j!}
\int_0^\infty
\varphi(\sqrt{j+1}\,t)\,
(2\Phi(t)-1)^{\,n-j-1}\,dt,
\end{equation}
where
\begin{equation}\label{eq:varphi_Phi_def}
\varphi(t)=\frac{1}{\sqrt{2\pi}}e^{-t^2/2}
\qquad\text{and}\qquad
\Phi(t)=\int_{-\infty}^t \varphi(u)\,du
\end{equation}
denote the standard normal density and distribution function, respectively.
In the remaining case \(j=n\), one has $V_n(\mathbb B_1^n)=\Vol_n(\mathbb B_1^n)=2^n/n!$.

Although the case \(p=1\) is excluded from Theorem~\ref{theo:V_j_ell_p_ball_exact}, because the integral defining \(\mathcal J_p(\lambda)\) diverges at \(p=1\), formula~\eqref{eq:V_j_crosspolytope} can nevertheless be recovered as a limiting case of~\eqref{eq:V_j_ell_p_ball_exact}. Indeed, as \(p\downarrow 1\), the \(\ell_p\)-balls \(\mathbb B_p^n\) converge to \(\mathbb B_1^n\) in the Hausdorff metric, and hence $V_j(\mathbb B_p^n)\to V_j(\mathbb B_1^n)$
for every \(j=0,\dots,n\), since intrinsic volumes are continuous with respect to Hausdorff convergence. On the other hand, one verifies that the right-hand side of~\eqref{eq:V_j_ell_p_ball_exact} converges to the right-hand side of~\eqref{eq:V_j_crosspolytope} as \(p\downarrow 1\). The key step is to observe that, for every fixed $\theta>0$,
\begin{multline*}
\lim_{p\downarrow 1} \mathcal I_p(\theta) =
\lim_{p\downarrow 1} \mathcal K_p(\theta) = 2e^{-\theta}\quad\text{and, using the substitution }u=x^{p-1},\\
\lim_{p\downarrow 1} (p-1)\mathcal J_p(\theta)
=2\lim_{p\downarrow 1} \int_0^{\infty}e^{-u^\frac{p}{p-1}-\theta u^2}\,du
= 2\int_0^1 e^{-\theta u^2}\,du
= \sqrt{\frac{\pi}{\theta}}\bigl(2\Phi(\sqrt{2\theta})-1\bigr).
\end{multline*}

\end{example}

\begin{example}[$p=\infty$: cubes]\label{ex:cubes_explicit}
The limiting case \(p=\infty\) corresponds to the cube $\mathbb B_\infty^n=[-1,1]^n$.
Its intrinsic volumes are given by
\begin{equation}\label{eq:V_j_cubes}
V_j(\mathbb B_\infty^n)=\binom{n}{j}2^j,
\qquad j=0,1,\dots,n.
\end{equation}
Moreover, it can be shown that the right-hand side of~\eqref{eq:V_j_ell_p_ball_exact} converges to \(\binom{n}{j}2^j\) as \(p\to\infty\), so that~\eqref{eq:V_j_cubes} is recovered as the limiting case \(p\to\infty\). The key step is to observe that, for every fixed $\theta>0$, one has
\begin{align*}
&\lim_{p\to\infty} \mathcal I_p(\theta) =  2,\\
&\lim_{p\to\infty} \theta^{1/2}p\,\mathcal J_p(\theta) = 2\theta^{1/2}\int_0^\infty e^{-u-\theta u^2}\,du=2\int_0^\infty e^{-v\theta^{-1/2}-v^2}\,du=:\widetilde{\mathcal{J}}_{\infty}(\theta),\\
&\lim_{p\to\infty} \theta^{-1/2} p\,\mathcal K_p(\theta)=2\theta^{-1/2}\int_0^\infty u\,e^{-u-\theta u^2}\,du=2\theta^{-3/2}\int_0^\infty v e^{-v\theta^{-1/2}-v^2}\,du=2(\widetilde{\mathcal{J}}_{\infty}(\theta))^{\prime}.
\end{align*}
Formula~\eqref{eq:V_j_cubes}, for $j<n$, now follows from~\eqref{eq:V_j_ell_p_ball_exact} by noting that
$$
\int_0^{\infty}(\widetilde{\mathcal{J}}_{\infty}(\theta))^{n-j-1}2(\widetilde{\mathcal{J}}_{\infty}(\theta))^{\prime}\,d\theta=2\int_{\widetilde{\mathcal{J}}_{\infty}(0)}^{\widetilde{\mathcal{J}}_{\infty}(\infty)}z^{n-j-1}\,dz=2\int_{0}^{\sqrt{\pi}}z^{n-j-1}\,dz=\frac{2\pi^{\frac{n-j}{2}}}{n-j}.
$$
\end{example}

\begin{example}[Surface area]
Fix \(1<p<\infty\) and \(n\ge 2\). Since $\mathcal H^{n-1}(\partial \mathbb B_p^n)=2V_{n-1}(\mathbb B_p^n)$,
the surface area of \(\mathbb B_p^n\) can be obtained from Theorem~\ref{theo:V_j_ell_p_ball_exact} by taking \(j=n-1\). After an integration by parts, this yields
\[
\mathcal H^{n-1}(\partial \mathbb B_p^n)
=
\frac{p}{2\sqrt\pi\,\Gamma\!\left(1+\frac{n-1}{p}\right)}
\int_0^\infty
\frac{\mathcal I_p(0)^n-\mathcal I_p(\theta)^n}{\theta^{3/2}}\,d\theta,
\quad
\mathcal I_p(0)=\int_{\mathbb R}e^{-|x|^p}\,dx
=
2\Gamma\!\left(1+\frac1p\right).
\]
We omit the derivation here, since a more general formula for the integrals of
\(|x_1|^{\lambda_1}\cdots |x_n|^{\lambda_n}\) with respect to the surface measure
will be obtained later in Proposition~\ref{prop:mixed_moments_surface_area_measure}.
\end{example}

Theorem~\ref{theo:V_j_ell_p_ball_exact} admits the following extension to coordinate-weighted \(\ell_p\)-balls.

\begin{theorem}[Intrinsic volumes of weighted \(\ell_p\)-balls]\label{theo:V_j_ell_p_ellipsoid_exact}
Let \(1<p<\infty\), \(n\ge 2\), and let $a=(a_1,\dots,a_n)\in(0,\infty)^n$
be a vector of positive weights. Define the weighted\footnote{The notation $\mathbb{B}_p^n(a)$ should not be confused with its common use to denote the standard unit $\ell_p$-ball $\mathbb{B}_p^n$ centered at $a$.} \(\ell_p\)-ball by
\[
\mathbb B_p^n(a)\equiv \mathbb B_p^n(a_1,\dots,a_n)
:=
\left\{
x\in\mathbb R^n:\ \sum_{r=1}^n |a_r x_r|^p\le 1
\right\}.
\]
Then, for every \(j\in\{0,\dots,n-1\}\), one has
\begin{multline*}
V_j(\mathbb B_p^n(a))
=
\frac{p(p-1)^{n-1-j}}{2\pi^{\frac{n-j}{2}}\Gamma\!\left(1+\frac jp\right)}
\sum_{i=1}^n
\sum_{\substack{I\subseteq\{1,\dots,n\}\setminus\{i\}\\ |I|=n-1-j}}
\\
\times \int_0^\infty \theta^{\frac{n-j}{2}-1}
a_i \mathcal K_p\!\left(a_i^2 \theta\right)
\prod_{r\in I} a_r \mathcal J_p\!\left(a_r^2 \theta\right)
\prod_{r\notin I\cup\{i\}} a_r^{-1} \mathcal I_p\!\left(a_r^2\theta\right)
\,d\theta,
\end{multline*}
where the functions $\mathcal I_p, \mathcal J_p, \mathcal K_p$ are defined by~\eqref{eq:def_I_p}.
\end{theorem}

As the following examples show, Theorem~\ref{theo:V_j_ell_p_ellipsoid_exact} unifies and generalizes several previously known explicit formulas, including those for ellipsoids, coordinate-weighted crosspolytopes, and rectangular boxes.

\begin{example}[$p=2$: ellipsoids]\label{ex:ellipsoids_intrinsic_volumes}
In the case \(p=2\), the weighted \(\ell_p\)-ball reduces to the ellipsoid
\[
\mathbb B_2^n(a)
=
\left\{
x\in\mathbb R^n:\sum_{r=1}^n a_r^2 x_r^2\le 1
\right\}.
\]
Explicit formulas for the intrinsic volumes of ellipsoids were obtained by~\citet{GusakovaSpodarevZaporozhets2025}. We now examine the specialization of Theorem~\ref{theo:V_j_ell_p_ellipsoid_exact} to \(p=2\) and compare the resulting formula with that of~\cite{GusakovaSpodarevZaporozhets2025}. Using the identities for \(\mathcal I_2\), \(\mathcal J_2\), and \(\mathcal K_2\) from~\eqref{eq:I_J_K_for_p_equal_2}, Theorem~\ref{theo:V_j_ell_p_ellipsoid_exact} yields
\[
V_j(\mathbb B_2^n(a))
=
\frac{\kappa_j}{2}
\left(\prod_{r=1}^n a_r^{-1}\right)
\sum_{i=1}^n
a_i^2\,\sigma_{n-1-j}(a_1^2,\dots,\widehat{a_i^2},\dots,a_n^2)
\int_0^\infty
\frac{\theta^{\frac{n-j}{2}-1}\,d\theta}
{(1+a_i^2\theta)\prod_{r=1}^n \sqrt{1+a_r^2\theta}},
\]
where \(\widehat{a_i^2}\) indicates that the variable \(a_i^2\) is omitted, and
\(\sigma_m(x_1,\dots,x_N)\) denotes the \(m\)-th elementary symmetric polynomial,
\begin{equation}\label{eq:sigma_eleme_symm_poly}
\sigma_m(x_1,\dots,x_N)
:=
\sum_{\substack{J\subseteq\{1,\dots,N\}\\ |J|=m}}
\prod_{r\in J} x_r =  \sum_{1\leq r_1 < \ldots < r_{m} \leq N} x_{r_1} \ldots x_{r_{m}}.
\end{equation}
Introducing the semiaxes \(b_i:=a_i^{-1}\) and making the substitution \(\theta=t^{-2}\), we obtain
\begin{equation}\label{eq:V_j_ellipsoids_1}
V_j(\mathbb B_2^n(b_1^{-1},\ldots, b_n^{-1}))
=
\kappa_j
\sum_{i=1}^n
b_i^2\,
\sigma_j\!\bigl(b_1^2,\dots,\widehat{b_i^2},\dots,b_n^2\bigr)
\int_0^\infty
\frac{t^{j+1}\,dt}
{(1+b_i^2t^2)\prod_{r=1}^n (1+b_r^2t^2)^{1/2}},
\end{equation}
for \(j=0,\dots,n-1\).
The formula proved in~\cite[Theorem~1]{GusakovaSpodarevZaporozhets2025} is written in a slightly different form, namely
\begin{equation}\label{eq:V_j_ellipsoids_2}
V_j(\mathbb B_2^n(b_1^{-1},\ldots, b_n^{-1}))
=
\kappa_j
\sum_{i=1}^n
b_i^2\,\sigma_{j-1}\!\bigl(b_1^2,\dots,\widehat{b_i^2},\dots,b_n^2\bigr)
\int_0^\infty
\frac{t^{j-1}\,dt}
{(1+b_i^2t^2)\prod_{r=1}^n (1+b_r^2t^2)^{1/2}},
\end{equation}
for \(j=1,\dots,n\). For \(j=1,\dots,n-1\), the right-hand sides of~\eqref{eq:V_j_ellipsoids_1} and~\eqref{eq:V_j_ellipsoids_2} agree, although verifying this directly is non-trivial.
\end{example}

\begin{example}[$p=1$: weighted crosspolytopes]
The case \(p=1\) corresponds to the weighted crosspolytope
\[
\mathbb B_1^n(a)
=
\Bigl\{x\in\mathbb R^n:\ a_1|x_1|+\cdots+a_n|x_n|\le 1\Bigr\}
=
\operatorname{conv}\{\pm e_1/a_1,\dots,\pm e_n/a_n\}.
\]
Explicit formulas for its intrinsic volumes were obtained in~\cite[Corollary~2.1]{HenkHernandezCifre2008a}; see also~\cite[Theorem~3.6]{KabluchkoZaporozhets2019} for a more general result. In particular, for \(j=0,\dots,n-1\), one has
\[
V_j(\mathbb B_1^n(a))
=
\frac{2^{j+1}}{j!}
\sum_{\substack{L\subseteq \{1,\dots,n\}\\ |L|=j+1}}
\frac{\sum_{i\in L} a_i^2}{\prod_{i\in L} a_i}
\int_0^\infty
\varphi \Biggl(
x\sqrt{\sum_{i\in L} a_i^2}
\Biggr)
\prod_{i\notin L}
\bigl(2\Phi(a_i x)-1\bigr)\,dx,
\]
see~\cite[Example~3.8]{KabluchkoZaporozhets2019}, where \(\varphi\) and \(\Phi\) are as in~\eqref{eq:varphi_Phi_def}.
As in the unweighted case, this formula may be recovered by passing to the limit \(p\downarrow 1\) in Theorem~\ref{theo:V_j_ell_p_ellipsoid_exact}.
\end{example}

\begin{example}[$p=\infty$: rectangular boxes]
As \(p\to\infty\), the weighted \(\ell_p\)-balls \(\mathbb B_p^n(a)\) converge in the Hausdorff metric to the rectangular box
$
\mathbb B_\infty^n(a)
:=
\prod_{i=1}^n [-a_i^{-1},a_i^{-1}].
$
Its intrinsic volumes are given by
\[
V_j(\mathbb B_\infty^n(a))
=
2^j\,\sigma_j(a_1^{-1},\dots,a_n^{-1}),
\qquad j=0,\dots,n,
\]
where \(\sigma_j\) denotes the \(j\)-th elementary symmetric polynomial defined in~\eqref{eq:sigma_eleme_symm_poly}. As in the unweighted case,  this formula may also be obtained as the limit \(p\to\infty\) in Theorem~\ref{theo:V_j_ell_p_ellipsoid_exact}.
\end{example}

\begin{example}[Surface area]
Taking \(j=n-1\) in Theorem~\ref{theo:V_j_ell_p_ellipsoid_exact},
it is possible to obtain the following formula for the surface area of the weighted \(\ell_p\)-ball:
\[
\mathcal H^{n-1}\!\bigl(\partial \mathbb B_p^n(a)\bigr)
=
\frac{
p\,\prod_{k=1}^n a_k^{-1}
}{
2\,\Gamma\left(1 + \frac{n-1}{p}\right)\sqrt\pi
}
\int_0^\infty
\frac{
\bigl(2\Gamma(1+\tfrac1p)\bigr)^n
-
\prod_{k=1}^n \mathcal I_p(\theta a_k^2)
}{
\theta^{3/2}
}\,d\theta.
\]
A more general formula will be given later in Proposition~\ref{prop:mixed_moments_surface_area_measure}; the above identity is obtained from it by specializing to \(\lambda_1=\ldots=\lambda_n=0\).

In the special case \(p=2\), that is, for ellipsoids, alternative expressions for the surface area are known. In arbitrary dimension, a representation in terms of the Lauricella hypergeometric function \(F_D\) was given by Rivin; see~\cite[Section~3, Equations~(11)-(13)]{Rivin2007}.   
See also~\cite{Tee2005,KabluchkoZaporozhets2014} for further results on the surface area of ellipsoids.
\end{example}

\subsection{Asymptotic formulas for intrinsic volumes}
The exact formula for the intrinsic volumes of \(\mathbb B_p^n\) given in Theorem~\ref{theo:V_j_ell_p_ball_exact}, combined with Laplace's method, makes it possible to derive asymptotic formulas for \(V_{j(n)}(\mathbb B_p^n)\) as \(n\to\infty\), in several regimes determined by the asymptotic behavior of the index \(j(n)\in\{0,\dots,n\}\). For simplicity, we restrict ourselves here to the unweighted case, that is, \(a_1=\cdots=a_n=1\), although some of the results admit extensions to the weighted setting.

Before stating the asymptotics, we record the basic properties of the phase function that governs the Laplace method in the bulk regime. In what follows, the notation \(x_n\sim y_n\) as \(n\to\infty\) means that \(x_n/y_n\to 1\).
\begin{proposition}[Phase function and its maximizer]\label{prop:maximizer_Phi_p_beta}
Fix \(1<p<\infty\) and \(\beta\in(0,1)\). Define the function $\Psi_{p,\beta}:(0,\infty)\to\mathbb R$
by
\begin{equation}\label{eq:Phi_alpha_def_prop_maximizer_theo_asympt_V_j}
\Psi_{p,\beta}(\theta)
=
\frac{1-\beta}{2}\log \theta
+\beta\log \mathcal I_p(\theta)
+(1-\beta)\log \mathcal J_p(\theta),
\qquad \theta>0.
\end{equation}
Then \(\Psi_{p,\beta}\) has a unique global maximizer \(\theta_{p,\beta}\in(0,\infty)\). Moreover, $\Psi_{p,\beta}''(\theta_{p,\beta})<0$,
and \(\theta_{p,\beta}\) is the unique solution in \((0,\infty)\) of the equation
\begin{equation}\label{eq:theta_alpha_equation_theo_V_j_asympt}
\theta_{p,\beta}\,\frac{\mathcal J_p(\theta_{p,\beta})}{\mathcal I_p(\theta_{p,\beta})}
=
\frac{(1-\beta)p}{2(p-1)\beta}.
\end{equation}
For $p=2$, this simplifies to $\theta_{2,\beta}=(1-\beta)/\beta$ and
\begin{equation}\label{eq:Phi_alpha_def_prop_maximizer_theo_asympt_V_j_p=2}
\Psi_{2,\beta}(\theta)=\frac{1-\beta}{2}\log\theta+\frac{1}{2}\log\pi-\frac{1}{2}\log(1+\theta),\quad \theta>0.
\end{equation}
\end{proposition}

We are now ready to state the asymptotic formulas for the intrinsic volumes of \(\mathbb B_p^n\).

\begin{theorem}[Asymptotic formulas for intrinsic volumes of \(\mathbb B_p^n\)]\label{theo:exact_asympt_V_j}
Fix \(1<p<\infty\). The following assertions hold as \(n\to\infty\).

\begin{itemize}
\item[(i)] \emph{Bulk regime.}
Let \(j(n)\in\{0,\dots,n\}\) be such that $\alpha(n) := j(n)/n \to \alpha \in (0,1)$, as $n\to\infty$.
Let \(\theta_{p,\alpha}\in(0,\infty)\) be the unique maximizer of the function \(\Psi_{p,\alpha}\) from Proposition~\ref{prop:maximizer_Phi_p_beta}. Then
\begin{equation}\label{eq:intrinsic_ell_p_ball_central_exact}
V_{j(n)}(\mathbb B_p^n)
\sim
\frac{p(p-1)^{n-j(n)-1}\,n\binom{n-1}{j(n)}}{2\pi^{(n-j(n))/2}\,
\Gamma\!\left(\frac{j(n)+p}{p}\right)}
\,
\frac{\mathcal K_p(\theta_{p,\alpha})}{\theta_{p,\alpha}\,\mathcal J_p(\theta_{p,\alpha})}
\,
e^{\,n\Psi_{p,\alpha(n)}(\theta_{p,\alpha(n)})}
\sqrt{\frac{2\pi}{n\,|\Psi_{p,\alpha}''(\theta_{p,\alpha})|}}.
\end{equation}

\item[(ii)] \emph{Left-edge regime.}
Let \(j\in\N\cup\{0\}\) be fixed. Then
\begin{equation}\label{eq:V_j_exact_asymptotics_left_edge}
V_j(\mathbb B_p^n)
\sim
\frac{1}{j!}
\left(
(2\sqrt{\pi})^{1/p}
\left(
\frac{\Gamma\!\left(\frac{1}{2p-2}\right)}{p-1}
\right)^{1-1/p}
\right)^j
n^{\,j(1-1/p)}.
\end{equation}

\item[(iii)] \emph{Right-edge regime.}
Let \(m\in\N\) be fixed. Then
\begin{equation}\label{eq:V_j_exact_asymptotics_right_edge}
V_{n-m}(\mathbb B_p^n)
\sim
\frac{\Gamma\!\left(\frac m2\right)}{2\,\Gamma(m)}
\left(
\frac{p(p-1)\Gamma\!\left(1-\frac1p\right)}
{\pi\,\Gamma\!\left(\frac1p\right)}
\right)^{m/2}
\frac{\left(\frac{2}{p}\Gamma\!\left(\frac1p\right)\right)^n\,n^{m/2}}
{\Gamma\!\left(\frac{n+p-m}{p}\right)}.
\end{equation}
\end{itemize}
\end{theorem}
\begin{remark}[The volume of $\mathbb B_p^n$]
Applying Stirling's formula to the exact volume formula~\eqref{eq:volume_l_p_ball_exact} yields an asymptotic expression that agrees with the one suggested by Part~(iii) of Theorem~\ref{theo:exact_asympt_V_j} when one formally sets \(m=0\) in~\eqref{eq:V_j_exact_asymptotics_right_edge} and interprets the factor \(\Gamma(0/2)/(2\Gamma(0))\) by its limiting value \(1\).
\end{remark}
\begin{example}[Surface area of \(\mathbb B_p^n\)]
Fix \(1<p<\infty\). Taking \(m=1\) in Part~(iii) of Theorem~\ref{theo:exact_asympt_V_j} yields an asymptotic formula for the surface area of \(\mathbb B_p^n\):
\begin{equation}\label{eq:surface_are_asymptotics}
\mathcal H^{n-1}(\partial \mathbb B_p^n)
=
2V_{n-1}(\mathbb B_p^n)
\sim
\left(
\frac{p(p-1)\Gamma\!\left(1-\frac{1}{p}\right)}
{\Gamma\!\left(\frac{1}{p}\right)}
\right)^{1/2}
\frac{\left(\frac{2}{p}\Gamma\!\left(\frac{1}{p}\right)\right)^n n^{1/2}}
{\Gamma\!\left(\frac{n+p-1}{p}\right)},
\qquad n\to\infty.
\end{equation}
Let $\mathbb D_p^n:=\bigl(\Vol_n(\mathbb B_p^n)\bigr)^{-1/n}\mathbb B_p^n$
denote the \(\ell_p\)-ball in \(\mathbb R^n\) normalized to have volume \(1\). Combining~\eqref{eq:surface_are_asymptotics} with the exact volume formula~\eqref{eq:volume_l_p_ball_exact}, and simplifying with the help of Stirling's approximation and the reflection formula for the gamma-function, we obtain
\begin{equation}\label{eq:surface_aread_normalized_l_p_ball}
\mathcal H^{n-1}(\partial \mathbb D_p^n)
\sim
2e^{1/p}
\sqrt{\frac{\pi(p-1)}{p\,\sin(\pi/p)}}\,n^{1/2},
\qquad n\to\infty.
\end{equation}
It is natural to ask whether a similar asymptotic behavior of the form $\mathrm{const} \cdot n^{1/2}$ holds for the surface area of other sequences of $n$-dimensional convex bodies of volume $1$. For example, for the crosspolytope of unit volume, this is indeed the case. Using~\eqref{eq:V_j_crosspolytope} with $j=n-1$, we conclude that
$$
\mathcal{H}^{n-1}(\partial \mathbb{B}_1^n)=2V_{n-1}(\mathbb{B}_1^n)=\frac{2^n\sqrt{n}}{(n-1)!},
$$
whence
$$
\mathcal{H}^{n-1}(\partial \mathbb{D}_1^n)=\frac{2^n\sqrt{n}}{(n-1)!}\left(\frac{n!}{2^n}\right)^{1-1/n}=\frac{2n^{3/2}}{(n!)^{1/n}}~\sim~2e\sqrt{n}.
$$
Observe that the coefficient $2e$ coincides with the limiting value $\lim_{p \downarrow 1} 2e^{1/p}\sqrt{\frac{\pi(p-1)}{p\sin(\pi/p)}}$ on the right-hand side of~\eqref{eq:surface_aread_normalized_l_p_ball}. However, in general, the asymptotic behavior of the form $\mathrm{const} \cdot n^{1/2}$ does not hold universally. For instance, for the cube $\mathbb D_{\infty}^n=[-1/2,1/2]^n$,
$$
\mathcal{H}^{n-1}(\partial \mathbb{D}_\infty^n) \sim 2n.
$$
The isoperimetric inequality states that for every $n$-dimensional convex body  $K_n$ with $\operatorname{Vol}_n(K_n)=1$ one has  $\mathcal H^{n-1}(\partial K_n) \geq n \kappa_n^{1/n}$. Moreover, as  $n\to\infty$, one has  $n\kappa_n^{1/n} \sim \sqrt{2\pi e\, n}$.
\end{example}

\begin{remark}[Random projections of $\ell_p$-balls]
Several papers, including~\citet{AdamczakPivovarovSimanjuntak2024} and~\citet{ProchnoThaleTuchel2024},  investigate limit theorems for random projections and sections of \(\ell_p\)-balls. Kubota's formula implies that the expected \(j\)-dimensional volume of the orthogonal projection of \(\mathbb B_p^n\) onto a uniformly random \(j\)-dimensional linear subspace of $\mathbb R^n$ equals
\[
\frac{\kappa_j\kappa_{n-j}}{\kappa_n\binom{n}{j}}\,V_j(\mathbb B_p^n).
\]
Hence, Theorems~\ref{theo:V_j_ell_p_ball_exact} and~\ref{theo:exact_asympt_V_j} provide exact and asymptotic expressions for these expectations.
\end{remark}

The bulk asymptotics from Theorem~\ref{theo:exact_asympt_V_j} combined with convexity arguments  yield the following ``exponential profile'' for the intrinsic volumes of the rescaled balls
\[
n^{1/p}\mathbb B_p^n
=
\Bigl\{x\in\mathbb R^n:\ |x_1|^p+\cdots+|x_n|^p\le n\Bigr\}.
\]

\begin{theorem}[Exponential profile of intrinsic volumes]\label{theo:exp_profile_V_j_l_p_balls}
Let $1 < p < \infty$.  
Then, as $n\to\infty$,
$$
\sup_{j\in \{0,\dots, n\}} \left|\frac 1n \log V_{j}(n^{1/p}\mathbb B_p^n) - g_p\left(\frac jn\right)\right|  \to 0
$$
with the exponential profile
$$
g_p(\alpha) :=
\kappa_{p}(\alpha)
+
\sup_{\theta>0}\Psi_{p,\alpha}(\theta),
\qquad
\alpha \in [0,1],
$$
where $\Psi_{p,\alpha}(\theta)$ is the phase function from Proposition~\ref{prop:maximizer_Phi_p_beta},  and
\[
\kappa_{p}(\alpha):=
\frac{\alpha}{p}\Bigl(1-\log\frac{\alpha}{p}\Bigr)
+(1-\alpha)\log(p-1)
-\alpha\log\alpha-(1-\alpha)\log(1-\alpha)
-\frac{1-\alpha}{2}\log\pi.
\]
Here and in the following, we adopt  the convention \(0\log 0=0\).
\end{theorem}

\begin{example}[\(\alpha=1\)]\label{example:exp_profile_ell_p_at_alpha_1}
In the case \(\alpha=1\), the quantity \(g_p(1)\) describes the exponential growth rate of the volume of the rescaled \(\ell_p\)-balls. Combining the exact volume formula~\eqref{eq:volume_l_p_ball_exact} with Stirling's formula, we obtain
\[
\Vol_n\bigl(n^{1/p}\mathbb B_p^n\bigr)
=
n^{n/p}\,\frac{\bigl(2\,\Gamma(1+\tfrac1p)\bigr)^n}{\Gamma(1+\tfrac np)}
\sim
\sqrt{\frac{p}{2\pi n}}\,
\Bigl(2\,(ep)^{1/p}\Gamma(1+\tfrac1p)\Bigr)^n,
\qquad n\to\infty.
\]
Consequently,
\[
g_p(1)
=
\lim_{n\to\infty}\frac{1}{n}\log \Vol_n\bigl(n^{1/p}\mathbb B_p^n\bigr)
=
\log\!\left(2\,(ep)^{1/p}\Gamma\!\left(1+\tfrac1p\right)\right).
\]
\end{example}

\subsection{Conjectures: Limit theorems for intrinsic volumes}

The exponential profile established above for \(\ell_p\)-balls suggests that a similar phenomenon should hold for much broader classes of high-dimensional convex bodies.

\begin{conjecture}[Exponential profile of intrinsic volumes]\label{conj:exp_profile_intr_vols}
Let \((K_n)_{n\ge1}\) be a sufficiently regular sequence of convex bodies such that $K_n\subseteq \mathbb R^n$ and
$\dim K_n=n$ for every \(n\ge1\), and assume that $\frac{1}{n}\log \Vol_n(K_n)\to c\in\mathbb R$, as $n\to\infty$.
Then there exists a concave function $g:[0,1]\to\mathbb R$, the exponential profile of intrinsic volumes, such that
\[
g(\alpha)
=
\lim_{n\to\infty}\frac{1}{n}\log V_{\lfloor \alpha n\rfloor}(K_n),
\qquad \alpha\in[0,1].
\]
\end{conjecture}

\begin{remark}[On concavity]\label{rem:logconcavity_intrinsic_vols}
The intrinsic volumes of any convex body $K$ satisfy $V_r(K)^2\geq \frac{r+1}{r}\, V_{r-1}(K) V_{r+1}(K)$ for all $r=1,\dots, \dim K - 1$; see~\cite{McMullen1991}. It follows that the exponential profile $g$, whenever it exists, is concave.
\end{remark}

Theorem~\ref{theo:exp_profile_V_j_l_p_balls} confirms Conjecture~\ref{conj:exp_profile_intr_vols} for $K_n=n^{1/p}\mathbb B_p^n$, $1<p<\infty$.
In the next two examples, we discuss the cases \(p=1,2,\infty\) and regular simplices, for which the conjecture is likewise valid.

\begin{example}[\(\ell_p\)-balls for \(p=1,2,\infty\)]
For cubes \(\mathbb B_\infty^n=[-1,1]^n\) and rescaled Euclidean balls \(\sqrt n\,\mathbb B_2^n\), the explicit formulas~\eqref{eq:V_j_cubes} and~\eqref{eq:V_j_balls}, combined with Stirling's formula, yield
\begin{align*}
g_{\infty}(\alpha)
&:=
\lim_{n\to\infty}\frac1n\log V_{j(n)}(\mathbb B_\infty^n) =
-\alpha\log\alpha-(1-\alpha)\log(1-\alpha)+\alpha\log 2,
\\
g_{2}(\alpha)
&:=
\lim_{n\to\infty}\frac1n\log V_{j(n)}(\sqrt n\,\mathbb B_2^n) =
-\alpha\log\alpha-\frac{1-\alpha}{2}\log(1-\alpha)
+\alpha\log \bigl(\sqrt{2\pi e}\bigr),
\end{align*}
for every \(\alpha\in[0,1]\) and every sequence \(j(n)\in\{0,\dots,n\}\) such that \(j(n)/n\to\alpha\) as \(n\to\infty\). These formulas are consistent with Theorem~\ref{theo:exp_profile_V_j_l_p_balls}. For $p=2$ this
follows from~\eqref{eq:Phi_alpha_def_prop_maximizer_theo_asympt_V_j_p=2}. For $p\to \infty$, we have using the limit relations from Example~\ref{ex:cubes_explicit}
\begin{align*}
\lim_{p\to\infty} g_p(\alpha)&=-\alpha\log\alpha-(1-\alpha)\log (1-\alpha)-\frac{1-\alpha}{2}\log \pi\\
&\hspace{3cm}+\lim_{p\to\infty} \sup_{\theta>0}\left(\alpha\log \mathcal{I}_p(\theta) +(1-\alpha)\log(\theta^{1/2}(p-1)\mathcal{\mathcal J}_p(\theta))\right)\\
&=-\alpha\log\alpha-(1-\alpha)\log (1-\alpha)-\frac{1-\alpha}{2}\log \pi+\alpha\log 2+\sup_{\theta>0}\left((1-\alpha)\log\widetilde{\mathcal J}_{\infty}(\theta)\right)\\
&=-\alpha\log\alpha-(1-\alpha)\log (1-\alpha)-\frac{1-\alpha}{2}\log \pi+\alpha\log 2+\left((1-\alpha)\log\widetilde{\mathcal J}_{\infty}(\infty)\right)\\
&=-\alpha\log\alpha-(1-\alpha)\log(1-\alpha)+\alpha\log 2 = g_\infty(\alpha),\qquad \alpha\in[0,1],
\end{align*}
where we have utilized the equality $\widetilde{\mathcal J}_{\infty}(\infty)=\sqrt{\pi}$.

For the crosspolytope, corresponding to \(p=1\), the explicit formula~\eqref{eq:V_j_crosspolytope}, combined with a standard Laplace asymptotic analysis, gives
\begin{align*}
g_{1}(\alpha)
&:=
\lim_{n\to\infty}\frac1n\log V_{j(n)}(n\mathbb B_1^n) \\
&=
\alpha\log(2e)-2\alpha\log\alpha-(1-\alpha)\log(1-\alpha)
+\sup_{t\ge 0}
\left(
-\frac{\alpha}{2}t^2+(1-\alpha)\log(2\Phi(t)-1)
\right).
\end{align*}
\end{example}

\begin{example}[Regular simplex]
Let $\Delta_n:=\operatorname{conv}\{e_1,\dots,e_{n+1}\}\subseteq\mathbb R^{n+1}$
be the regular \(n\)-dimensional simplex. Its intrinsic volumes are given by
\[
V_j(\Delta_n)
=
\binom{n+1}{j+1}\frac{j+1}{j!}
\int_{-\infty}^{\infty}
\varphi(\sqrt{j+1}\,x)\,\Phi(x)^{\,n-j}\,dx,
\qquad j=0,\dots,n;
\]
see~\cite{hadwiger,ruben} and also~\cite[Section~15.2.3]{HenkRichterGebertZiegler2017}.
If \(j(n)/n\to\alpha\in[0,1]\) as \(n\to\infty\), then a combination of Laplace's method with Stirling's formula yields the exponential profile
\begin{align*}
g_{\mathrm{simp}}(\alpha)
&:=
\lim_{n\to\infty}\frac1n\log V_{j(n)}(n\Delta_n)
\\
&=
\alpha
-2\alpha\log\alpha-(1-\alpha)\log(1-\alpha)
+\sup_{x\in\mathbb R}
\left(
-\frac{\alpha}{2}x^2+(1-\alpha)\log\Phi(x)
\right).
\end{align*}
\end{example}

Following~\citet{McCoyTropp2014} and~\citet{LotzMcCoyNourdinPeccatiTropp2020}, one may associate with an \(n\)-dimensional convex body \(K_n\) the \emph{intrinsic volume random variable} \(Z_n\), taking values in \(\{0,\dots,n\}\) and defined by
$$
\mathbb P[Z_n=j]
=
\frac{V_j(K_n)}{\sum_{r=0}^n V_r(K_n)}, \qquad   j=0,\dots,n.
$$

\begin{conjecture}[Central limit theorem for intrinsic volumes]
Let \((K_n)_{n\ge1}\) be a sufficiently regular sequence of convex bodies with \(\dim K_n=n\), and let \((Z_n)_{n\ge1}\) be the associated intrinsic volume random variables. Then there exists a sequence of positive numbers \((\sigma_n^2)_{n\ge1}\) such that $(Z_n-\mathbb EZ_n)/\sigma_n$
converges in distribution to the standard normal law.
\end{conjecture}

What is known in general is that \(Z_n\) exhibits Gaussian concentration around its mean; see~\cite{LotzMcCoyNourdinPeccatiTropp2020} and, for the conic analogue,~\cite{McCoyTropp2014}. Furthermore, Theorem~\ref{theo:exact_asympt_V_j}~(i), combined with standard though lengthy computations, can be used to verify the conjecture for $K_n=r\,\mathbb B_p^n$, $r>0$. We do not pursue these details here.

\section{Main results on curvature measures}\label{sec:main_results_curvature_measures}
\subsection{Curvature measures}
Curvature measures, introduced by~\citet{Federer1959},  are finite measures on the boundary of a convex body that localize the notion of intrinsic volumes. The following facts can be found in the book of~\citet[Chapter~4]{schneider_book_convex_cones_probab_geom}; see also~\cite{Schneider1993} and~\cite[Section~14.2]{schneider_weil_book}.
Let \(K\subseteq \mathbb R^n\) be a convex body.  For a Borel set \(B\subseteq \partial K\) and \(\rho\ge 0\), define the local outer parallel set over \(B\) by
\[
A_\rho(K,B)
:=
\bigl\{y\in (K+\rho \mathbb B_2^n)\setminus K:\ p_K(y)\in B\bigr\},
\]
where \(p_K(y)\) is the metric projection of \(y\in \mathbb R^n\) onto \(K\).

The curvature measures $\Phi_0(K,\cdot),\dots,\Phi_{n-1}(K,\cdot)$ are the unique finite Borel measures on \(\partial K\) characterized by the local Steiner formula
\[
\Vol_n \bigl(A_\rho(K,B)\bigr)
=
\sum_{j=0}^{n-1}\kappa_{\,n-j}\,\rho^{\,n-j}\,\Phi_j(K,B),
\qquad \rho\ge 0,
\]
for every Borel set \(B\subseteq \partial K\); see Equations~(4.12), (4.19) in~\cite{SchneiderBook} or Equation~(14.12) in~\cite{schneider_weil_book}. In particular, taking \(B=\partial K\), one recovers the Steiner formula for \(K+\rho \mathbb B_2^n\), and hence $\Phi_j(K,\partial K)=V_j(K)$, $j=0,\dots,n-1$, where \(V_j(K)\) denotes the \(j\)-th intrinsic volume.

One important special case is  $\Phi_{n-1}(K, \cdot)= \frac 12 \mathcal H^{n-1}(\cdot) \llcorner \partial K$; see Equations~(4.31), (4.19) in~\cite{SchneiderBook}.  Here $\mathcal H^{n-1}\llcorner \partial K$ denotes the restriction of the $(n-1)$-dimensional Hausdorff measure $\mathcal H^{n-1}$ to $\partial K$, i.e., the surface measure on $\partial K$.

Assume now that \(B\subseteq \partial K\) is Borel and $\partial K$ is of class \(C^2\) in a neighborhood of $B$.
Let $\lambda_1(x),\dots,\lambda_{n-1}(x)$ be the principal curvatures at \(x\in \partial K\).
Then,  for \(m=1,\dots,n\),
\begin{equation}\label{eq:density_curvature_measure}
\Phi_{n-m}(K,B)
=
\frac{1}{m\,\kappa_m}
\int_B
\sigma_{m-1}\bigl(\lambda_1(x),\dots,\lambda_{n-1}(x)\bigr)\,
d\mathcal H^{n-1}(x),
\end{equation}
where $\sigma_\ell$ is the $\ell$-th elementary symmetric polynomial defined in~\eqref{eq:sigma_eleme_symm_poly}; see Equations~(4.25), (4.19), (2.36) in~\cite{SchneiderBook} or~\cite[p.~607]{schneider_weil_book}.

\medskip
In the following we shall obtain  several exact results on the curvature measures of the weighted $\ell_p$-ball $\mathbb B_p^n(a)$. Since these results are too technical to be stated here, we just briefly mention them:
\begin{itemize}
\item Proposition~\ref{prop:curvature_measures_density} provides an explicit formula for the density of $\Phi_{n-m}(\mathbb B_p^n(a), \cdot)$ with respect to the surface measure $\mathcal H^{n-1}$ on $\partial \mathbb B_p^n(a)$.
\item Proposition~\ref{prop:mixed_moments_curvature_measure} gives a closed formula for the integral of $|x_1|^{\lambda_1} \ldots |x_n|^{\lambda_n}$ w.r.t.\ the measure $\Phi_{n-m}(\mathbb B_p^n(a), \cdot)$. In particular, taking $\lambda_1=\ldots = \lambda_n = 0$ we shall obtain a formula for the intrinsic volumes $V_{n-m}(\mathbb B_p^n(a))$, proving Theorems~\ref{theo:V_j_ell_p_ellipsoid_exact} and~\ref{theo:V_j_ell_p_ball_exact}.
\end{itemize}

\subsection{A continuum of Maxwell--Poincar\'{e}--Borel laws  for curvature measures}
A generalization of the classical  Maxwell--Poincar\'{e}--Borel theorem to  $\ell_p$-balls due to~\citet{Mogulskii1991English} states that if $X_n=(X_{1;n},\ldots, X_{n;n})$ is a random point on $\partial \mathbb B_p^n$ distributed according to the normalized surface measure $(\mathcal H^{n-1} \llcorner \partial \mathbb B_p^n) /\mathcal H^{n-1}(\partial \mathbb B_p^n)$, with $1\leq p < \infty$, then for every $r\in \N$,
\begin{equation}\label{eq:maxwell_for_ell_p_weak_conv}
n^{1/p} (X_{1;n},\ldots, X_{r;n}) \toweak p^{1/p}(\xi_{1},\ldots, \xi_{r}),
\end{equation}
where $\overset{w}{\longrightarrow}$ denotes weak convergence and  the components $\xi_1,\ldots, \xi_r$ are independent identically distributed (i.i.d.)\ each having the $p$-Gaussian probability density
\begin{equation}\label{eq:p_gauss_density}
f_{p}(u)
=
\frac{1}{2\Gamma(1+\frac1p)}\,e^{-|u|^p},
\qquad u\in\mathbb R.
\end{equation}
The classical Maxwell--Poincar\'{e}--Borel theorem (see, e.g., \cite{DiaconisFreedman1987}) corresponds to the case $p=2$. Then $\sqrt 2\, \xi_1,\ldots, \sqrt 2 \,\xi_r$ are independent and  standard Gaussian.
Around the same time, \citet{Borovkov1991} and~\citet{RachevRuschendorf1991} proved~\eqref{eq:maxwell_for_ell_p_weak_conv}, \eqref{eq:p_gauss_density} if $X_n$ is distributed according to the so-called cone measure on the boundary of $\mathbb B_p^n$, or uniformly distributed on $\mathbb B_p^n$. These two measures admit a simple probabilistic representation in terms of i.i.d.\ $p$-Gaussian random variables~\cite{BartheGuedonMendelsonNaor2005,Berman1980,Borovkov1991,NaorRomik2003,SchechtmanZinn1990}, whereas for the normalized surface measure on $\partial \mathbb B_p^n$ (which is different from the cone measure unless $p=1,2$) such representation is not known. \citet{NaorRomik2003} and~\citet{Naor2007} showed that the normalized surface measure and the cone measure are ``close'' in the sense of total variation distance, which gives another proof of the Maxwell--Poincar\'{e}--Borel theorem for the normalized surface measure. Several papers including~\cite{GantertKimRamanan2017,KimRamanan2018,LiaoRamanan2024} studied large deviations for \emph{random} projections of the uniform distribution on $\mathbb B_p^n$.  Results of Maxwell--Poincar\'{e}--Borel  type are also available for the uniform distribution on Orlicz balls, which generalize $\ell_p$-balls; see~\cite{KabluchkoProchno2021,JohnstonProchno2023,FruehwirthProchno2024,BartheWolff2023}.

The next theorem proves  Maxwell--Poincar\'{e}--Borel type laws for normalized curvature measures $\Phi_{j(n)}(\mathbb B_p^n, \cdot)/ V_{j(n)}(\mathbb B_p^n)$. Since these  have an additional index $j=j(n)$, which we allow to depend on $n$, we in fact obtain a ``continuum'' of different laws indexed by $\alpha = \lim_{n\to \infty} j(n)/n \in [0,1]$.  

\begin{theorem}[Maxwell--Poincar\'{e}--Borel laws for curvature measures]\label{theo:maxwell_curvature}
Fix $1<p <\infty$. For $n\geq 2$ let $j(n)\in \{0,\ldots, n-1\}$ and let $X_n=(X_{1;n},\ldots, X_{n;n})$ be a random point on $\partial \mathbb B_p^n$ distributed according to the normalized curvature measure $\Phi_{j(n)}(\mathbb B_p^n, \cdot)/ V_{j(n)}(\mathbb B_p^n)$, i.e.,
$$
\P[X_n \in A] = \frac{\Phi_{j(n)}(\mathbb B_p^n, A)}{V_{j(n)}(\mathbb B_p^n)},
\qquad
\text{ for all Borel } A\subseteq \partial\mathbb B_p^n.
$$
\begin{itemize}
\item[(i)] Bulk regime:
Suppose that $j(n) /n \to\alpha \in (0,1)$ as $n\to\infty$.  Then, for every $r\in \N$,
$$
n^{1/p} (X_{1;n}, \ldots, X_{r;n}) \toweak (p/\alpha)^{1/p} (\xi_1^{(\alpha)},\ldots, \xi_r^{(\alpha)}),
$$
where the random variables  $\xi_1^{(\alpha)},\ldots, \xi_r^{(\alpha)}$ are independent and each $\xi_i^{(\alpha)}$ has the density
\begin{equation}\label{eq:f_p_alpha_density_in_maxwell_laws}
f_{p,\alpha}(u) = \alpha \frac{e^{-|u|^p-\theta_{p,\alpha} |u|^{2p-2}}}{\mathcal I_p(\theta_{p,\alpha})}
+
(1-\alpha)\frac{|u|^{p-2}e^{-|u|^p-\theta_{p,\alpha} |u|^{2p-2}}}{\mathcal J_p(\theta_{p,\alpha})}, \qquad u\in \R.
\end{equation}
Here $\theta_{p,\alpha}$ is the same as in Theorem~\ref{theo:exact_asympt_V_j} (i), i.e., the unique solution in $(0,\infty)$ to
\begin{equation}\label{eq:theta_alpha_equation_theo_maxwell_curvature}
\theta_{p,\alpha}\,\frac{\mathcal J_p(\theta_{p,\alpha})}{\mathcal I_p(\theta_{p,\alpha})}
=
\frac{(1-\alpha) p}{2(p-1)\alpha}.
\end{equation}
\item[(ii)] Left-edge regime: Let $j(n) = j \in \N\cup\{0\}$ be constant.  Then, for every $r\in \N$,
$$
n^{1/p} (X_{1;n}, \ldots, X_{r;n}) \toweak p^{1/p} (\eta_1,\ldots, \eta_r),
$$
where the random variables $\eta_1,\ldots, \eta_r$ are independent and each $\eta_i$ has the density
\begin{equation}\label{eq:g_p_u_density_maxwell}
g_p(u)
=
\frac{(p-1)\sqrt{\lambda_0}}{\sqrt{\pi}}\,
|u|^{p-2}e^{-\lambda_0 |u|^{2p-2}},
\qquad u\in\mathbb R,
\qquad
\lambda_0 := \left(
\frac{p\,\Gamma\!\left(\frac{2p-1}{2p-2}\right)}{\sqrt{\pi}}
\right)^{\frac{2(p-1)}{p}}.
\end{equation}
\item[(iii)] Right-edge regime: Let $j(n) = n-m$, where $m\in \N$ is fixed. Then, for every $r\in \N$,
$$
n^{1/p} (X_{1;n}, \ldots, X_{r;n}) \toweak p^{1/p} (\xi_1,\ldots, \xi_r),
$$
where the random variables $\xi_1,\ldots, \xi_r$ are independent and have the same $p$-Gaussian density
\[
f_{p}(u)
=
\frac{1}{2\Gamma(1+\frac1p)}\,e^{-|u|^p},
\qquad u\in\mathbb R.
\]
\end{itemize}
\end{theorem}

\begin{example}[Classical Maxwell--Poincar\'{e}--Borel]
Let $p=2$. All normalized curvature measures $\Phi_j(\mathbb B_2^n, \cdot)/V_j(\mathbb B_2^n)$, $j=0,\ldots, n-1$,  of $\mathbb B_2^n$, being rotationally invariant,  coincide with the uniform distribution on the unit sphere $\mathbb S^{n-1}$. On the other hand,  using~\eqref{eq:Phi_alpha_def_prop_maximizer_theo_asympt_V_j_p=2} one obtains $\theta_{2,\alpha}=\frac{1-\alpha}{\alpha}$, for  all $\alpha \in (0,1)$. Hence, $f_{2,\alpha}(u)
=
(\pi\alpha)^{-1/2}\,e^{-u^2/\alpha}$,  $u\in\mathbb R$, is the density of a centered Gaussian distribution with variance $\alpha/2$. So, for $p=2$  Part~(i) becomes the classical Maxwell--Poincar\'{e}--Borel theorem. The same conclusion applies to Parts~(ii) and (iii) since the densities $f_{2}(u)$ and $g_2(u)$ coincide with the  centered Gaussian density of variance $1/2$.
\end{example}

\begin{example}[Projections of the normalized surface measure on $\partial \mathbb B_p^n$]
Let $j(n) = n-1$ and recall that  $2 \Phi_{n-1}(\mathbb B_p^n, \cdot)= \mathcal H^{n-1}(\cdot) \llcorner \partial \mathbb B_p^n$.  So $X_n$ is distributed according to the normalized surface measure on $\partial \mathbb B_p^n$.  Part~(iii) of Theorem~\ref{theo:maxwell_curvature} with $m=1$ recovers the result of~\citet{Mogulskii1991English} as stated in~\eqref{eq:maxwell_for_ell_p_weak_conv}, \eqref{eq:p_gauss_density}. (Note, however, that the papers~\cite{Borovkov1991,Mogulskii1991English,RachevRuschendorf1991} more generally consider the setting in which $r=r(n)$ may depend on $n$ and provide estimates for the total variation distance.)
\end{example}

\begin{remark}[Continuity at the right edge]
Part~(iii) of Theorem~\ref{theo:maxwell_curvature}  is the formal limit of Part~(i) as $\alpha \uparrow 1$ in the sense that $\lim_{\alpha \uparrow 1} f_{p,\alpha} (u)= f_p(u)$ for all $u\in \R$.  To prove this, observe that by dominated convergence,  $\mathcal I_p(\theta) \to \mathcal I_p(0)>0$ and $\mathcal J_p(\theta) \to \mathcal J_p(0)>0$ as $\theta \downarrow 0$, which implies  $\theta\,\mathcal J_p(\theta)/ \mathcal I_p(\theta) \to 0$ as $\theta \downarrow 0$.  Since  $\theta_{p,\alpha}$ is the unique solution to~\eqref{eq:theta_alpha_equation_theo_maxwell_curvature} whose right-hand side goes to $0$ as $\alpha \uparrow 1$, it follows that  $\theta_{p,\alpha}\to 0$ as $\alpha\uparrow 1$. Substituting this in Equation~\eqref{eq:f_p_alpha_density_in_maxwell_laws} gives the claim.
\end{remark}

\begin{remark}[Continuity at the left edge]
Similarly, Part~(ii) of Theorem~\ref{theo:maxwell_curvature} is the formal limit of Part~(i) as $\alpha \downarrow 0$ in the sense that $\alpha^{-1/p} \xi_1^{(\alpha)}$ converges weakly to $\eta_1$ as $\alpha \downarrow 0$. We only sketch the argument.
Using the change of variables $x=\theta^{-1/(2p-2)}y$ in~\eqref{eq:def_I_p} and~\eqref{eq:def_J_p} together with  dominated convergence one shows that
\[
\mathcal I_p(\theta)\sim A_p\,\theta^{-\frac1{2p-2}},
\quad
\mathcal J_p(\theta)\sim B_p\,\theta^{-\frac12},
\quad
\theta\frac{\mathcal J_p(\theta)}{\mathcal I_p(\theta)}
\sim
\frac{B_p}{A_p}\,
\theta^{\frac{p}{2(p-1)}},
\quad
\theta\to\infty,
\]
where
\[
A_p:=\int_{\mathbb R}e^{-|y|^{2p-2}}\,dy
=
2\Gamma\!\left(\frac{2p-1}{2p-2}\right),
\qquad
B_p:=\int_{\mathbb R}|y|^{p-2}e^{-|y|^{2p-2}}\,dy
=
\frac{\sqrt{\pi}}{p-1}.
\]
From Equation~\eqref{eq:theta_alpha_equation_theo_maxwell_curvature} one obtains $\alpha^{\frac{2(p-1)}{p}}\theta_{p,\alpha} \to \lambda_0$ as $\alpha\downarrow 0$, where $\lambda_0$ is given by~\eqref{eq:g_p_u_density_maxwell}. Substituting this into~\eqref{eq:f_p_alpha_density_in_maxwell_laws} one shows that $\alpha^{1/p}f_{p,\alpha}(\alpha^{1/p}u) \to g_p(u)$ as $\alpha \downarrow 0$, where $g_p(u)$ is given by~\eqref{eq:g_p_u_density_maxwell}.
\end{remark}

A phenomenon similar to the one established in Theorem~\ref{theo:maxwell_curvature} for $\ell_p$-balls should hold for more general classes of high-dimensional convex bodies.

\begin{conjecture}[A continuum of Maxwell--Poincar\'{e}--Borel laws for curvature measures]
Consider a  ``sufficiently regular'' sequence of convex bodies $(K_n)_{n\geq 1}$ such that $\dim K_n = n$ for every $n$ and $\frac 1n \log \Vol_n K_n \to c\in \R$ as $n\to\infty$. Then there exists a family of probability measures $(\nu_\alpha)_{\alpha\in [0,1]}$ such that the following holds.  If $X_n=(X_{1;n},\ldots, X_{n;n})$ is a random point on $\partial K_n$ distributed according to the normalized curvature measure $\Phi_{j(n)}(K_n, \cdot)/ V_{j(n)}(K_n)$, where $j(n)/n \to \alpha \in [0,1]$  as $n\to\infty$, then for every $r\in \N$, the joint distribution of any $r$ coordinates of $X_n$ converges weakly to the product measure $\nu_\alpha^{\otimes r}$.
\end{conjecture}

Theorem~\ref{theo:maxwell_curvature} confirms this conjecture for $K_n = n^{1/p} \mathbb B_p^n$ with  $1<p<\infty$ (subject to  some additional restrictions if $\alpha =0$ or $1$). In the following examples we verify the conjecture for $p=1$ (scaled crosspolytopes, $K_n = n\, \mathbb B_1^n$) and $p=\infty$ (cubes,  $K_n = [-1,1]^n$). Since both families are polytopes and their external angles at all $j$-dimensional faces are equal, the $j$-th normalized  curvature measure  is the uniform distribution on the $j$-skeleton of the corresponding polytope; see Equation~(4.22) in~\cite{SchneiderBook}.   (The $j$-skeleton is defined as the union of all $j$-dimensional faces of the polytope.)

\begin{example}[Maxwell laws for curvature measures of cubes]
Let \(j\in\{0,\dots,n\}\). A convenient probabilistic construction of a random vector $X_n=(X_{1;n},\dots,X_{n;n})$
which is uniformly distributed on the \(j\)-skeleton of \([-1,1]^n\) is the following. Choose a random subset $J\subseteq \{1,\dots,n\}$
uniformly among all subsets of cardinality \(j\). Let \(U_1,\dots,U_n\) be i.i.d.\ random variables with law \(\mathrm{Unif}[-1,1]\), and let \(\varepsilon_1,\dots,\varepsilon_n\) be i.i.d.\ symmetric \(\{\pm1\}\)-valued random variables. Let all these random objects be independent. Define
\[
X_{i;n}=
\begin{cases}
U_i, & i\in J,\\[1mm]
\varepsilon_i, & i\notin J.
\end{cases}
\]
Then \(X_n\) is uniformly distributed on the \(j\)-skeleton of \([-1,1]^n\) with respect to \(j\)-dimensional Hausdorff measure; equivalently, its law is the normalized \(j\)-th curvature measure of \([-1,1]^n\).
Now consider the asymptotic regime in which \(n\to\infty\) and \(j=j(n)\) satisfies $j(n)/n \to \alpha\in[0,1]$ as $n\to\infty$.
Then, for every fixed \(r\in\mathbb N\), the distribution of $(X_{1;n},\dots,X_{r;n})$ converges weakly to the product measure $\nu_{\infty, \alpha}^{\otimes r}$,
where
$\nu_{\infty, \alpha}
=
\alpha\,\mathrm{Unif}[-1,1]
+
(1-\alpha)\left(\frac12\delta_{-1}+\frac12\delta_1\right).
$
In other words, the first \(r\) coordinates become asymptotically independent, and each coordinate converges in distribution to the mixture \(\nu_{\infty, \alpha}\).
\end{example}

\begin{example}[Maxwell laws for curvature measures of crosspolytopes]
Let \(j\in\{0,\dots,n-1\}\). A convenient probabilistic construction of a random vector $X_n=(X_{1;n},\dots,X_{n;n})$
which is uniformly distributed on the \(j\)-skeleton of \(n\mathbb B_1^n\) is as follows. Choose a random subset $J\subseteq \{1,\dots,n\}$
uniformly among all subsets of cardinality \(j+1\). Let $\eps_1,\dots, \eps_n$ be i.i.d.\ symmetric \(\{\pm1\}\)-valued random variables, and let $E_1,\ldots, E_n$ be i.i.d.\ \(\mathrm{Exp}(1)\) random variables, independent of each other and of \(J\). Define
\[
X_{i;n}
=
\begin{cases}
n\, \eps_i\,\frac{E_i}{\sum_{k\in J}E_k}, & i\in J,\\[3mm]
0, & i\notin J.
\end{cases}
\]
Then \(X_n\) is uniformly distributed on the \(j\)-skeleton of \(n\mathbb B_1^n\) with respect to \(j\)-dimensional Hausdorff measure.
Now consider the asymptotic regime in which \(n\to\infty\) and \(j=j(n)\) satisfies $j(n)/n \to \alpha\in[0,1]$.
Then, for every fixed \(r\in\mathbb N\), the distribution of $(X_{1;n},\dots,X_{r;n})$ converges weakly to the product measure $\nu_{1,\alpha}^{\otimes r}$, where $\nu_{1,\alpha}=(1-\alpha)\delta_0+\alpha\,\mathrm{Laplace}(\alpha)$.
Here, for \(\alpha>0\), \(\mathrm{Laplace}(\alpha)\) denotes the symmetric Laplace distribution on \(\mathbb R\) with density
$x\mapsto (\alpha/2)e^{-\alpha|x|}$, $x\in\mathbb R$,
and  $\mathrm{Laplace}(0)=\delta_0$.
\end{example}

\begin{remark}[Continuity at $p=\infty$ and at $p=1$]
Let $(\nu_{p,\alpha})_{\alpha \in [0,1]}$ denote the one-dimensional marginals of the limiting product measures appearing in Theorem~\ref{theo:maxwell_curvature}. Then, for every $\alpha \in [0,1]$, the measure $\nu_{p,\alpha}$ converges weakly to $\nu_{\infty, \alpha}$ as $p\to\infty$, and to $\nu_{1, \alpha}$ as $p\downarrow 1$. We omit the verification.
\end{remark}

\section{Proofs: Exact formulas} 
In this section we prove  Theorems~\ref{theo:V_j_ell_p_ball_exact} and~\ref{theo:V_j_ell_p_ellipsoid_exact}.

\subsection{Principal curvatures}
We begin by calculating  principal curvatures of the hypersurface in $\R^n$ given by an equation of the form $\varphi_1(x_1) +\ldots + \varphi_n(x_n) = 1$.

Let $\varphi_1,\ldots, \varphi_n:\R\to\R$ be functions such that, for each $i=1,\dots,n$,
\begin{itemize}
\item $\varphi_i(0)=0$ and $\varphi_i(z)>0$ for  $z\neq 0$;
\item $\varphi_i$ is strictly convex;
\item $\varphi_i$ is of class $C^1$ on $\mathbb R$, and of class $C^2$ on $\mathbb R\setminus\{0\}$.
\end{itemize}
A typical example (and in fact the one we are actually interested in) is  $\varphi_i(x_i) = |a_ix_i|^p$, where $p>1$ and $a_i>0$. The above properties imply that $\varphi_i'(0)=0$ and $\varphi_i'(x_i) \neq 0$ for $x_i\neq 0$. We consider
$$
K_n:=\{x\in\mathbb R^n:F(x)\leq 1\}  \qquad \text{ with }  \qquad F(x):=\sum_{i=1}^n \varphi_i(x_i), \quad x\in \R^n.
$$
It is easy to check that  $K_n$ is a convex body containing $0$ in its interior.  The boundary of $K_n$ is $\partial K_n = \{x\in \R^n: F(x) = 1\}$. Since the function $F$ is $C^1$ on $\R^n$ and since its gradient
$$
\nabla F(x)=\bigl(\varphi_1'(x_1),\dots,\varphi_n'(x_n)\bigr)
$$
does not vanish for $x\in \partial K_n$, the regular level set theorem implies  that $\partial K_n$ is a $C^1$-hypersurface. The outer unit normal at $x\in \partial K_n$ is given by
\[
\nu(x)=\frac{\nabla F(x)}{|\nabla F(x)|}.
\]
The tangent space at $x$ is  $T_x (\partial K_n)= \{v\in\mathbb R^n:\sum_{i=1}^n \varphi_i'(x_i) v_i=0\}$.  For the differential-geometric notions used below, such as principal curvatures, we refer to~\cite[Sections~2.4 and~2.5]{SchneiderBook}.

\begin{proposition}[Principal curvatures of $\partial K_n$]\label{prop:curvatures_separable_hypersurface}
Fix a point $x\in \partial K_n$ such that $x_i\neq 0$ for all $i=1,\ldots,n$.
Then $\partial K_n$ is a $C^2$-hypersurface in the neighborhood of $x$ and
the principal curvatures $\lambda_1(x),\dots,\lambda_{n-1}(x)$  are precisely the solutions (in the variable $\lambda$) of the equation
\begin{equation}\label{eq:characteristic_eq_principal_curvatures}
\sum_{i=1}^n \bigl(\varphi_i'(x_i)\bigr)^2
\prod_{\substack{j=1\\ j\neq i}}^n
\left(\varphi_j''(x_j)-|\nabla F(x)|\,\lambda\right) = 0.
\end{equation}
The degree of the polynomial on the right-hand side is exactly $n-1$.
\end{proposition}

\begin{proof}
The function $F$ is of class $C^2$ in the neighborhood of $x$. Its gradient and Hessian are given by
$$
\nabla F(x)=\bigl(\varphi_1'(x_1),\dots,\varphi_n'(x_n)\bigr),\qquad
\nabla^2F(x)=\operatorname{diag}\bigl(\varphi_1''(x_1),\dots,\varphi_n''(x_n)\bigr).
$$
The second fundamental form of the hypersurface $\partial K_n =\{F(x)=1\}$ at $x\in \partial K_n$ can be expressed in terms of the Hessian and gradient of $F$ by
$$
\mathrm{II}_x(v,w)
=
\frac{\nabla^2F(x)[v,w]}{|\nabla F(x)|}
=
\frac{1}{|\nabla F(x)|}\sum_{i=1}^n \varphi_i''(x_i)\,v_i w_i,
\qquad v,w\in T_x(\partial K_n),
$$
see, for example, Equation~(2.4) in~\cite{Jubin2023}.
A number \(\lambda\in\mathbb R\) is a principal curvature of \(\partial K_n\) at \(x\) if and only if there exists $v\in T_x(\partial K_n)\setminus\{0\}$ such that
$$
\mathrm{II}_x(v,w)=\lambda \,\langle v,w\rangle
\qquad\text{for all } w\in T_x(\partial K_n).
$$
Equivalently,
$$
\frac{1}{|\nabla F(x)|}\sum_{i=1}^n \varphi_i''(x_i)\,v_i w_i
=
\lambda\sum_{i=1}^n v_i w_i
\qquad\text{for all } w\in T_x(\partial K_n).
$$
This is equivalent to saying that the vector
$$
\bigl(\varphi_1''(x_1)v_1-|\nabla F(x)|\lambda v_1,\dots,
\varphi_n''(x_n)v_n-|\nabla F(x)|\lambda v_n\bigr)
$$
is orthogonal to \(T_x(\partial K_n)\). Now,  $T_x(\partial K_n)^\perp$ is the line spanned by $\nabla F(x)$. Hence, $\lambda$ is a principal curvature if and only if the linear system
$$
\bigl(\varphi_i''(x_i)-|\nabla F(x)|\lambda\bigr)v_i
=
\mu\,\varphi_i'(x_i),
\qquad
i=1,\dots,n,
\qquad
\sum_{i=1}^n \varphi_i'(x_i)\,v_i=0
$$
admits a nonzero solution \((v_1,\dots,v_n,\mu)\). Equivalently,
$$
\det
\begin{pmatrix}
\varphi_1''(x_1)-|\nabla F(x)|\lambda & 0 & \cdots & 0 & -\varphi_1'(x_1)\\
0 & \varphi_2''(x_2)-|\nabla F(x)|\lambda & \cdots & 0 & -\varphi_2'(x_2)\\
\vdots & \vdots & \ddots & \vdots & \vdots\\
0 & 0 & \cdots & \varphi_n''(x_n)-|\nabla F(x)|\lambda & -\varphi_n'(x_n)\\
\varphi_1'(x_1) & \varphi_2'(x_2) & \cdots & \varphi_n'(x_n) & 0
\end{pmatrix}
=0.
$$
Expanding this determinant, one obtains the equivalent equation~\eqref{eq:characteristic_eq_principal_curvatures}.
\end{proof}

\subsection{Elementary symmetric polynomials of the principal curvatures}
Let $\sigma_{m-1}(x_1,\ldots, x_r)$ denote the $(m-1)$-st elementary symmetric polynomial in $r$ variables
$$
\sigma_{m-1}(x_1,\dots, x_r) = \sum_{1\leq i_1 < \ldots < i_{m-1} \leq r} x_{i_1} \ldots x_{i_{m-1}}=\sum_{\substack{I\subset \{1,\ldots,r\}\\|I|=m-1}}\prod_{i\in I}x_i.
$$

\begin{proposition}[Elementary symmetric functions of the principal curvatures] \label{prop:curvatures_separable_hypersurface_symmetric_poly}
Fix a point $x\in \partial K_n$ such that $x_i\neq 0$ for all $i=1,\ldots,n$.
Let $\lambda_1(x),\dots,\lambda_{n-1}(x)$ be the principal curvatures of $\partial K_n$ at $x$. 
Then, for all \(m=1,\dots,n\),
\begin{equation}\label{eq:curvatures_elementary_symm_poly}
\sigma_{m-1}\bigl(\lambda_1(x),\dots,\lambda_{n-1}(x)\bigr)\\
=\frac{1}{
\left(\sum_{i=1}^n (\varphi_i'(x_i))^2\right)^{\frac{m+1}{2}}
}\sum_{i=1}^{n}(\varphi_i'(x_i))^2\sum_{\substack{I\subset \{1,\ldots,r\}\setminus\{i\}\\|I|=m-1}}\prod_{j\in I}\phi''_j(x_j).
\end{equation}

\end{proposition}
\begin{remark}
Observe that the right-hand side of formula~\eqref{eq:curvatures_elementary_symm_poly} can also be written as
$$
\frac{1}{\left(\sum_{i=1}^n (\varphi_i'(x_i))^2\right)^{\frac{m+1}{2}}}
\sum_{i=1}^n \bigl(\varphi_i'(x_i)\bigr)^2\,
\sigma_{m-1}\!\bigl(
\varphi_1''(x_1),\dots,\widehat{\varphi_i''(x_i)},\dots,\varphi_n''(x_n)
\bigr)
$$
in order to conform with notation introduced in Example~\ref{ex:ellipsoids_intrinsic_volumes}. Here, \(\widehat{\varphi_i''(x_i)}\) indicates that the \(i\)-th entry is omitted. However, throughout the proofs, for the sake of notational simplicity, we shall use the alternative notation given above.
\end{remark}

\begin{proof}
Since the principal curvatures \(\lambda_1(x),\dots,\lambda_{n-1}(x)\) are the solutions of~\eqref{eq:characteristic_eq_principal_curvatures}, we may write
$$
\sum_{i=1}^n \bigl(\varphi_i'(x_i)\bigr)^2
\prod_{\substack{j=1\\ j\neq i}}^n
\left(\varphi_j''(x_j)-|\nabla F(x)|\,\lambda\right)
=
|\nabla F(x)|^{\,n+1}\prod_{r=1}^{n-1}(\lambda_r(x)-\lambda).
$$
Expanding both sides in powers of \(\lambda\) and comparing coefficients, one obtains~\eqref{eq:curvatures_elementary_symm_poly}.
\end{proof}

\subsection{Curvature measures on the weighted \texorpdfstring{$\ell_p$}{ell\_p}-balls}
Fix $n\ge 2$,  $1<p<\infty$, the weights $a_1,\ldots, a_n >0$ and consider the function $F:\R^n\to [0,\infty)$ given by
$$
F(x):= \sum_{r=1}^n |a_r x_r|^p, \qquad x= (x_1,\ldots, x_n)\in \R^n.
$$
The weighted $\ell_p$-ball and its boundary are then given by
\[
\mathbb B_p^n(a) = 
\{x\in \R^n: F(x) \leq 1\} \qquad\text{and}
\qquad
\partial \mathbb B_p^n(a) = \{x\in \R^n: F(x) = 1\}.
\]

\begin{proposition}[Boundary of  weighted $\ell_p$-balls]
With the above notation, the following hold.
\begin{itemize}
\item[(a)] For every $p>1$, the boundary  $\partial \mathbb B_p^n(a)$ is a $C^1$-hypersurface.  The outer unit normal at $x\in \partial \mathbb B_p^n(a)$ (the Gauss map of $\mathbb B_p^n(a)$ at $x$) is given by
\begin{equation}\label{eq:nu_outer_normal}
\nu(x)=\frac{\nabla F(x)}{\|\nabla F(x)\|}
=
\frac{\left(a_1^p |x_1|^{p-2}x_1,\ldots, a_n^p |x_n|^{p-2}x_n\right)}
{\left(\sum_{i=1}^n a_i^{2p}|x_i|^{2p-2}\right)^{1/2}}.
\end{equation}
Let $q=\frac{p}{p-1}$ be the H\"older conjugate of $p$.  The support function of $\mathbb B_p^n(a)$ is
\begin{equation}\label{eq:support_funct_ell_p_ellipsoid}
h_{\mathbb B_p^n(a)}(u) := \sup_{x\in \mathbb B_p^n(a)} \langle u, x\rangle
=
\left(\sum_{i=1}^n |u_i/a_i|^q\right)^{1/q}, \qquad u\in \R^n.
\end{equation}
\item[(b)] For every $p>1$, the boundary is $C^2$ near every point $x\in \partial \mathbb B_p^n(a)$ such that $x_i \neq 0$ for every $i=1,\ldots, n$.  If $p\geq 2$, then the boundary $\partial \mathbb B_p^n(a)$ is $C^2$ at every point.
\item[(c)] For all $p>1$, the curvature measures of $\mathbb B_p^n(a)$ do not charge the set $\partial \mathbb B_p^n(a) \cap \{x_1\ldots x_n=0\}$.
\end{itemize}
\end{proposition}
\begin{proof}
The function $F(x)$ is $C^1$ for every $p>1$ and its gradient is
$$
\nabla F(x)
=
p\left(a_1^p |x_1|^{p-2}x_1, \ldots, a_n^p |x_n|^{p-2}x_n\right), \qquad x\in \R^n.
$$
In particular, $\nabla F(x) \neq 0$ for every $x\in \partial \mathbb B_p^n(a)$. It follows that $\partial \mathbb B_p^n(a)$ is a hypersurface of class $C^1$ for every $p>1$ with an outer unit normal given by~\eqref{eq:nu_outer_normal}. If, moreover, $p>2$, then the function $F(x)$ is $C^2$ and  $\partial \mathbb B_p^n(a)$ is a hypersurface of class $C^2$.  For every $p>1$, the function $x_i \mapsto |x_i|^p$ is $C^2$ on $\R\backslash \{0\}$. It follows that $F(x)$ is $C^2$ on the complement of $\{x_1\ldots x_n = 0\}$. The formula for the support function follows from the H\"{o}lder inequality, which implies that the right-hand side of~\eqref{eq:support_funct_ell_p_ellipsoid} is an upper bound for $h_{\mathbb{B}_p^n(a)}(u)$; moreover, it is easy to verify that this bound is attained for some $x = x(u) \in \mathbb{B}_p^n(a)$. This proves (a) and~(b).

\vspace*{2mm}
\noindent
Proof of (c).  If $p\geq 2$, then the boundary of $\mathbb B_p^n(a)$ is $C^2$ everywhere and the curvature measures are absolutely continuous w.r.t.\ the surface measure $\mathcal H^{n-1} \llcorner \partial \mathbb B_p^n(a)$; see~\eqref{eq:density_curvature_measure}. This gives the claim.

For $1<p<2$, we observe that although the boundary is not $C^2$, the support function of $\mathbb B_{p}^n(a)$, given by~\eqref{eq:support_funct_ell_p_ellipsoid}, is $C^2$ on the unit sphere $\mathbb S^{n-1}$ since $q>2$.  This implies that the area measures (living on the unit sphere in $\R^n$) are absolutely continuous w.r.t.\ the surface measure on $\mathbb S^{n-1}$; see~\cite[p.~285]{Schneider1993} or~\cite[Theorem~4.9]{HugWeil2020}.   (The area measures are the images of the support measures under the Gauss map $x\mapsto \nu(x)$; see~\cite[Theorem~4.2.5]{SchneiderBook}.) It follows from~\eqref{eq:nu_outer_normal} that the Gauss map $x\mapsto \nu(x)$ is a continuous bijection between $\partial \mathbb B_p^n(a)$ and $\mathbb S^{n-1}$. In particular, the inverse of the Gauss map is explicitly given by
\[
\mathcal \nu^{-1}(u)
=
\frac{(a_1^{-q}|u_1|^{q-2}u_1,\ldots, a_n^{-q}|u_n|^{q-2}u_n)}{h_{\mathbb B_p^n(a)}(u)^{q-1}},
\qquad u\in \mathbb S^{n-1}.
\]
Hence, $\nu$ maps the set $\partial \mathbb B_p^n(a)\cap \{x_1\ldots x_n=0\}$ to the set $\mathbb S^{n-1} \cap \{u_1\ldots u_n=0\}$. The latter set is not charged by the area measures since these are absolutely continuous w.r.t.\ the surface measure on $\mathbb S^{n-1}$. Now, the curvature measures are the push-forward of the area measures under the inverse Gauss map $u\mapsto \nu^{-1}(u)$; see~\cite[Theorem~4.2.5]{SchneiderBook}.  It follows that the curvature measures do not charge the set $\partial \mathbb B_p^n(a) \cap \{x_1\ldots x_n=0\}$, which is exactly (c).
\end{proof}

\begin{proposition}[Principal curvatures for weighted $\ell_p$-balls]\label{prop:curvatures_elementary_symmetric_functions}
Let $x\in \partial \mathbb B_p^n(a)$ be such that $x_i\neq 0$ for all $i=1,\dots, n$.  Then the principal curvatures of $\mathbb B_p^n(a)$ at $x$, denoted by $\lambda_1(x),\dots,\lambda_{n-1}(x)$, are precisely the solutions of the degree $n-1$ polynomial equation
\[
\sum_{i=1}^n a_i^{2p}|x_i|^{2p-2}
\prod_{\substack{j=1\\ j\ne i}}^n
\left((p-1)a_j^p|x_j|^{p-2}
-
\left(\sum_{r=1}^n a_r^{2p}|x_r|^{2p-2}\right)^{1/2}\lambda
\right)=0.
\]
For $m=1,\ldots, n$ the $(m-1)$-st elementary symmetric polynomial of the principal curvatures is given by
\[
\sigma_{m-1}\bigl(\lambda_1(x),\dots,\lambda_{n-1}(x)\bigr)
=
\frac{(p-1)^{m-1}}
{\left(\sum_{r=1}^n a_r^{2p}|x_r|^{2p-2}\right)^{(m+1)/2}}
\sum_{i=1}^n a_i^{2p}|x_i|^{2p-2}\sum_{\substack{I\subseteq\{1,\dots,n\}\setminus\{i\}\\ |I|=m-1}}
\prod_{j\in I} a_j^p|x_j|^{p-2}.
\]
\end{proposition}
\begin{proof}
Apply Propositions~\ref{prop:curvatures_separable_hypersurface} and~\ref{prop:curvatures_separable_hypersurface_symmetric_poly} with $\varphi_i(x_i)=|a_i x_i|^p$. Then
\begin{align*}
&\varphi_i'(x_i)=p\,a_i^p |x_i|^{p-2}x_i, \quad x_i\in \R,\\
&\varphi_i''(x_i)=p(p-1)a_i^p |x_i|^{p-2}, \quad x_i\neq 0,
\end{align*}
and
$$
F(x)=\sum_{r=1}^n |a_r x_r|^p,
\qquad
|\nabla F(x)|
=
p\left(\sum_{r=1}^n a_r^{2p}|x_r|^{2p-2}\right)^{1/2},\quad x\in\mathbb{R}^n.
$$
The results follow readily.
\end{proof}
\begin{example}[Gauss curvature]\label{example:gauss_curvature}
In particular,  the Gauss curvature is
\[
K(x) := \prod_{r=1}^{n-1}\lambda_r(x)
=
\frac{(p-1)^{n-1}}
{\left(\sum_{r=1}^n a_r^{2p}|x_r|^{2p-2}\right)^{(n+1)/2}}
\sum_{i=1}^n a_i^{2p}|x_i|^{2p-2}
\prod_{\substack{j=1\\ j\ne i}}^n a_j^p|x_j|^{p-2},\quad x\in \partial \mathbb{B}_p^n(a).
\]
Using the identity $\sum_{i=1}^n |a_i x_i|^p=1$, the Gauss curvature may also be written as
\[
K(x)
=
\frac{(p-1)^{n-1}\prod_{j=1}^n a_j^p|x_j|^{p-2}}
{\left(\sum_{r=1}^n a_r^{2p}|x_r|^{2p-2}\right)^{(n+1)/2}},\quad x\in \partial \mathbb{B}_p^n(a).
\]
\end{example}

\begin{proposition}[Curvature measures of weighted $\ell_p$-balls]\label{prop:curvature_measures_density}
For every \(m=1,\dots,n\), the curvature measure $\Phi_{n-m}(\mathbb B_p^n(a),\cdot)$ is absolutely continuous with respect to the surface measure $\mathcal H^{n-1} \llcorner \partial \mathbb B_p^n(a)$ and the density is
\[
\frac{d\Phi_{n-m}\bigl(\mathbb B_p^n(a),x\bigr)}{d\mathcal H^{n-1}(x)}
=
\frac{(p-1)^{m-1}}{m\,\kappa_m}\,
\sum_{i=1}^n  \sum_{\substack{I\subseteq\{1,\dots,n\}\setminus\{i\}\\ |I|=m-1}}
\frac{
a_i^{2p}|x_i|^{2p-2}
\prod_{j\in I} a_j^p|x_j|^{p-2}
}{
\left(\sum_{r=1}^n a_r^{2p}|x_r|^{2p-2}\right)^{(m+1)/2}
}.
\]
\end{proposition}
\begin{proof}
Since the curvature measure $\Phi_{n-m}$ does not charge the set $\partial \mathbb B_p^n(a) \cap \{x_1\ldots x_n = 0\}$ and since the complement of this set, $\partial \mathbb B_p^n(a) \cap \{x_1\ldots x_n \neq 0\}$, is a collection of $2^n$ disjoint hypersurfaces of class $C^2$, the curvature measure is absolutely continuous with respect to $\mathcal H^{n-1}\llcorner \partial \mathbb B_p^n(a)$ and its density is
\[
\frac{d\Phi_{n-m}\bigl(\mathbb B_p^n(a),x\bigr)}{d\mathcal H^{n-1}(x)}
=
\frac{1}{m\kappa_m}\,
\sigma_{m-1}\bigl(\lambda_1(x),\dots,\lambda_{n-1}(x)\bigr).
\]
Plugging in the expression for $\sigma_{m-1}(\lambda_1(x),\dots,\lambda_{n-1}(x))$ given in~Proposition~\ref{prop:curvatures_elementary_symmetric_functions} completes the proof.
\end{proof}

\begin{example}[Intrinsic volumes of weighted $\ell_p$-balls]\label{exam:intrinsic_vols_ell_p_as_big_integral}
Recalling that $V_{n-m} (\mathbb B_p^n(a))$ is the total mass of the measure $\Phi_{n-m}(\mathbb B_p^n(a),\cdot)$ we immediately obtain the following formula for the intrinsic volumes of the weighted $\ell_p$-balls: For every \(m=1,\dots,n\),
\[
V_{n-m} (\mathbb B_p^n(a))
=
\frac{(p-1)^{m-1}}{m\,\kappa_m}\,
\sum_{i=1}^n  \sum_{\substack{I\subseteq\{1,\dots,n\}\setminus\{i\}\\ |I|=m-1}}
\int_{\partial \mathbb B_p^n(a)}
\frac{
a_i^{2p}|x_i|^{2p-2}
\prod_{j\in I} a_j^p|x_j|^{p-2}
}{
\left(\sum_{r=1}^n a_r^{2p}|x_r|^{2p-2}\right)^{(m+1)/2}
}\, d\mathcal H^{n-1}(x).
\]
\end{example}

\subsection{The key integral}
In this section, we evaluate the integral appearing in Example~\ref{exam:intrinsic_vols_ell_p_as_big_integral}, and more general integrals.
Let $1 < p<\infty$, $n\geq 2$ and $a_1,\dots, a_n >0$. Recall that
$$
F(x)=\sum_{r=1}^n |a_r x_r|^p,
\qquad
|\nabla F(x)|
=
p\left(\sum_{r=1}^n a_r^{2p}|x_r|^{2p-2}\right)^{1/2}.
$$
\begin{proposition}[The key integral]\label{prop:key_integral}
Let \(\alpha_1,\dots,\alpha_n\in\mathbb C\) satisfy $\Re \alpha_i>-1$ for all $i=1,\dots,n$. Then, for every \(\alpha\in\mathbb C\) satisfying
\begin{equation}\label{eq:key_integral_strip}
0<\Re\alpha<\frac{n+\sum_{i=1}^n \Re\alpha_i}{p-1}
\end{equation}
one has
\begin{multline}\label{eq:key_integral}
\int_{\partial \mathbb B_p^n(a)}
\frac{\prod_{i=1}^n |x_i|^{\alpha_i}}{\Bigl(\sum_{r=1}^n a_r^{2p}|x_r|^{2p-2}\Bigr)^{(\alpha+1)/2}}\, d\mathcal H^{n-1}(x)
\\
=
\frac{p \, \prod_{i=1}^n a_i^{-\alpha_i-1}}
{\Gamma\!\left(\dfrac{n+\sum_{i=1}^n \alpha_i-\alpha(p-1)}{p}\right)\Gamma(\frac \alpha 2)}
\int_0^\infty
\theta^{\alpha/2-1}
\prod_{i=1}^n
\left(
\int_{\mathbb R}
|y|^{\alpha_i}e^{-|y|^p-\theta a_i^2|y|^{2p-2}}\,dy
\right)\,d\theta.
\end{multline}
The integral on the left converges absolutely for every \(\alpha\in\mathbb C\), whereas the integral on the right converges absolutely for all \(\alpha\) satisfying~\eqref{eq:key_integral_strip}.
\end{proposition}

\begin{proof}
We consider a more general integral defined by
\begin{equation}\label{eq:key_integral_A(t)_def}
A(t):=\int_{\{F=t\}} H(x)\,\frac{d\mathcal H^{n-1}(x)}{|\nabla F(x)|}, \qquad t>0, \qquad H(x):=\prod_{i=1}^n |x_i|^{\alpha_i} \cdot |\nabla F(x)|^{-\alpha}.
\end{equation}
The integral on the left of~\eqref{eq:key_integral} is $p^{\alpha+1} A(1)$. Note that the integral defining $A(t)$ is absolutely convergent for all $\alpha \in \C$ since $0 < m_t < |\nabla F(x)| < M_t$ on the level set $\{F=t\}$, $t>0$. Therefore, absolute integrability reduces to the integrability of $\prod_{i=1}^n |x_i|^{\Re \alpha_i}$ on $\{F=t\}$, which is ensured by the assumption $\Re \alpha _i >-1$.

\medskip
\noindent
\textit{Step 1: Homogeneity.}
Since \(F\) is homogeneous of degree \(p\), meaning that  \(F(\lambda x)=\lambda^p F(x)\) for all $\lambda>0$ and $x\in \R^n$,  the gradient \(|\nabla F|\) is homogeneous of
degree \(p-1\), and hence \(H\) is homogeneous of degree $\deg H = \sum_{i=1}^n \alpha_i-\alpha(p-1)$.
Using the substitution \(x=t^{1/p}y\) in~\eqref{eq:key_integral_A(t)_def}, the equalities $\{F=t\}=t^{1/p}\{F=1\}$ and
$$
H(t^{1/p}y)=t^{\deg H/p}H(y),
\quad
|\nabla F(t^{1/p}y)|=t^{(p-1)/p}|\nabla F(y)|,
\quad
d\mathcal H^{n-1}(t^{1/p}y)=t^{(n-1)/p}\,d\mathcal H^{n-1}(y),
$$
one obtains
\begin{equation}\label{eq:key_integral_A_homogeneous}
A(t)= t^{(\deg H+n-p)/p}A(1) = t^{\mu-1}A(1), \qquad \text{ for all } t>0,
\end{equation}
where
$$
\mu:=\frac{n+\sum_{i=1}^n\alpha_i-\alpha(p-1)}{p}.
$$

\medskip
\noindent
\textit{Step 2: Coarea formula.}
Now consider
\[
I:=\int_{\mathbb R^n} e^{-F(x)}H(x)\,dx.
\]
At the end of this step, we shall justify that $e^{-F}H \in L^1(\R^n)$ for every $\alpha\in \C$ satisfying~\eqref{eq:key_integral_strip}.  By the coarea formula applied to the level sets of the $C^1$-function $F$,
\[
I
=
\int_{\mathbb R^n} e^{-F(x)}H(x)\,dx
=
\int_{0}^\infty \int_{\{F=t\}} e^{-F(x)}H(x) \frac{d\, \mathcal H^{n-1}(x)}{|\nabla F(x)|} \, dt
=
\int_0^\infty e^{-t}A(t)\,dt.
\]
Using the homogeneity of \(A\) stated in~\eqref{eq:key_integral_A_homogeneous}, we obtain
\[
I
=A(1)\int_0^\infty e^{-t}t^{\mu-1}\,dt
=\Gamma(\mu)A(1),
\]
whenever \(\Re \mu>0\), which is satisfied by~\eqref{eq:key_integral_strip}.  Therefore,
\[
A(1)=\frac{I}{\Gamma(\mu)}.
\]

To complete this step we have to justify the use of the coarea formula by showing that for every $\alpha\in \C$ satisfying $0 < \Re\alpha<\frac{n+\sum_{i=1}^n \Re\alpha_i}{p-1}$ one has
\[
e^{-F(x)}H(x)
=
e^{-\sum_{r=1}^n |a_r x_r|^p}
\prod_{i=1}^n |x_i|^{\alpha_i}\,|\nabla F(x)|^{-\alpha}
\in L^1(\mathbb R^n).
\]
Indeed, the function $x\mapsto \prod_{i=1}^n |x_i|^{\Re\alpha_i}\,|\nabla F(x)|^{-\Re\alpha}$ is homogeneous of degree $d := \sum_{i=1}^n \Re\alpha_i-(p-1)\Re\alpha$,  and since $|\nabla F|>0$  on $\mathbb S^{n-1}$,  it is locally integrable near 0  iff $d>-n$. On the other hand, one has $\liminf_{|x|\to\infty}{F(x)}{|x|^p}>0$ while $|\nabla F(x)|^{-\alpha}$ stays bounded as \(|x|\to\infty\),  so the factor \(e^{-F(x)}\) gives exponential decay, which dominates any polynomial growth or singularity coming from $\prod_i |x_i|^{\alpha_i} |\nabla F(x)|^{-\alpha}$.  Therefore, \(e^{-F}H\in L^1(\mathbb R^n)\).

\medskip
\noindent
\textit{Step 3: Decoupling of variables.}
It remains to evaluate
$$
I=\int_{\mathbb R^n} e^{-F(x)}H(x)\,dx = \int_{\mathbb R^n} e^{-\sum_{r=1}^n |a_r x_r|^p} \prod_{i=1}^n |x_i|^{\alpha_i} \cdot |\nabla F(x)|^{-\alpha}  \,dx .
$$
Note that the term $|\nabla F(x)|^{-\alpha}$ couples the variables $x_1,\ldots, x_n$. Without this term, the $n$-dim\-en\-sion\-al integral would factor into a product of $n$ one-dimensional integrals.
Since
\[
|\nabla F(x)|
=
p\Bigl(\sum_{r=1}^n a_r^{2p}|x_r|^{2p-2}\Bigr)^{1/2},
\]
the identity
$$
u^{-s}
=
\frac{1}{\Gamma(s)}
\int_0^\infty \theta^{\,s-1} e^{-u\theta}\,d\theta,
\qquad
\Re s>0, \quad   u>0,
$$
yields
\[
|\nabla F(x)|^{-\alpha}
=
\frac{p^{-\alpha}}{\Gamma(\alpha/2)}
\int_0^\infty
\theta^{\alpha/2-1}
\exp\!\left(-\theta\sum_{r=1}^n a_r^{2p}|x_r|^{2p-2}\right)\,d\theta,
\]
assuming  that $\Re \alpha>0$ and $x\neq 0$.
Substituting this into \(I\) and interchanging the \(x\)- and
\(\theta\)-integrals (which will be justified at the end of the proof), we get
\[
I=
\frac{p^{-\alpha}}{\Gamma(\alpha/2)}
\int_0^\infty
\theta^{\alpha/2-1}
\left(
\int_{\mathbb R^n}
\exp\!\left(
-\sum_{r=1}^n |a_r x_r|^p
-\theta\sum_{r=1}^n a_r^{2p}|x_r|^{2p-2}
\right)
\prod_{i=1}^n |x_i|^{\alpha_i}\,dx
\right)\,d\theta.
\]
The inner integral factors, so
\[
I=
\frac{p^{-\alpha}}{\Gamma(\alpha/2)}
\int_0^\infty
\theta^{\alpha/2-1}
\prod_{i=1}^n
\left(
\int_{\mathbb R}
|x_i|^{\alpha_i}e^{-a_i^p|x_i|^p-\theta a_i^{2p}|x_i|^{2p-2}}\,dx_i
\right)\,d\theta.
\]
After the change of variables \(y=a_i x_i\) in the \(i\)-th factor,
\[
I=
\frac{p^{-\alpha}\prod_{i=1}^n a_i^{-\alpha_i-1}}{\Gamma(\alpha/2)}
\int_0^\infty
\theta^{\alpha/2-1}
\prod_{i=1}^n
\left(
\int_{\mathbb R}
|y|^{\alpha_i}e^{-|y|^p-\theta a_i^2|y|^{2p-2}}\,dy
\right)\,d\theta.
\]
Combining this with \(p^{\alpha+1} A(1)= p^{\alpha+1} I/\Gamma(\mu)\) gives the stated
formula~\eqref{eq:key_integral}. Recall that the left-hand side of~\eqref{eq:key_integral} equals $p^{\alpha+1} A(1)$.

It remains to justify the interchange of the \(x\)- and \(\theta\)-integrals. To this end, set $J_i(\theta):=\int_{\mathbb R}|y|^{\Re\alpha_i}e^{-|y|^p-\theta a_i^2|y|^{2p-2}}\, dy$, $\theta>0$. Then \(J_i(\theta)=O(1)\) as \(\theta\downarrow 0\), while
\[
J_i(\theta)\le \int_{\mathbb R}|y|^{\Re\alpha_i}e^{-\theta a_i^2|y|^{2p-2}}\,dy
=O\!\left(\theta^{-(\Re\alpha_i+1)/(2p-2)}\right),
\qquad \theta\to\infty.
\]
Hence, $\theta^{\Re\alpha/2-1}\prod_{i=1}^n J_i(\theta)\in L^1(0,\infty)$ whenever $0<\Re\alpha<\frac{n+\sum_{i=1}^n \Re\alpha_i}{p-1}$.
Therefore, Tonelli's theorem applies, and the \(x\)- and \(\theta\)-integrals may be interchanged.
\end{proof}

\subsection{Mixed moments of curvature measures}
We are now ready to evaluate the integral of the function $|x_1|^{\lambda_1} \ldots |x_n|^{\lambda_n}$ with respect to the curvature measure $\Phi_{n-m}(\mathbb B_p^n(a),x)$. Taking $\lambda_1 = \ldots = \lambda_n = 0$, we shall obtain a formula for the intrinsic volume $V_{n-m}(\mathbb B_p^n(a))$ by evaluating the integral in Example~\ref{exam:intrinsic_vols_ell_p_as_big_integral}.
\begin{proposition}[Mixed moments of curvature measures]\label{prop:mixed_moments_curvature_measure}
For $1<p<\infty$ and $n\geq 2$ consider the weighted $\ell_p$-ball $\mathbb B_p^n(a)$ with weights $a_1,\ldots, a_n >0$. For $t>0$ and $\nu\in\mathbb C$ with $\Re \nu>-1$, recall the notation
\[
\mathcal F_p(t;\nu)=\int_{\mathbb R}|y|^\nu e^{-|y|^p-t|y|^{2p-2}}\,dy.
\]
Let $\lambda_1,\dots,\lambda_n\in\mathbb C$ and put $\Lambda:=\sum_{k=1}^n \lambda_k$.
Then, for all  $m\in\{1,\dots,n\}$,
\begin{equation}\label{eq:Cnm-moment-Fp}
\begin{aligned}
&\int_{\partial \mathbb B_p^n(a)}
\prod_{k=1}^n |x_k|^{\lambda_k}\,
d\Phi_{n-m}\bigl(\mathbb B_p^n(a),x\bigr)
\\[1mm]
&\quad=
\frac{
p(p-1)^{m-1}
}{
m\,\kappa_m\,
\Gamma\!\left(\dfrac{n+\Lambda+p-m}{p}\right)
\Gamma\!\left(\dfrac m2\right)
}
\Bigl(\prod_{k=1}^n a_k^{-\lambda_k-1}\Bigr)
\sum_{i=1}^n  \sum_{\substack{I\subseteq \{1,\ldots,n\}\setminus\{i\}\\ |I|=m-1}}
\\
&\qquad
\int_0^\infty \theta^{\frac m2-1}
a_i^2\,\mathcal F_p(\theta a_i^2;\lambda_i+2p-2)
\prod_{j\in I}
\Bigl(a_j^2\,\mathcal F_p(\theta a_j^2;\lambda_j+p-2)\Bigr)
\prod_{k\notin I\cup\{i\}}
\mathcal F_p(\theta a_k^2;\lambda_k)\,d\theta
\end{aligned}
\end{equation}
provided that  $\Re \Lambda>m-n-p$ and
\[
\begin{cases}
\Re\lambda_k>-1,& \text{for all }k,\qquad\text{if }m=1,\\[1mm]
\Re\lambda_k>\max\{-1,\,1-p\},& \text{for all }k,\qquad\text{if }2\le m\le n-1,\\[1mm]
\Re\lambda_k>1-p,& \text{for all }k,\qquad\text{if }m=n.
\end{cases}
\]
\end{proposition}

\begin{proof}
By the density formula for $\Phi_{n-m}\bigl(\mathbb B_p^n(a),\cdot\bigr)$ stated in Proposition~\ref{prop:curvature_measures_density},
\begin{align}
&\int_{\partial \mathbb B_p^n(a)}
\prod_{k=1}^n |x_k|^{\lambda_k}\,
d\Phi_{n-m}\bigl(\mathbb B_p^n(a),x\bigr)
\notag\\
&\quad=
\frac{(p-1)^{m-1}}{m\,\kappa_m}
\sum_{i=1}^n
\sum_{\substack{I\subseteq \{1,\ldots,n\}\setminus\{i\}\\ |I|=m-1}}
a_i^{2p}\prod_{j\in I} a_j^p
\int_{\partial \mathbb B_p^n(a)}
\frac{
|x_i|^{2p-2}
\prod_{j\in I}|x_j|^{p-2}
\prod_{k=1}^n |x_k|^{\lambda_k}
}{
\left(\sum_{r=1}^n a_r^{2p}|x_r|^{2p-2}\right)^{(m+1)/2}
}\,d\mathcal H^{n-1}(x).
\label{eq:proof-start}
\end{align}
Fix $i\in\{1,\ldots,n\}$ and $I\subseteq\{1,\ldots,n\}\setminus\{i\}$ with $|I|=m-1$.  Define
\begin{equation}\label{eq:beta_k_i_I}
\beta_k^{(i,I)}
:=\lambda_k+(2p-2)\ind_{\{k=i\}}+(p-2)\ind_{\{k\in I\}},
\qquad k=1,\dots,n.
\end{equation}
Then the  integral on the right of~\eqref{eq:proof-start} is
\[
\int_{\partial \mathbb B_p^n(a)}
\frac{\prod_{k=1}^n |x_k|^{\beta_k^{(i,I)}}}{\left(\sum_{r=1}^n a_r^{2p}|x_r|^{2p-2}\right)^{(m+1)/2}}
\,d\mathcal H^{n-1}(x).
\]
By the assumptions on $\lambda_1,\dots,\lambda_n$, one has $\Re \beta_k^{(i,I)}>-1$ for all $k=1,\dots,n$. Moreover,
$$
n+\sum_{k=1}^n \beta_k^{(i,I)}-m(p-1)=n+\left(\Lambda+(2p-2)+(m-1)(p-2)\right)-m(p-1) = n+\Lambda+p-m.
$$
It follows that $0<m<\frac{n+\sum_{k=1}^n \Re\beta_k^{(i,I)}}{p-1}$.
Hence, \eqref{eq:key_integral} applies with $\alpha=m$ and
$\alpha_k=\beta_k^{(i,I)}$, giving
\begin{align}
&\int_{\partial \mathbb B_p^n(a)}
\frac{\prod_{k=1}^n |x_k|^{\beta_k^{(i,I)}}}{\left(\sum_{r=1}^n a_r^{2p}|x_r|^{2p-2}\right)^{(m+1)/2}}
\,d\mathcal H^{n-1}(x)
\notag\\
&\quad=
\frac{
p\,\prod_{k=1}^n a_k^{-\beta_k^{(i,I)}-1}
}{
\Gamma\!\left(\dfrac{n+\sum_{k=1}^n \beta_k^{(i,I)}-m(p-1)}{p}\right)
\Gamma\!\left(\dfrac m2\right)
}
\int_0^\infty \theta^{\frac m2-1}
\prod_{k=1}^n
\mathcal F_p(\theta a_k^2;\beta_k^{(i,I)})\,d\theta.
\label{eq:proof-key}
\end{align}
As we have already seen, the argument of the first gamma-function equals $(n+\Lambda+p-m)/p$. By~\eqref{eq:beta_k_i_I},
\begin{align*}
& a_i^{2p}\prod_{j\in I} a_j^p
\prod_{k=1}^n a_k^{-\beta_k^{(i,I)}-1}
=
\Bigl(\prod_{k=1}^n a_k^{-\lambda_k-1}\Bigr)
a_i^2\prod_{j\in I} a_j^2,
\\
&\prod_{k=1}^n \mathcal F_p(\theta a_k^2;\beta_k^{(i,I)})
=
\mathcal F_p(\theta a_i^2;\lambda_i+2p-2)
\prod_{j\in I}\mathcal F_p(\theta a_j^2;\lambda_j+p-2)
\prod_{k\notin I\cup\{i\}}\mathcal F_p(\theta a_k^2;\lambda_k).
\end{align*}
Substituting \eqref{eq:proof-key} into \eqref{eq:proof-start}
we obtain exactly \eqref{eq:Cnm-moment-Fp}.
\end{proof}

\begin{proposition}[Mixed moments of the surface measure]\label{prop:mixed_moments_surface_area_measure}
For $1<p<\infty$ and $n\geq 2$ consider the weighted $\ell_p$-ball $\mathbb B_p^n(a)$ with weights $a_1,\ldots, a_n >0$.
Let $\lambda_1,\dots,\lambda_n\in\mathbb C$ satisfy $\Re \lambda_k>-1$ for all $k=1,\dots,n$. Then
\begin{equation}\label{eq:surface-moment-G}
\int_{\partial \mathbb B_p^n(a)}
\prod_{k=1}^n |x_k|^{\lambda_k}\,d\mathcal H^{n-1}(x)
=
\frac{
p\, \prod_{k=1}^n a_k^{-\lambda_k-1}
}{
2\,\Gamma\!\left(\dfrac{n+\Lambda+p-1}{p}\right)\sqrt\pi
}
\int_0^\infty
\frac{G_\lambda(0)-G_\lambda(\theta)}{\theta^{3/2}}\,d\theta,
\end{equation}
where $\Lambda=\sum_{k=1}^n \lambda_k$ and
$$
G_\lambda(\theta):=\prod_{k=1}^n \mathcal F_p(\theta a_k^2;\lambda_k), \qquad \theta \geq 0,
\qquad
G_\lambda(0)
=
\left(\frac{2}{p}\right)^n
\prod_{k=1}^n
\Gamma\!\left(\frac{\lambda_k+1}{p}\right).
$$
\end{proposition}

\begin{proof}
Recall that $\Phi_{n-1}(\mathbb B_p^n(a),\cdot)= \frac 12 \mathcal H^{n-1}(\cdot)\llcorner \partial \mathbb B_p^n(a)$. Taking $m=1$ in Proposition~\ref{prop:mixed_moments_curvature_measure} gives
\begin{multline}\label{eq:surface-proof-before-G}
\int_{\partial \mathbb B_p^n(a)}
\prod_{k=1}^n |x_k|^{\lambda_k}\,d\mathcal H^{n-1}(x)
\\
=
\frac{
p\, \prod_{k=1}^n a_k^{-\lambda_k-1}
}{
\Gamma\!\left(\dfrac{n+\Lambda+p-1}{p}\right)\sqrt\pi
}
\int_0^\infty
\theta^{-1/2}
\sum_{i=1}^n
a_i^2\,
\mathcal F_p(\theta a_i^2;\lambda_i+2p-2)
\prod_{k\ne i}\mathcal F_p(\theta a_k^2;\lambda_k)\,d\theta.
\end{multline}
Next, for $\Re \nu>-1$, differentiation under the integral sign gives $\frac{\partial}{\partial t}\mathcal F_p(t;\nu)
=
-\mathcal F_p(t;\nu+2p-2)$, $t\ge 0$.
Hence, by the product rule,
\[
G_\lambda'(\theta)
=
-\sum_{i=1}^n
a_i^2\,
\mathcal F_p(\theta a_i^2;\lambda_i+2p-2)
\prod_{k\ne i}\mathcal F_p(\theta a_k^2;\lambda_k).
\]
Substituting this into \eqref{eq:surface-proof-before-G}, we get
\begin{equation}\label{eq:surface-proof-derivative}
\int_{\partial \mathbb B_p^n(a)}
\prod_{k=1}^n |x_k|^{\lambda_k}\,d\mathcal H^{n-1}(x)
=
\frac{
p \prod_{k=1}^n a_k^{-\lambda_k-1}
}{
\Gamma\!\left(\dfrac{n+\Lambda+p-1}{p}\right)\sqrt\pi
}
\int_0^\infty \theta^{-1/2}\bigl(-G_\lambda'(\theta)\bigr)\,d\theta.
\end{equation}
Writing $-G_\lambda'(\theta) = (G_\lambda(0) - G_\lambda(\theta))'$, integrating by parts and observing that both boundary terms vanish gives~\eqref{eq:surface-moment-G}.
\end{proof}

\subsection{Proof of the exact formula for intrinsic volumes}
We are ready to prove Theorems~\ref{theo:V_j_ell_p_ball_exact} and~\ref{theo:V_j_ell_p_ellipsoid_exact}. Theorem~\ref{theo:V_j_ell_p_ellipsoid_exact} is a special case of  Proposition~\ref{prop:mixed_moments_curvature_measure} with $m=n-j$ and $\lambda_1 = \ldots = \lambda_n  = 0$. Theorem~\ref{theo:V_j_ell_p_ball_exact} then follows by further setting $a_1=\ldots=a_n=1$.

\section{Proofs: Asymptotic results}

\subsection{Setup}
In this section we prove Theorem~\ref{theo:maxwell_curvature} and, as a byproduct, Theorem~\ref{theo:exact_asympt_V_j}.  Fix $1<p<\infty$.  For sufficiently large $n$, let  $j(n)\in \{0,\ldots, n-1\}$.  Our main objective is to evaluate the large $n$ asymptotics of the integrals
\begin{equation}\label{eq:I_n_j_n_def_intro}
I_{n,j(n)}(\lambda_1,\dots,\lambda_r)
:=
\int_{\partial \mathbb B_p^n}
\prod_{j=1}^r |x_j|^{\lambda_j}\,
d\Phi_{j(n)}\bigl(\mathbb B_p^n,x\bigr),
\end{equation}
where $r\in\N$ is fixed,  $\lambda_1,\dots,\lambda_r\in\mathbb C$  are complex parameters (with exact ranges to be specified below) and $\Phi_{j(n)}(\mathbb B_p^n,\cdot)$ is the $j(n)$-th curvature measure on $\partial \mathbb B_p^n$.  Putting $\lambda_1=\ldots = \lambda_r = 0$, we recover the intrinsic volumes,
\begin{equation}\label{eq:I_n_j_n_def_all_lambdas_0_intro}
I_{n,j(n)}(0,\ldots, 0) = \int_{\partial \mathbb B_p^n} d\Phi_{j(n)}(\mathbb B_p^n,x) = V_{j(n)}(\mathbb B_p^n),
\end{equation}
which will allow us to prove Theorem~\ref{theo:exact_asympt_V_j}.
Furthermore, if $X_n=(X_{1;n},\ldots, X_{n;n})$ is a random point on $\partial \mathbb B_p^n$ distributed according to the normalized curvature measure $\Phi_{j(n)}(\mathbb B_p^n, \cdot)/ V_{j(n)}(\mathbb B_p^n)$, as in Theorem~\ref{theo:maxwell_curvature}, then the mixed moments (or the multivariate Mellin transform) of the random vector $(|X_{1;n}|,\ldots, |X_{r;n}|)$ are given by
\begin{equation}\label{eq:mixed_moments_as_quotient_I_intro}
\E \left[|X_{1;n}|^{\lambda_1} \ldots |X_{r;n}|^{\lambda_r}\right]
=
\frac{I_{n,j(n)}(\lambda_1,\dots,\lambda_r)}{I_{n,j(n)}(0,\ldots, 0)}.
\end{equation}
In the following, we shall evaluate the asymptotic of the right-hand side, which will allow us to  deduce weak convergence of  a suitably rescaled vector $(|X_{1;n}|,\ldots, |X_{r;n}|)$.

To evaluate the integral in~\eqref{eq:I_n_j_n_def_intro}, we apply Proposition~\ref{prop:mixed_moments_curvature_measure} with $a_1=\cdots=a_n=1$ and $m= n-j(n)$.  We also assume $n>r$, put $\lambda_{r+1} = \ldots= \lambda_n = 0$  and  $\Lambda:=\lambda_1+\cdots+\lambda_r$. Assuming that
\begin{equation}\label{eq:I_n_lambda_1_ldots_lambda_n_exact_formula_conditions_lambdas}
\lambda_1,\dots,\lambda_r\in\mathbb C,
\qquad
\Re \lambda_j>\max\{-1,1-p\},\quad j=1,\dots,r,
\qquad
\Re \Lambda>-j(n)-p,
\end{equation}
Proposition~\ref{prop:mixed_moments_curvature_measure} gives the identity
\begin{equation}\label{eq:I_n_lambda_1_ldots_lambda_n_exact_formula_setup}
I_{n,j(n)}(\lambda_1,\dots,\lambda_r)
=
\frac{p(p-1)^{n-j(n)-1}}{2\pi^{(n-j(n))/2}\,
\Gamma\!\left(\frac{j(n)+\Lambda+p}{p}\right)}
\int_0^\infty \theta^{\frac {n-j(n)}2-1} S_n(\theta)\,d\theta,
\end{equation}
where
\begin{align}
S_n(\theta)
&:=
\sum_{i=1}^n
\sum_{\substack{I\subseteq \{1,\ldots,n\}\setminus\{i\}\\ |I|=n-j(n)-1}}
\mathcal F_p(\theta;\lambda_i+2p-2)
\prod_{j\in I}\mathcal F_p(\theta;\lambda_j+p-2)
\prod_{k\notin I\cup\{i\}}\mathcal F_p(\theta;\lambda_k),
\label{eq:S_n_def_setup_subsec}
\\
\mathcal F_p(\theta;\nu)
&:=
\int_{\mathbb R}|y|^\nu e^{-|y|^p-\theta |y|^{2p-2}}\,dy,
\qquad \theta >0, \qquad \Re \nu>-1. \label{eq:mathcal F_def_setup_subsec}
\end{align}
We used the formula  $m\kappa_m\Gamma (\frac m2)=2\pi^{m/2}$ with $m= n-j(n)$  to simplify the denominator in~\eqref{eq:I_n_lambda_1_ldots_lambda_n_exact_formula_setup}. Note that $S_n(\theta)=S_{n,j(n)}(\theta)$ also depends on $j(n)$, but we supress this in the notation.

\subsection{The bulk regime} In this section, we prove an asymptotic formula for $I_{n,j(n)}(\lambda_1,\dots,\lambda_r)$  in the regime where $j(n)/n \to \alpha \in (0,1)$. As we shall see, the main contribution to the integral $\int_0^\infty \theta^{(n-j(n))/2-1} S_n(\theta)\,d\theta$ comes from a term of the form $e^{n \Psi_{p,\alpha(n)}(\theta)}$, where $\alpha(n)= j(n)/n$ and $\Psi_{p,\beta}$ is a function defined by formula~\eqref{eq:Phi_alpha_def_prop_maximizer_theo_asympt_V_j} in Proposition~\ref{prop:maximizer_Phi_p_beta}, which we are going to prove now.

\begin{proof}[Proof of Proposition~\ref{prop:maximizer_Phi_p_beta}]
It will be convenient to use the notation $\mathcal L_p(\theta):=-\mathcal{J}_p'(\theta)=\mathcal F_p(\theta;3p-4)$, where $\theta>0$. It is clear that the quantities $\mathcal I_p(\theta),\mathcal J_p(\theta),\mathcal K_p(\theta),\mathcal L_p(\theta)$ are finite and strictly positive for every fixed $\theta>0$.

\medskip

\noindent\emph{Step 1: An  identity between $\mathcal J_p, \mathcal K_p, \mathcal L_p$.}
Integrating the identity
$$
\frac{d}{dy}\Bigl(y^{p-1}e^{-y^p-\theta y^{2p-2}}\Bigr) = \Bigl((p-1)y^{p-2}-p y^{2p-2}-(2p-2)\theta y^{3p-4}\Bigr)
e^{-y^p-\theta y^{2p-2}}.
$$
over $(0,\infty)$ yields
$$
0
=
\int_0^\infty
\Bigl((p-1)y^{p-2}-p y^{2p-2}-(2p-2)\theta y^{3p-4}\Bigr)
e^{-y^p-\theta y^{2p-2}}\,dy.
$$
We may replace $\int_0^\infty$ by $\int_\R$ and $y$ by $|y|$. This gives
\begin{equation}\label{eq:identity_JKL}
(p-1)\mathcal J_p(\theta)
=
p\,\mathcal K_p(\theta)
+
2(p-1)\theta\,\mathcal L_p(\theta).
\end{equation}

\medskip

\noindent\emph{Step 2: Formula for $\Psi_{p,\beta}'(\theta)$.}
Differentiating the definition of $\Psi_{p,\beta}$, we get
\[
\Psi_{p,\beta}'(\theta)
=
\frac{1-\beta}{2\theta}
+\beta\frac{\mathcal I_p'(\theta)}{\mathcal I_p(\theta)}
+(1-\beta)\frac{\mathcal J_p'(\theta)}{\mathcal J_p(\theta)}
=
\frac{1-\beta}{2\theta}
-\beta\frac{\mathcal K_p(\theta)}{\mathcal I_p(\theta)}
-(1-\beta)\frac{\mathcal L_p(\theta)}{\mathcal J_p(\theta)},
\]
using $\mathcal I_p'(\theta)=-\mathcal K_p(\theta)$ and  $\mathcal J_p'(\theta)=-\mathcal L_p(\theta)$.
Now solve~\eqref{eq:identity_JKL} for $\mathcal L_p(\theta)$:
\[
\mathcal L_p(\theta)
=
\frac{\mathcal J_p(\theta)}{2\theta}
-\frac{p}{2(p-1)\theta}\,\mathcal K_p(\theta).
\]
Substituting this into the formula for $\Psi_{p,\beta}'(\theta)$, the terms involving
$(1-\beta)/(2\theta)$ cancel, and we obtain
\begin{equation}\label{eq:Phi_alpha_der_aux}
\Psi_{p,\beta}'(\theta)
=
-\beta\frac{\mathcal K_p(\theta)}{\mathcal I_p(\theta)}
+\frac{(1-\beta)p}{2(p-1)\theta}\,
\frac{\mathcal K_p(\theta)}{\mathcal J_p(\theta)}
=
\mathcal K_p(\theta)
\left(
\frac{(1-\beta)p}{2(p-1)\theta\,\mathcal J_p(\theta)}
-\frac{\beta}{\mathcal I_p(\theta)}
\right).
\end{equation}
Since $\mathcal K_p(\theta)>0$, the critical-point equation
$\Psi_{p,\beta}'(\theta)=0$ is equivalent to
\begin{equation}\label{eq:crit_point_eq_Phi_alpha_equiv_to}
\theta\,\frac{\mathcal J_p(\theta)}{\mathcal I_p(\theta)}
=
\frac{(1-\beta)p}{2(p-1)\beta}.
\end{equation}

\medskip

\noindent\emph{Step 3: Monotonicity of an auxiliary function.}
Define
\[
g(\theta):=\theta\,\frac{\mathcal J_p(\theta)}{\mathcal I_p(\theta)},
\qquad \theta>0.
\]
We claim that $g$ is strictly increasing on $(0,\infty)$. Indeed,
\[
g'(\theta)
=
\frac{\mathcal I_p(\theta)\mathcal J_p(\theta)
+\theta \mathcal I_p(\theta)\mathcal J_p'(\theta)
-\theta \mathcal J_p(\theta)\mathcal I_p'(\theta)}
{\mathcal I_p(\theta)^2} = \frac{
\mathcal I_p(\theta)\mathcal J_p(\theta)
-\theta \mathcal I_p(\theta)\mathcal L_p(\theta)
+\theta \mathcal J_p(\theta)\mathcal K_p(\theta)}
{\mathcal I_p(\theta)^2},
\]
again using $\mathcal I_p'=-\mathcal K_p$ and $\mathcal J_p'=-\mathcal L_p$.
From~\eqref{eq:identity_JKL},
\[
\theta\mathcal L_p(\theta)
=
\frac{\mathcal J_p(\theta)}{2}
-\frac{p\mathcal K_p(\theta)}{2(p-1)}.
\]
Therefore, we have
\[
g'(\theta)
=
\frac{1}{\mathcal I_p(\theta)^2}
\left(
\mathcal I_p(\theta)
\left(
\frac{\mathcal J_p(\theta)}{2}
+\frac{p\mathcal K_p(\theta)}{2(p-1)}
\right)
+
\theta \mathcal J_p(\theta)\mathcal K_p(\theta)
\right).
\]
Every term on the right-hand side is strictly positive, hence $g'(\theta)>0$ for all $\theta>0$. In particular, $g$ is strictly increasing on $(0,\infty)$.

\medskip

\noindent\emph{Step 4: The range of $g$.}
As $\theta\downarrow 0$, dominated convergence gives $\mathcal I_p(\theta)\to \mathcal I_p(0)\in(0,\infty)$ and  $\mathcal J_p(\theta)\to \mathcal J_p(0)\in(0,\infty)$. Thus,
\[
g(\theta)=\theta\,\frac{\mathcal J_p(\theta)}{\mathcal I_p(\theta)}\to 0, \qquad \theta\downarrow 0.
\]
To analyze the behavior of $g(\theta)$ as  $\theta\to\infty$, make the change of variables $y=\theta^{-\frac{1}{2p-2}}z$.
Then
\[
\mathcal I_p(\theta)
=
\theta^{-\frac 1{2p-2}}
\int_{\mathbb R}
e^{-\theta^{-\frac{p}{2p-2}}|z|^p-|z|^{2p-2}}\,dz,
\qquad
\mathcal J_p(\theta)
=
\theta^{-\frac 12}
\int_{\mathbb R}
|z|^{p-2}
e^{-\theta^{-\frac p{2p-2}}|z|^p-|z|^{2p-2}}\,dz.
\]
Hence, we obtain
\[
g(\theta)
=
\theta\,\frac{\mathcal J_p(\theta)}{\mathcal I_p(\theta)}
=
\theta^{\frac p{2p-2}}
\frac{
\int_{\mathbb R}|z|^{p-2}e^{-\theta^{-\frac p{2p-2}}|z|^p-|z|^{2p-2}}\,dz
}{
\int_{\mathbb R}e^{-\theta^{-\frac p{2p-2}}|z|^p-|z|^{2p-2}}\,dz
}.
\]
Note that $2p-2 >0$ since $p>1$. By dominated convergence, the quotient converges to the positive finite constant
\[
\frac{
\int_{\mathbb R}|z|^{p-2}e^{-|z|^{2p-2}}\,dz
}{
\int_{\mathbb R}e^{-|z|^{2p-2}}\,dz
},
\]
and therefore $g(\theta)\to+\infty$ as $\theta\to\infty$.

\medskip

\noindent\emph{Step 5: The unique global maximizer.}
We have shown that the function  $g:(0,\infty)\to(0,\infty)$ is continuous and strictly increasing, with $\lim_{\theta\downarrow 0}g(\theta)=0$ and $\lim_{\theta\to\infty}g(\theta)=+\infty$.
Thus, $g$ is a bijection from $(0,\infty)$ onto $(0,\infty)$. Consequently,
equation~\eqref{eq:crit_point_eq_Phi_alpha_equiv_to} has a unique solution $\theta=\theta_{p,\beta}\in(0,\infty)$. In particular,
\begin{equation}\label{eq:g_at_theta_alpha}
g(\theta_{p,\beta}) = \theta_{p,\beta}\,\frac{\mathcal J_p(\theta_{p,\beta})}{\mathcal I_p(\theta_{p,\beta})}
=
\frac{(1-\beta)p}{2(p-1)\beta}.
\end{equation}
From~\eqref{eq:Phi_alpha_der_aux} and using that $g$ is strictly increasing, it follows that
$$
\Psi_{p,\beta}'(\theta)>0 \quad \text{ for }0<\theta<\theta_{p,\beta},
\qquad
\text{ and }
\qquad
\Psi_{p,\beta}'(\theta)<0 \quad \text{ for }\theta>\theta_{p,\beta}.
$$
Therefore, $\theta_{p,\beta}$ is the unique global maximizer of $\Psi_{p,\beta}$.

\medskip

\noindent\emph{Step 6: Strict negativity of the second derivative at the maximizer.}
Rewrite~\eqref{eq:Phi_alpha_der_aux} as
\[
\Psi_{p,\beta}'(\theta)
=
\beta\,
\frac{\mathcal K_p(\theta)}{\mathcal I_p(\theta)\,g(\theta)}
\left(
\frac{(1-\beta)p}{2(p-1)\beta}-g(\theta)
\right).
\]
At $\theta=\theta_{p,\beta}$, the factor in parentheses vanishes by~\eqref{eq:g_at_theta_alpha}. Differentiating at
$\theta=\theta_{p,\beta}$ therefore gives
\[
\Psi_{p,\beta}''(\theta_{p,\beta})
=
-\beta\,
\frac{\mathcal K_p(\theta_{p,\beta})}{\mathcal I_p(\theta_{p,\beta})\,g(\theta_{p,\beta})}
\,g'(\theta_{p,\beta}).
\]
Since $\beta>0$, $\mathcal K_p(\theta_{p,\beta})>0$, $\mathcal I_p(\theta_{p,\beta})>0$,
$g(\theta_{p,\beta})>0$, and $g'(\theta_{p,\beta})>0$, we conclude that $\Psi_{p,\beta}''(\theta_{p,\beta})<0$.

The proof of Proposition~\ref{prop:maximizer_Phi_p_beta} is complete.
\end{proof}

\begin{proposition}[Laplace asymptotics in the bulk regime]
\label{prop:asymptotics_laplace_central}
Fix $1<p<\infty$, an integer $r\ge 1$, and complex parameters
\begin{equation}\label{eq:lambdas_assumptions_central_asymptotics}
\lambda_1,\dots,\lambda_r\in\mathbb C,
\qquad
\Re \lambda_j>\max\{-1,1-p\},\quad j=1,\dots,r.
\end{equation}
Put $\Lambda:=\lambda_1+\cdots+\lambda_r$.
For every integer $n\geq 2$ let  $j(n)\in \{0,\ldots, n-1\}$ and suppose that
\[
\alpha(n) := \frac{j(n)}{n}\to \alpha\in(0,1).
\]
For $\beta\in(0,1)$, let $\Psi_{p,\beta}: (0,\infty)\to\R$ and its unique global maximizer $\theta_{p,\beta}$  be defined as in Proposition~\ref{prop:maximizer_Phi_p_beta}. Then $\theta_{p,\alpha(n)}\to \theta_{p,\alpha}$ as $n\to\infty$, and the following assertions hold:

\vspace{2mm}
\noindent
(i) Absolute asymptotic: As $n\to\infty$, one has
\[
\begin{aligned}
I_{n,j(n)}&(\lambda_1,\dots,\lambda_r)
\sim
\frac{p(p-1)^{n-j(n)-1}\,n\binom{n-1}{j(n)}}{2\pi^{(n-j(n))/2}\,
\Gamma\!\left(\frac{j(n)+\Lambda+p}{p}\right)}
\,
\sqrt{\frac{2\pi}{n\,|\Psi_{p,\alpha(n)}''(\theta_{p,\alpha(n)})|}}
\,
\frac{\mathcal{K}_p(\theta_{p,\alpha(n)})}
{\theta_{p,\alpha(n)}\,\mathcal{J}_p(\theta_{p,\alpha(n)})}
\\
&\times
\,e^{\,n\Psi_{p,\alpha(n)}(\theta_{p,\alpha(n)})}
\prod_{j=1}^r
\left(
\alpha(n)\frac{\mathcal F_p(\theta_{p,\alpha(n)};\lambda_j)}{\mathcal F_p(\theta_{p,\alpha(n)};0)}
+
(1-\alpha(n)) \frac{\mathcal F_p(\theta_{p,\alpha(n)};\lambda_j+p-2)}{\mathcal F_p(\theta_{p,\alpha(n)};p-2)}
\right).
\end{aligned}
\]
In particular, for $\lambda_1= \ldots = \lambda_r = 0$ this formula takes the form
\begin{equation}\label{eq:intrinsic_ell_p_ball_central_exact_repeat}
I_{n,j(n)}(0,\dots,0)
\sim
\frac{p(p-1)^{n-j(n)-1}\,n\binom{n-1}{j(n)}}{2\pi^{(n-j(n))/2}\,
\Gamma\!\left(\frac{j(n)+p}{p}\right)}
\,
\frac{\mathcal K_p(\theta_{p,\alpha(n)})}{\theta_{p,\alpha(n)}\,\mathcal J_p(\theta_{p,\alpha(n)})}
\,
e^{\,n\Psi_{p,\alpha(n)}(\theta_{p,\alpha(n)})}
\sqrt{\frac{2\pi}{n\,|\Psi_{p,\alpha(n)}''(\theta_{p,\alpha(n)})|}}.
\end{equation}

\vspace{2mm}
\noindent
(ii) Relative asymptotic: As $n\to\infty$, one has
\[
\frac{I_{n,j(n)}(\lambda_1,\dots,\lambda_r)}{I_{n,j(n)}(0,\ldots, 0)}
\sim
\left(\frac{p}{\alpha \, n}\right)^{\Lambda/p}
\prod_{j=1}^r
\left(
\alpha\frac{\mathcal F_p(\theta_{p,\alpha};\lambda_j)}{\mathcal F_p(\theta_{p,\alpha};0)}
+
(1-\alpha)\frac{\mathcal F_p(\theta_{p,\alpha};\lambda_j+p-2)}{\mathcal F_p(\theta_{p,\alpha};p-2)}
\right).
\]
\end{proposition}
\begin{remark}
In view of $V_{j(n)}(\mathbb B_p^n) = I_{n,j(n)}(0,\dots,0)$, formula~\eqref{eq:intrinsic_ell_p_ball_central_exact_repeat} yields Part~(i) of Theorem~\ref{theo:exact_asympt_V_j}.
\end{remark}
\begin{proof}[Proof of Proposition~\ref{prop:asymptotics_laplace_central}]
Put   $m = m(n) := n-j(n)$. Let $\lambda_1,\ldots,\lambda_r$ satisfy~\eqref{eq:lambdas_assumptions_central_asymptotics} and set $\lambda_{r+1}=\cdots=\lambda_n=0$.

\medskip

\noindent
\textit{Step 1: Exact representation.} Recall from Equation~\eqref{eq:I_n_lambda_1_ldots_lambda_n_exact_formula_setup} that, for sufficiently large $n$, \begin{equation}\label{eq:I_n_lambda_1_ldots_lambda_n_exact_formula}
I_{n,j(n)}(\lambda_1,\dots,\lambda_r)
=
\frac{p(p-1)^{m-1}}{2\pi^{m/2}\,
\Gamma\!\left(\frac{n+\Lambda+p-m}{p}\right)}
\int_0^\infty \theta^{\frac m2-1} S_n(\theta)\,d\theta.
\end{equation}
where
\[
S_n(\theta)
:=
\sum_{i=1}^n
\sum_{\substack{I\subseteq \{1,\ldots,n\}\setminus\{i\}\\ |I|=m-1}}
\mathcal F_p(\theta;\lambda_i+2p-2)
\prod_{j\in I}\mathcal F_p(\theta;\lambda_j+p-2)
\prod_{k\notin I\cup\{i\}}\mathcal F_p(\theta;\lambda_k).
\]

\medskip

\noindent
\textit{Step 2: Decomposition of $S_n(\theta)$.}
We split
\[
S_n(\theta)=S_{1,n}(\theta)+S_{2,n}(\theta),
\]
where $S_{1,n}(\theta)$ is the contribution of $i\in\{1,\dots,r\}$, and $S_{2,n}(\theta)$ the contribution of $i\in\{r+1,\dots,n\}$.

\medskip

\noindent
\textit{Step 2A: Formula for $S_{1,n}(\theta)$.}
For $i\in\{1,\dots,r\}$ and $I\subseteq\{1,\ldots,n\}\setminus\{i\}$ with $|I|=m-1$, we write
\[
J:=I\cap(\{1,\ldots,r\}\setminus\{i\})
\qquad\text{and}\qquad
L:=I\cap\{r+1,\dots,n\}.
\]
Then $|J|=\ell$ and $|L|=m-1-\ell$ for some $\ell\in\{0,\dots,\min(r-1,m-1)\}$. Once $\ell$ is fixed, the set $L$ may be chosen in $\binom{n-r}{m-1-\ell}$
ways. Since $\lambda_{r+1}=\cdots=\lambda_n=0$, every index in $L$ contributes a factor $\mathcal F_p(\theta;p-2)$, and every index in $\{r+1,\dots,n\}\setminus L$ contributes a factor $\mathcal F_p(\theta;0)$. Hence,
\[
\begin{aligned}
S_{1,n}(\theta)
&=
\sum_{i=1}^r
\mathcal F_p(\theta;\lambda_i+2p-2)
\sum_{\ell=0}^{\min(r-1,m-1)}
\binom{n-r}{m-1-\ell}\,
\mathcal F_p(\theta;p-2)^{\,m-1-\ell}\,
\mathcal F_p(\theta;0)^{\,n-r-m+1+\ell}
\\
&\qquad\times
\sum_{\substack{J\subseteq \{1,\ldots,r\}\setminus\{i\}\\ |J|=\ell}}
\prod_{j\in J}\mathcal F_p(\theta;\lambda_j+p-2)
\prod_{k\in \{1,\ldots,r\}\setminus(J\cup\{i\})}\mathcal F_p(\theta;\lambda_k).
\end{aligned}
\]

\noindent
\textit{Step 2B: Formula for $S_{2,n}(\theta)$.}
For $i\in\{r+1,\dots,n\}$ and $I\subseteq\{1,\ldots,n\}\setminus\{i\}$ with $|I|=m-1$ we write
\[
J:=I\cap \{1,\ldots,r\},
\qquad
L:=I\cap\bigl(\{r+1,\dots,n\}\setminus\{i\}\bigr).
\]
Then $|J|=\ell$ and  $|L|=m-1-\ell$ for some $\ell\in\{0,\dots,\min(r,m-1)\}$. Once $\ell$ is fixed, the set $L$ may be chosen in $\binom{n-r-1}{m-1-\ell}$
ways. Note that $\lambda_i = 0$ and there are $n-r$ choices for $i$.   Therefore,
\[
\begin{aligned}
S_{2,n}(\theta)
&=
(n-r)\,\mathcal F_p(\theta;2p-2)
\sum_{\ell=0}^{\min(r,m-1)}
\binom{n-r-1}{m-1-\ell}\,
\mathcal F_p(\theta;p-2)^{\,m-1-\ell}\,
\mathcal F_p(\theta;0)^{\,n-r-m+\ell}
\\
&\qquad\times
\sum_{\substack{J\subseteq \{1,\ldots,r\}\\ |J|=\ell}}
\prod_{j\in J}\mathcal F_p(\theta;\lambda_j+p-2)
\prod_{k\in \{1,\ldots,r\}\setminus J}\mathcal F_p(\theta;\lambda_k).
\end{aligned}
\]

\medskip

\noindent
\textit{Step 3: Factoring out the main contribution.}
For $\theta >0$ define
$$
\Xi(\theta):=\frac{\mathcal F_p(\theta;2p-2)}{\theta\,\mathcal F_p(\theta;p-2)}=\frac{\mathcal K_p(\theta)}{\theta\,\mathcal J_p(\theta)},
$$
and observe that
$$
\Xi(\theta)
\,e^{\,n\Psi_{p,\alpha(n)}(\theta)}
=
\theta^{m/2-1}
\mathcal K_p(\theta)
\mathcal J_p(\theta)^{\,m-1}
\mathcal I_p(\theta)^{\,n-m}.
$$
In the special case when $\lambda_1= \ldots = \lambda_r = 0$ (and $\lambda_{r+1} = \ldots = \lambda_n =0$ as always in this proof) Proposition~\ref{prop:mixed_moments_curvature_measure} with $a_1=\cdots=a_n=1$ gives
\[
I_{n,j(n)}(0, \ldots, 0)
=
V_{n-m}(\mathbb B_p^n)
=
\frac{p(p-1)^{m-1}\,n\binom{n-1}{m-1}}{2\pi^{m/2}\,
\Gamma\!\left(\frac{n+p-m}{p}\right)}
\int_0^\infty \Xi(\theta)e^{\,n\Psi_{p,\alpha(n)}(\theta)}\,d\theta.
\]
For arbitrary  $\lambda_1,\ldots,\lambda_r$, an additional correction term appears in this representation. To describe it, we define $H_{1,n}(\theta)$ and $H_{2,n}(\theta)$ by
\begin{align*}
\theta^{m/2-1}S_{1,n}(\theta)
&=
n\binom{n-1}{m-1}\,\Xi(\theta)\,e^{\,n\Psi_{p,\alpha(n)}(\theta)}\,H_{1,n}(\theta),
\\
\theta^{m/2-1}S_{2,n}(\theta)
&=
n\binom{n-1}{m-1}\,\Xi(\theta)\,e^{\,n\Psi_{p,\alpha(n)}(\theta)}\,H_{2,n}(\theta).
\end{align*}
A direct comparison with the formulas for $S_{1,n}(\theta)$ and $S_{2,n}(\theta)$ given above yields
\[
\begin{aligned}
H_{1,n}(\theta)
&=
\frac1n
\sum_{i=1}^r
\frac{\mathcal F_p(\theta;\lambda_i+2p-2)}{\mathcal F_p(\theta;2p-2)}
\sum_{\ell=0}^{\min(r-1,m-1)}
\frac{\binom{n-r}{m-1-\ell}}{\binom{n-1}{m-1}}
\\
&\qquad\times
\sum_{\substack{J\subseteq \{1,\ldots,r\}\setminus\{i\}\\ |J|=\ell}}
\prod_{j\in J}
\frac{\mathcal F_p(\theta;\lambda_j+p-2)}{\mathcal F_p(\theta;p-2)}
\prod_{k\in \{1,\ldots,r\}\setminus(J\cup\{i\})}
\frac{\mathcal F_p(\theta;\lambda_k)}{\mathcal F_p(\theta;0)},
\end{aligned}
\]
\[
\begin{aligned}
H_{2,n}(\theta)
&=
\frac{n-r}{n}
\sum_{\ell=0}^{\min(r,m-1)}
\frac{\binom{n-r-1}{m-1-\ell}}{\binom{n-1}{m-1}}
\sum_{\substack{J\subseteq \{1,\ldots,r\}\\ |J|=\ell}}
\prod_{j\in J}
\frac{\mathcal F_p(\theta;\lambda_j+p-2)}{\mathcal F_p(\theta;p-2)}
\prod_{k\in \{1,\ldots,r\}\setminus J}
\frac{\mathcal F_p(\theta;\lambda_k)}{\mathcal F_p(\theta;0)}.
\end{aligned}
\]
With this notation, \eqref{eq:I_n_lambda_1_ldots_lambda_n_exact_formula} becomes
\begin{equation}\label{eq:I_n_lambda_1_ldots_lambda_n_exact_formula_decomposed}
\begin{aligned}
I_{n,j(n)}(\lambda_1,\dots,\lambda_r)
&=
\frac{p(p-1)^{m-1} \, n\binom{n-1}{m-1}}{2\pi^{m/2}\,
\Gamma\!\left(\frac{n+\Lambda+p-m}{p}\right)}
\int_0^\infty
\Xi(\theta)\,e^{\,n\Psi_{p,\alpha(n)}(\theta)}
\bigl(H_{1,n}(\theta)+H_{2,n}(\theta)\bigr)\,d\theta.
\end{aligned}
\end{equation}

\medskip

\noindent
\textit{Step 4: Asymptotics of $H_{1,n}$ and $H_{2,n}$.}
For $\beta\in(0,1)$, define
\[
H_\beta(\theta; \lambda_1,\ldots, \lambda_r)
:=
\prod_{j=1}^r
\left(
\beta\frac{\mathcal F_p(\theta;\lambda_j)}{\mathcal F_p(\theta;0)}
+
(1-\beta)\frac{\mathcal F_p(\theta;\lambda_j+p-2)}{\mathcal F_p(\theta;p-2)}
\right).
\]
We claim that, locally uniformly in $\theta\in(0,\infty)$,
\begin{equation}\label{eq:step_4_asymptotics_H_1_n_H_2_n}
H_{1,n}(\theta)+H_{2,n}(\theta)
=
H_{\alpha(n)}(\theta; \lambda_1,\dots, \lambda_r)+O (n^{-1}), \qquad n\to\infty.
\end{equation}

\medskip
\noindent
\textit{Step 4A: Asymptotics of $H_{2,n}$.}
Recall that $r$ is fixed and that $m/n  = 1- \alpha(n) \to 1-\alpha \in (0,1)$.   For every  $\ell\in\{0,\dots,r\}$,
\[
\frac{n-r}{n}\,
\frac{\binom{n-r-1}{m-1-\ell}}{\binom{n-1}{m-1}}
=
\alpha(n)^{r-\ell}(1-\alpha(n))^\ell+O (n^{-1}),
\qquad n\to\infty.
\]
Therefore,
\[
\begin{aligned}
H_{2,n}(\theta)
&=
\sum_{\ell=0}^r
\Bigl(\alpha(n)^{r-\ell}(1-\alpha(n))^\ell+O(n^{-1})\Bigr)
\sum_{\substack{J\subseteq \{1,\ldots,r\}\\ |J|=\ell}}
\prod_{j\in J}
\frac{\mathcal F_p(\theta;\lambda_j+p-2)}{\mathcal F_p(\theta;p-2)}
\prod_{k\in \{1,\ldots,r\}\setminus J}
\frac{\mathcal F_p(\theta;\lambda_k)}{\mathcal F_p(\theta;0)}.
\end{aligned}
\]
Since the number of subsets $J\subseteq \{1,\ldots,r\}$ is finite and all $\mathcal F_p$-ratios are continuous, this gives
\[
H_{2,n}(\theta)=H_{\alpha(n)}(\theta; \lambda_1,\ldots, \lambda_r)+O(n^{-1}), \qquad n\to\infty,
\]
locally uniformly on $(0,\infty)$.

\medskip
\noindent
\textit{Step 4B: Asymptotics of $H_{1,n}$.}
On the other hand, $H_{1,n}$ contains an explicit prefactor $1/n$. Since $r$ is fixed, all binomial ratios
\[
\frac{\binom{n-r}{m-1-\ell}}{\binom{n-1}{m-1}}
\]
remain $O(1)$ uniformly in $\ell\in \{0,\ldots, r-1\}$ for large $n$, and the remaining factors are continuous in $\theta$. Because the sums over $i$ and $\ell$ are finite, it follows that
\[
H_{1,n}(\theta)=O(n^{-1})
\]
locally uniformly on $(0,\infty)$.

\medskip
Taken together, Steps 4A and 4B yield~\eqref{eq:step_4_asymptotics_H_1_n_H_2_n}.

\medskip

\noindent
\textit{Step 5: Laplace asymptotics.}
By Proposition~\ref{prop:maximizer_Phi_p_beta}, for each $\beta\in(0,1)$, the function $\Psi_{p,\beta}:(0,\infty) \to \R$ has a unique maximizer $\theta_{p,\beta}\in(0,\infty)$, and $\Psi_{p,\beta}''(\theta_{p,\beta})<0$. Moreover,
\begin{itemize}
\item[(i)] $\Psi_{p,\beta}(\theta)$ is $C^2$ as a function of $\theta$ and jointly continuous in $(\theta,\beta)$;
\item[(ii)] the maximizer $\theta_{p,\beta}$ depends continuously on $\beta$;
\item[(iii)] $\Psi_{p,\beta}(\theta) \to -\infty$ as $\theta \downarrow 0$ or $\theta \to \infty$, uniformly for $\beta$ in a neighborhood of $\alpha$;
\item[(iv)] $\Xi(\theta)$ is continuous and bounded in a neighborhood of \(\theta_{p,\alpha}\).
\end{itemize}
Therefore, for every fixed $\beta\in (0,1)$, the standard one-dimensional Laplace's method applies and yields
\begin{multline*}
\int_0^\infty \Xi(\theta)e^{\,n\Psi_{p,\beta}(\theta)}
\bigl(H_{1,n}(\theta)+H_{2,n}(\theta)\bigr)\,d\theta
\sim
\Xi(\theta_{p,\beta})\,H_{\beta}(\theta_{p,\beta})\\
\times e^{\,n\Psi_{p,\beta}(\theta_{p,\beta})}
\sqrt{\frac{2\pi}{n\,|\Psi_{p,\beta}''(\theta_{p,\beta})|}},\quad n\to\infty.
\end{multline*}
Moreover, by the above conditions (i)-(iv), Laplace’s method applies locally uniformly in $\beta$ in a neighborhood of $\alpha$. In particular, since $\alpha(n)\to\alpha$, we have $\theta_{p,\alpha(n)}\to \theta_{p,\alpha}$ and we obtain, see \cite[Chapter II, \S 2]{Fedoryuk},
\begin{multline*}
\int_0^\infty \Xi(\theta)e^{\,n\Psi_{p,\alpha(n)}(\theta)}
\bigl(H_{1,n}(\theta)+H_{2,n}(\theta)\bigr)\,d\theta
\sim
\Xi(\theta_{p,\alpha(n)})\,H_{\alpha(n)}(\theta_{p,\alpha(n)})\\
\times e^{\,n\Psi_{p,\alpha(n)}(\theta_{p,\alpha(n)})}
\sqrt{\frac{2\pi}{n\,|\Psi_{p,\alpha(n)}''(\theta_{p,\alpha(n)})|}},\quad n\to\infty.
\end{multline*}

Substituting this into the representation~\eqref{eq:I_n_lambda_1_ldots_lambda_n_exact_formula_decomposed}, we obtain
\begin{multline}\label{eq:I_n_lambda_1_ldots_lambda_n_asymptotic_formula}
I_{n,j(n)}(\lambda_1,\dots,\lambda_r)
\sim
\frac{p(p-1)^{m-1}\,n\binom{n-1}{m-1}}{2\pi^{m/2}\,
\Gamma\!\left(\frac{n+\Lambda+p-m}{p}\right)}
\,
\Xi(\theta_{p,\alpha(n)})\\
\times e^{\,n\Psi_{p,\alpha(n)}(\theta_{p,\alpha(n)})}
\sqrt{\frac{2\pi}{n\,|\Psi_{p,\alpha(n)}''(\theta_{p,\alpha(n)})|}}
H_{\alpha(n)}(\theta_{p,\alpha(n)}; \lambda_1,\ldots, \lambda_r),
\end{multline}
as $n\to\infty$.
This proves Part~(i) of Proposition~\ref{prop:asymptotics_laplace_central}.

\medskip

\noindent
\textit{Step 6: Asymptotics of  $I_{n,j(n)}(0,\ldots, 0)$.}
If $\lambda_1=\ldots = \lambda_r = 0$, then $\Lambda = 0$ by definition and $H_{\alpha(n)}(\theta_{p,\alpha(n)}; 0,\ldots, 0) = 1$. Therefore, \eqref{eq:I_n_lambda_1_ldots_lambda_n_asymptotic_formula} takes the form
\begin{equation}\label{eq:I_n_lambda_1_ldots_lambda_n_asymptotic_formula_if_all_zero}
V_{j(n)}(\mathbb B_p^n) = I_{n,j(n)}(0,\dots,0)
\sim
\frac{p(p-1)^{m-1}\,n\binom{n-1}{m-1}}{2\pi^{m/2}\,
\Gamma\!\left(\frac{n+p-m}{p}\right)}
\,
\Xi(\theta_{p,\alpha(n)})\,
e^{\,n\Psi_{p,\alpha(n)}(\theta_{p,\alpha(n)})}
\sqrt{\frac{2\pi}{n\,|\Psi_{p,\alpha(n)}''(\theta_{p,\alpha(n)})|}}.
\end{equation}
Substituting $m= n-j(n)$ and the formula for $\Xi(\theta_{p,\alpha(n)})$, this becomes exactly the formula~\eqref{eq:intrinsic_ell_p_ball_central_exact_repeat} stated in Proposition~\ref{prop:asymptotics_laplace_central}.

\medskip
\noindent
\textit{Step 7: Relative asymptotic.}
Dividing the asymptotic formula for $I_{n,j(n)}(\lambda_1,\dots,\lambda_r)$ by that for $I_{n,j(n)}(0,\ldots, 0)$ and cancelling common factors, we obtain
\[
\frac{I_{n,j(n)}(\lambda_1,\dots,\lambda_r)}{I_{n,j(n)}(0,\ldots, 0)}
\sim
\frac{
\Gamma\!\left(\frac{n+p-m}{p}\right)
}{
\Gamma\!\left(\frac{n+\Lambda+p-m}{p}\right)
}
\,H_{\alpha(n)}(\theta_{p,\alpha(n)}; \lambda_1,\ldots, \lambda_r)
\sim
\left(\frac{p}{\alpha \, n}\right)^{\Lambda/p} \,H_{\alpha(n)}(\theta_{p,\alpha(n)}; \lambda_1,\ldots, \lambda_r),
\]
where we used the standard gamma-ratio asymptotics together with $\frac{n+p-m}{p}\sim \frac{\alpha\, n}{p}$.
Also, since $\alpha(n)\to\alpha$ and $\theta_{p,\alpha(n)}\to \theta_{p,\alpha}$,
\[
H_{\alpha(n)}(\theta_{p,\alpha(n)};\lambda_1,\ldots, \lambda_r)\to
\prod_{j=1}^r
\left(
\alpha\frac{\mathcal F_p(\theta_{p,\alpha};\lambda_j)}{\mathcal F_p(\theta_{p,\alpha};0)}
+
(1-\alpha)\frac{\mathcal F_p(\theta_{p,\alpha};\lambda_j+p-2)}{\mathcal F_p(\theta_{p,\alpha};p-2)}
\right).
\]
Therefore,
\[
\frac{I_{n,j(n)}(\lambda_1,\dots,\lambda_r)}{I_{n,j(n)}(0,\ldots, 0)}
\sim
\left(\frac{p}{\alpha\, n}\right)^{\Lambda/p}
\prod_{j=1}^r
\left(
\alpha\frac{\mathcal F_p(\theta_{p,\alpha};\lambda_j)}{\mathcal F_p(\theta_{p,\alpha};0)}
+
(1-\alpha)\frac{\mathcal F_p(\theta_{p,\alpha};\lambda_j+p-2)}{\mathcal F_p(\theta_{p,\alpha};p-2)}
\right).
\]
This proves~(ii), and the proposition follows.
\end{proof}

\subsection{Proof of Theorem~\ref{theo:exp_profile_V_j_l_p_balls}}
The next lemma is folklore but difficult to locate in the literature as a standalone result, and so we provide it here together with a proof.
\begin{lemma}[Pointwise convergence is uniform for concave functions]\label{lem:pointwise_implies_uniform_concave}
Let $[a,b]\subset\mathbb{R}$ be a non-empty closed interval. Suppose that, for each $n\in\mathbb{N}$, $f_n:[a,b]\to\mathbb{R}$ is concave. If $\lim_{n\to\infty }f_n(x)=f(x)$ for each $x\in [a,b]$, then
$$
\lim_{n\to\infty}\sup_{x\in [a,b]}|f_n(x)-f(x)|=0.
$$
\end{lemma}
\begin{proof}
Without loss of generality assume that $[a,b]=[0,1]$. Observe that the limit function $f:[0,1]\to\mathbb{R}$ is automatically concave, hence continuous on $[0,1]$. By concavity, for every fixed $n\in\mathbb{N}$, the function $x\mapsto (f_n(x)-f_n(0))/x$ is monotone decreasing on $(0,1]$ and converges pointwise on $[1/3,1]$ to a continuous function $x\mapsto (f(x)-f(0))/x$. By Polya's extension of the Dini theorem, this convergence is uniform on $[1/3,1]$. Combining this with the convergence of $f_n(0)$ to $f(0)$, we conclude that $\sup_{x\in [1/3,1]}|f_n(x)-f(x)|\to 0$, as $n\to\infty$. Using the same argument for $g_n(x):=f_n(1-x)$ and $g(x)=f(1-x)$, we deduce that $\sup_{x\in [0,2/3]}|f_n(x)-f(x)|\to 0$. The proof is complete.
\end{proof}

\begin{proof}[Proof of Theorem~\ref{theo:exp_profile_V_j_l_p_balls}]
Let $f_n:[0,1]\to \R$ be a function obtained by a linear interpolation of the points
$$
\left(\frac{j}{n},\frac 1n \log V_{j}(n^{1/p}\mathbb B_p^n)\right),\quad j\in\{0,1,\ldots,n\}.
$$
For each $n\in\mathbb{N}$, the function $f_n$ is concave by the log-concavity of the intrinsic volumes; see Remark~\ref{rem:logconcavity_intrinsic_vols}. Theorem~\ref{theo:exact_asympt_V_j} (i) and the Stirling formula imply
$$
\lim_{n\to\infty} f_n(\alpha) = \kappa_{p}(\alpha)
+
\sup_{\theta>0}\Psi_{p,\alpha}(\theta) =:  g_p(\alpha),
$$
for every $\alpha \in (0,1)$. This convergence holds also for $\alpha = 0$ since $V_0(n^{1/p}\mathbb B_p^n)=1$, and can be directly verified for $\alpha=1$; see Example~\ref{example:exp_profile_ell_p_at_alpha_1}.   Thus, $f_n$ converges to $g_p$ pointwise on $[0,1]$. Since each $f_n$ is concave, Lemma~\ref{lem:pointwise_implies_uniform_concave} gives $f_n\to g_p$
uniformly on $[0,1]$.
\end{proof}

\subsection{Left-edge regime} In this section, we derive the asymptotics of $I_{n,j}(\lambda_1,\dots,\lambda_r)$  in the regime where $j\geq 1$ is a fixed integer.
\begin{proposition}[Laplace asymptotics in the left-edge regime]\label{prop:laplace_asymptotics_left_edge}
Fix $1<p<\infty$, an integer $r\ge 1$, and consider complex parameters $\lambda_1,\dots,\lambda_r\in\mathbb C$ satisfying
\begin{equation}\label{eq:lambdas_assumptions_asymptotics_left}
\Re\lambda_s>\max\{-1,1-p\},
\qquad s=1,\dots,r,
\qquad
\Re(\Lambda+j+p)>0,
\end{equation}
where $\Lambda:=\lambda_1+\cdots+\lambda_r$. Fix an integer $j\geq 1$ and recall that for sufficiently large  integer $n$ we defined
\[
I_{n, j}(\lambda_1,\dots,\lambda_r)
:=
\int_{\partial \mathbb B_p^n}
\prod_{j=1}^r |x_j|^{\lambda_j}\,
d\Phi_{j}\bigl(\mathbb B_p^n,x\bigr).
\]
Then the following assertions hold:

\vspace{2mm}
\noindent
(i) Absolute asymptotic: As $n\to\infty$, one has
\[
I_{n,j}(\lambda_1,\dots,\lambda_r)\sim
\frac{
\Gamma\!\left(\frac{2p-1}{2p-2}\right)\,
\Gamma\!\left(\frac1{2p-2}\right)^j
\prod_{s=1}^r \Gamma\!\left(\frac{\lambda_s+p-1}{2p-2}\right)
}{
j!\,(p-1)^j\,\pi^{(r+1)/2}
}
\left(
\frac{\sqrt\pi}{\Gamma\!\left(\frac{2p-1}{2p-2}\right)}
\right)^{\frac{\Lambda+j+p}{p}}
n^{\,j+1-\frac{\Lambda+j+p}{p}}.
\]
In particular, setting $\lambda_1=\ldots= \lambda_r = 0$ gives
\[
I_{n,j}(0,\dots,0)\sim
\frac{
\Gamma\!\left(\frac{2p-1}{2p-2}\right)\,
\Gamma\!\left(\frac1{2p-2}\right)^j
}{
j!\,(p-1)^j\,\sqrt{\pi}
}
\left(
\frac{\sqrt\pi}{\Gamma\!\left(\frac{2p-1}{2p-2}\right)}
\right)^{\frac{j+p}{p}}
n^{\frac{j(p-1)}{p}}.
\]

\vspace{2mm}
\noindent
(ii) Relative asymptotic: As $n\to\infty$, one has
\[
\frac{I_{n,j}(\lambda_1,\dots,\lambda_r)}{I_{n,j}(0,\dots,0)}
\sim
\left(\prod_{s=1}^r
\frac{\Gamma\!\left(\frac{\lambda_s+p-1}{2p-2}\right)}{\sqrt{\pi}}
\right)
\left(
\frac{\sqrt\pi}{\Gamma\!\left(\frac{2p-1}{2p-2}\right)}
\right)^{\frac{\Lambda}{p}}
n^{-\Lambda/p}.
\]
\end{proposition}

\begin{remark}
In view of the identity $V_{j}(\mathbb B_p^n) = I_{n,j}(0,\dots,0)$, Part (i) of Proposition~\ref{prop:laplace_asymptotics_left_edge} proves Part~(ii) of Theorem~\ref{theo:exact_asympt_V_j}.
\end{remark}

\begin{proof}[Proof of Proposition~\ref{prop:laplace_asymptotics_left_edge}]
Put $\lambda_{r+1}=\cdots=\lambda_n=0$.
Write
\[
N_n:=n-r-j-1.
\]
Since $r$ and $j$ are fixed, $N_n\to\infty$ and $N_n\sim n$ as $n\to\infty$.

\medskip
\noindent
\textit{Step 1: Large-$\theta$ asymptotic of  $\mathcal F_p(\theta;\nu)$.}
Fix $\nu\in\mathbb C$ with $\Re \nu>-1$ and recall that we defined
$$
\mathcal F_p(\theta;\nu) =\int_{\mathbb R}|y|^\nu e^{-|y|^p-\theta |y|^{2p-2}}\,dy, \qquad \theta >0.
$$
We start by showing  that, as $\theta\to\infty$,
\begin{equation}\label{eq:mathcal_F_p_large_theta_asympt}
\mathcal F_p(\theta;\nu)
=
\frac{\Gamma\!\left(\frac{\nu+1}{2p-2}\right)}{p-1}\,
\theta^{-\frac{\nu+1}{2p-2}}
\left(
1-
\frac{
\Gamma\!\left(\frac{\nu+p+1}{2p-2}\right)
}{
\Gamma\!\left(\frac{\nu+1}{2p-2}\right)
}
\theta^{-\frac p{2p-2}}
+
O\!\left(\theta^{-\frac{2p}{2p-2}}\right)
\right).
\end{equation}
The substitution  $y=\theta^{-1/(2p-2)}x$  gives
\[
\mathcal F_p(\theta;\nu)
=
\theta^{-\frac{\nu+1}{2p-2}}
\int_{\mathbb R}|x|^\nu e^{-|x|^{2p-2}}e^{-\theta^{-\frac p{2p-2}}|x|^p}\,dx.
\]
Using $e^{-u}=1-u+O(u^2)$ with the remainder bounded by $Cu^2$ uniformly for $u\ge 0$, we obtain
\[
e^{-\theta^{-\frac p{2p-2}}|x|^p}
=
1-\theta^{-\frac p{2p-2}}|x|^p
+O\!\left(\theta^{-\frac{2p}{2p-2}}|x|^{2p}\right),
\qquad x\in \mathbb R, \; \theta \geq 0.
\]
Since $|x|^{\Re\nu+2p}e^{-|x|^{2p-2}}\in L^1(\mathbb R)$, dominated convergence yields~\eqref{eq:mathcal_F_p_large_theta_asympt}.

\medskip
\noindent\textit{Step 2: Decomposition of  $S_n(\theta)$.}
Recall from  Equations~\eqref{eq:I_n_lambda_1_ldots_lambda_n_exact_formula_conditions_lambdas},  \eqref{eq:I_n_lambda_1_ldots_lambda_n_exact_formula_setup}, \eqref{eq:S_n_def_setup_subsec} the representation
\begin{equation}\label{eq:I_n_lambda_1_ldots_lambda_n_exact_formula_left_edge_proof}
I_{n,j}(\lambda_1,\dots,\lambda_r)
=
\frac{p(p-1)^{n-j-1}}{2\pi^{(n-j)/2}\,
\Gamma\!\left(\frac{j+\Lambda+p}{p}\right)}
\int_0^\infty \theta^{\frac {n-j}2-1} S_n(\theta)\,d\theta,
\end{equation}
where
\begin{align}
S_n(\theta)
&:=
\sum_{i=1}^n
\sum_{\substack{I\subseteq \{1,\ldots,n\}\setminus\{i\}\\ |I|=n-j-1}}
\mathcal F_p(\theta;\lambda_i+2p-2)
\prod_{j\in I}\mathcal F_p(\theta;\lambda_j+p-2)
\prod_{k\notin I\cup\{i\}}\mathcal F_p(\theta;\lambda_k).
\label{eq:S_n_def_left_edge_proof_recall}
\end{align}
The subsets $I\subseteq\{1,\ldots,n\}\setminus\{i\}$ with $|I|=n-j-1$ are in bijection with the subsets $E=\{1,\ldots,n\}\setminus (I\cup\{i\}) \subseteq\{1,\ldots,n\}\setminus\{i\}$ satisfying  $|E|=j$.
Thus,
\[
S_n(\theta)=
\sum_{i=1}^n
\sum_{\substack{E\subseteq\{1,\ldots,n\}\setminus\{i\}\\ |E|=j}}
T_{n; i,E}(\theta),
\]
where
\[
T_{n;i,E}(\theta):=
\mathcal F_p(\theta;\lambda_i+2p-2)
\prod_{\ell\notin E\cup\{i\}}\mathcal F_p(\theta;\lambda_\ell+p-2)
\prod_{k\in E}\mathcal F_p(\theta;\lambda_k).
\]
As we shall see, the main contribution comes from the pairs $(i,E)$ satisfying  $i\in \{r+1,\ldots, n\}$, $E\subseteq \{r+1,\dots,n\}\setminus\{i\}$, and $|E|=j$. We call these pairs ``generic''. Since $\lambda_{r+1}= \ldots = \lambda_n = 0$ and
$$
\mathcal I_p(\theta)=\mathcal F_p(\theta;0),\qquad
\mathcal J_p(\theta)=\mathcal F_p(\theta;p-2),\qquad
\mathcal K_p(\theta)=\mathcal F_p(\theta;2p-2),
$$
for every generic pair $(i,E)$ one has
\begin{equation}\label{eq:T_n_i_E_generic_left_edge}
T_{n;i,E}(\theta) =
\mathcal K_p(\theta)\,
\mathcal I_p(\theta)^j\,
\mathcal J_p(\theta)^{N_n}
\prod_{s=1}^r \mathcal F_p(\theta;\lambda_s+p-2).
\end{equation}
The number of generic pairs $(i,E)$ is
\[
M_n^\#:=(n-r)\binom{n-r-1}{j}\sim \frac{n^{j+1}}{j!}, \qquad n\to\infty.
\]
We therefore decompose $S_n(\theta) = M_n(\theta) + R_n(\theta)$, where  the ``main contribution'' $M_n(\theta)$ and the ``remainder'' $R_n(\theta)$ are defined by
\begin{equation}\label{eq:main_contribution_left_edge_proof}
M_n(\theta):=
M_n^\#\,
\mathcal K_p(\theta)\,
\mathcal I_p(\theta)^j\,
\mathcal J_p(\theta)^{N_n}
\prod_{s=1}^r \mathcal F_p(\theta;\lambda_s+p-2),
\qquad
R_n(\theta):=S_n(\theta)-M_n(\theta).
\end{equation}
Since the total number of pairs $(i,E)$ is $n\binom{n-1}{j}$, the number of terms in $R_n$ is
\begin{equation}\label{eq:R_n_num_versus_M_n_num}
R_n^\# := n\binom{n-1}{j}-(n-r)\binom{n-r-1}{j}=O(n^j) = o(M_n^\#), \qquad n\to\infty.
\end{equation}
In the next two steps we derive the asymptotics of $\int_0^\infty \theta^{(n-j)/2-1}M_n(\theta)\,d\theta$ (the main contribution) and $\int_0^\infty \theta^{(n-j)/2-1}R_n(\theta)\,d\theta$ (the negligible contribution), as $n\to\infty$.

\medskip
\noindent\textit{Step 3: Main contribution.}
We now determine the asymptotics of  $\int_{0}^\infty \theta^{(n-j)/2-1}M_n(\theta)\, d\theta$.
From Equation~\eqref{eq:mathcal_F_p_large_theta_asympt} we have, as $\theta\to\infty$,
\begin{align}
\mathcal K_p(\theta)
&=
\frac{\Gamma\!\left(\frac{2p-1}{2p-2}\right)}{p-1}
\theta^{-\frac{2p-1}{2p-2}}
\left(1+O\!\left(\theta^{-\frac p{2p-2}}\right)\right),
\label{eq:K_p_asympt}
\\
\mathcal I_p(\theta)
&=
\frac{\Gamma\!\left(\frac1{2p-2}\right)}{p-1}
\theta^{-\frac1{2p-2}}
\left(1+O\!\left(\theta^{-\frac p{2p-2}}\right)\right),
\label{eq:I_p_asympt}
\\
\mathcal J_p(\theta)
&=
\frac{\sqrt\pi}{p-1}\,\theta^{-1/2}
\left(
1-
\frac{\Gamma\!\left(\frac{2p-1}{2p-2}\right)}{\sqrt\pi}\,
\theta^{-\frac p{2p-2}}
+
O\!\left(\theta^{-\frac{2p}{2p-2}}\right)
\right).  \label{eq:J_p_asympt}
\end{align}
Additionally, for all $s=1,\dots,r$, we have
$$
\mathcal F_p(\theta;\lambda_s+p-2)
=
\frac{\Gamma\!\left(\frac{\lambda_s+p-1}{2p-2}\right)}{p-1}
\theta^{-\frac{\lambda_s+p-1}{2p-2}}
\left(1+O\!\left(\theta^{-\frac p{2p-2}}\right)\right),
$$
as $\theta\to\infty$.
Substituting these results  into~\eqref{eq:main_contribution_left_edge_proof} yields
\begin{multline}\label{eq:tech_estimate_124}
\theta^{(n-j)/2-1}M_n(\theta)
=
M_n^\#
\frac{
\Gamma\!\left(\frac{2p-1}{2p-2}\right)\,
\Gamma\!\left(\frac1{2p-2}\right)^j
\prod_{s=1}^r \Gamma\!\left(\frac{\lambda_s+p-1}{2p-2}\right)
}{
(p-1)^{j+r+1}
}
\left(\frac{\sqrt\pi}{p-1}\right)^{N_n}
\theta^{-1-\frac{\Lambda+j+p}{2p-2}}
\\
\qquad\times
\left(1+O\!\left(\theta^{-\frac p{2p-2}}\right)\right)
\left(
1-
\frac{\Gamma\!\left(\frac{2p-1}{2p-2}\right)}{\sqrt\pi}
\theta^{-\frac p{2p-2}}
+
O\!\left(\theta^{-\frac{2p}{2p-2}}\right)
\right)^{N_n},
\end{multline}
as $\theta \to\infty$.
In a moment we shall see that the main contribution to $\int_0^{\infty} \theta^{(n-j)/2-1}M_n(\theta) \, d\theta$ comes from those values of $\theta$ for which   the quantity
\[
u:=
N_n\frac{\Gamma\!\left(\frac{2p-1}{2p-2}\right)}{\sqrt\pi}\,
\theta^{-\frac p{2p-2}}
\]
is of constant order.  First observe that under this change of variables,
\[
\theta^{-1-\frac{\Lambda+j+p}{2p-2}}\,d\theta
=
\frac{2p-2}{p}
\left(
N_n\frac{\Gamma\!\left(\frac{2p-1}{2p-2}\right)}{\sqrt\pi}
\right)^{-\frac{\Lambda+j+p}{p}}
u^{\frac{\Lambda+j+p}{p}-1}\,du.
\]
Let $T>0$ be sufficiently large. Then
\begin{multline*}
\int_T^\infty \theta^{(n-j)/2-1}M_n(\theta)\,d\theta
=
\frac{2p-2}{p}\,
M_n^\#
\frac{
\Gamma\!\left(\frac{2p-1}{2p-2}\right)\,
\Gamma\!\left(\frac1{2p-2}\right)^j
\prod_{s=1}^r \Gamma\!\left(\frac{\lambda_s+p-1}{2p-2}\right)
}{
(p-1)^{j+r+1}
}
\left(\frac{\sqrt\pi}{p-1}\right)^{N_n}
\\
\qquad\times
\left(
N_n\frac{\Gamma\!\left(\frac{2p-1}{2p-2}\right)}{\sqrt\pi}
\right)^{-\frac{\Lambda+j+p}{p}}
\int_0^{\delta N_n}
u^{\frac{\Lambda+j+p}{p}-1} G_n(u)\,du,
\end{multline*}
where $ \delta= \Gamma(\frac{2p-1}{2p-2}) \pi^{-1/2} T^{-\frac p{2p-2}}$ and
\[
G_n(u)=
\left(1+O\!\left(\frac{u}{N_n}\right)\right)
\left(1-\frac{u}{N_n}+O\!\left(\frac{u^2}{N_n^2}\right)\right)^{N_n},
\]
uniformly for $0\le u\le \delta N_n$, once $T$ is chosen so large that $\delta$ is sufficiently small. Now, for each fixed $u>0$, one has
\[
G_n(u)\to e^{-u}, \qquad n\to\infty.
\]
Moreover, if $\delta>0$ is small enough, then for $0\le u\le \delta N_n$, one has $|G_n(u)|\le C e^{-u/2}$ and hence
\[
\left|u^{\frac{\Lambda+j+p}{p}-1}G_n(u)\right|
\le
C\,u^{\frac{\Re(\Lambda+j+p)}{p}-1}e^{-u/2}.
\]
The right-hand side is integrable on $(0,\infty)$ by the assumption $\Re(\Lambda+j+p)>0$. By dominated convergence,
\[
\int_0^{\delta N_n}
u^{\frac{\Lambda+j+p}{p}-1}G_n(u)\,du
\to
\Gamma\!\left(\frac{\Lambda+j+p}{p}\right),
\qquad n\to\infty.
\]
Thus, as $n\to\infty$,
\begin{multline}\label{eq:int_T_infty_main_contrib_left_edge}
\int_T^\infty \theta^{(n-j)/2-1}M_n(\theta)\,d\theta
\sim
\frac{2p-2}{p}\,
M_n^\#
\frac{
\Gamma\!\left(\frac{2p-1}{2p-2}\right)\,
\Gamma\!\left(\frac1{2p-2}\right)^j
\prod_{s=1}^r \Gamma\!\left(\frac{\lambda_s+p-1}{2p-2}\right)
}{
(p-1)^{j+r+1}
}
\left(\frac{\sqrt\pi}{p-1}\right)^{N_n}
\\
\qquad\times
\left(
N_n\frac{\Gamma\!\left(\frac{2p-1}{2p-2}\right)}{\sqrt\pi}
\right)^{-\frac{\Lambda+j+p}{p}}
\Gamma\!\left(\frac{\Lambda+j+p}{p}\right).
\end{multline}

Next we show that $\int_0^T \theta^{(n-j)/2-1}M_n(\theta)\,d\theta$ is exponentially smaller than this expression.
We observe that  for every finite $\theta>0$,
\[
\mathcal J_p(\theta)
=
\int_{\mathbb R} |x|^{p-2}e^{-|x|^p-\theta |x|^{2p-2}}\,dx
<
\int_{\mathbb R} |x|^{p-2}e^{-\theta |x|^{2p-2}}\,dx
=
\frac{\sqrt\pi}{p-1}\,\theta^{-1/2},
\]
because $e^{-|x|^p}<1$ for $x\neq 0$. Since $\mathcal J_p(0)$ is finite and  $\theta\mapsto \frac{p-1}{\sqrt\pi}\,\theta^{1/2}\mathcal J_p(\theta)$ is continuous on $[0,T]$, there exists $\rho_T\in(0,1)$ such that
\begin{equation}\label{eq:J_p_theta_estimate_theta_-1/2}
\mathcal J_p(\theta)\le \frac{\sqrt\pi}{p-1}\rho_T\,\theta^{-1/2},
\qquad 0<\theta\le T.
\end{equation}
On the other hand, \eqref{eq:T_n_i_E_generic_left_edge} implies that for $\theta \in (0,T]$ and every generic pair $(i,E)$ we have $|T_{n;i,E}(\theta)| \leq C |\mathcal J_p(\theta)|^{N_n}$. Since $N_n=n-r-j-1$, we get
\begin{equation}\label{eq:int_0_T_M_n_left_edge_proof}
\int_0^T \theta^{(n-j)/2-1}M_n(\theta)\,d\theta
=
O\!\left(
M_n^\#
\left(\frac{\sqrt\pi}{p-1}\rho_T\right)^{N_n}
\right),
\end{equation}
which is smaller than the right-hand side of~\eqref{eq:int_T_infty_main_contrib_left_edge} by the exponentially decaying factor $\rho_T^{N_n}$.  Thus,
\[
\int_0^T \theta^{(n-j)/2-1}M_n(\theta)\,d\theta
=
o\left(\int_T^\infty \theta^{(n-j)/2-1}M_n(\theta)\,d\theta\right), \qquad n\to\infty.
\]

\medskip
\noindent\textit{Step 4: Negligible contribution.}
Fix some pair $(i, E)$ with $i\in \{1,\ldots,n\}$ and  $E\subseteq\{1,\ldots,n\}\setminus\{i\}$ satisfying  $|E|=j$.  Consider the summand
\begin{equation}\label{eq:T_i_E}
T_{n; i,E}(\theta) =
\mathcal F_p(\theta;\lambda_i+2p-2)
\prod_{\ell\notin E\cup\{i\}}\mathcal F_p(\theta;\lambda_\ell+p-2)
\prod_{k\in E}\mathcal F_p(\theta;\lambda_k).
\end{equation}

Recall the convention $\lambda_{r+1} = \ldots = \lambda_n = 0$.
Among the indices $\ell = r+1,\dots,n$, at most $j+1$ fail to contribute a factor $\mathcal J_p(\theta)=\mathcal F_p(\theta; p-2)$: these failures come  from the set $E\cup \{i\}$. Thus, every $T_{n; i,E}$ contains at least $n-r-(j+1)=N_n$ factors $\mathcal J_p(\theta)$. To every such factor we apply the asymptotic~\eqref{eq:J_p_asympt}.
Every remaining factor belongs to the fixed finite collection
\[
\mathcal F_p(\theta;\lambda_s),\quad
\mathcal F_p(\theta;\lambda_s+p-2),\quad
\mathcal F_p(\theta;\lambda_s+2p-2),\quad
\mathcal F_p(\theta;0),\quad
\mathcal F_p(\theta; 2p-2), \qquad s=1,\ldots, r.
\]
The number of these factors is at most $r+j+1$, which does not depend on $n$. To any of these factors, say $\mathcal F_p(\theta;\nu)$, we apply the bound
$$
\mathcal F_p(\theta;\nu)
\leq C
\,
\theta^{-\frac{\nu+1}{2p-2}},
\qquad \theta >T,
$$
which follows from~\eqref{eq:mathcal_F_p_large_theta_asympt}. Here $T$ is sufficiently large. Combining bounds  for all factors gives
\begin{multline}\label{eq:tech_estimate_127}
|T_{n; i,E}(\theta)|
\leq
C_1
\left(\frac{\sqrt\pi}{p-1}\right)^{N_n}
\theta^{-\frac{\Re \Lambda + (2p - 2) + (n - j - 1)(p - 2) + n}{2p-2}}
\left(1+O\!\left(\theta^{-\frac p{2p-2}}\right)\right)
\\
\left(
1-
\frac{\Gamma\!\left(\frac{2p-1}{2p-2}\right)}{\sqrt\pi}
\theta^{-\frac p{2p-2}}
+
O\!\left(\theta^{-\frac{2p}{2p-2}}\right)
\right)^{N_n},
\qquad \theta >T.
\end{multline}
The power of $\theta$ is explained as follows: The sum of all parameters $\nu$  in the factors  $\mathcal F_p (\theta; \nu)$ appearing in~\eqref{eq:T_i_E} is always $\Lambda + (2p - 2) + (n - j - 1)(p - 2)$,
which is independent of $(i,E)$.

The right-hand side of~\eqref{eq:tech_estimate_127} is independent of $(i, E)$. Recalling that $R_n(\theta)$ has $R_n^\#$ summands of the form $T_{n; i, E}$, and multiplying by $\theta^{(n-j)/2 - 1}$,  we obtain
\begin{multline*}
\bigl|\theta^{(n-j)/2-1}R_n(\theta)\bigr|
\le
C_1 \,R_n^\#\,
\left(\frac{\sqrt\pi}{p-1}\right)^{N_n}
\theta^{-1-\frac{\Re \Lambda + j+p}{2p-2}}
\left(1+O\!\left(\theta^{-\frac p{2p-2}}\right)\right)
\\
\left(
1-
\frac{\Gamma\!\left(\frac{2p-1}{2p-2}\right)}{\sqrt\pi}
\theta^{-\frac p{2p-2}}
+
O\!\left(\theta^{-\frac{2p}{2p-2}}\right)
\right)^{N_n},
\qquad \theta >T.
\end{multline*}

This is the same right-hand side as in~\eqref{eq:tech_estimate_124}, except that the multiplicative constant is different, $\Lambda$ is replaced by $\Re \Lambda$ and, most importantly, $M_n^\#$ is replaced by $R_n^\#$. Recall from~\eqref{eq:R_n_num_versus_M_n_num} that
$R_n^\#  = o(M_n^\#)$ as  $n\to\infty$. Repeating the argument of Step~3 therefore gives
$$
\int_T^\infty \theta^{(n-j)/2-1}R_n(\theta)\,d\theta = o\left(\int_T^\infty \theta^{(n-j)/2-1}M_n(\theta)\,d\theta\right).
$$

It remains to show that $\int_0^T \theta^{(n-j)/2-1}R_n(\theta)\,d\theta$ is also negligible.  For $0<\theta\le T$, every non-$\mathcal J_p$ factor in~\eqref{eq:T_i_E}  is bounded, while every $\mathcal J_p(\theta)$ is also bounded and additionally satisfies
\[
\mathcal J_p(\theta)\le \frac{\sqrt\pi}{p-1}\rho_T\,\theta^{-1/2},
\qquad 0<\theta\le T,
\]
with $\rho_T\in (0,1)$; see~\eqref{eq:J_p_theta_estimate_theta_-1/2}.  The number of $\mathcal J_p$-factors is at least $N_n$, hence $|T_{n;i,E}(\theta)| \leq C |\mathcal J_p(\theta)|^{N_n}$ for every summand in $R_n(\theta)$. Since  $N_n=n-r-j-1$, we obtain the following analogue of~\eqref{eq:int_0_T_M_n_left_edge_proof}:
\begin{equation}\label{eq:int_0_T_R_n_left_edge_proof}
\int_0^T \theta^{(n-j)/2-1}R_n(\theta)\,d\theta
=
O\!\left(
R_n^\#
\left(\frac{\sqrt\pi}{p-1}\rho_T\right)^{N_n}
\right)
=
o\!\left(
M_n^\#
\left(\frac{\sqrt\pi}{p-1}\rho_T\right)^{N_n}
\right).
\end{equation}
This is the same right-hand side as in~\eqref{eq:int_0_T_M_n_left_edge_proof}, and it is again  smaller than the right-hand side of~\eqref{eq:int_T_infty_main_contrib_left_edge} by the exponentially decaying factor $\rho_T^{N_n}$. So
$$
\int_0^T \theta^{(n-j)/2-1}R_n(\theta)\,d\theta = o\left(\int_T^\infty \theta^{(n-j)/2-1}M_n(\theta)\,d\theta\right), \qquad n\to\infty.
$$

\medskip
\noindent\textit{Step 5: Taking pieces together.}
Recall that
\[
I_{n,j}(\lambda_1,\dots,\lambda_r)
=
\frac{p(p-1)^{n-j-1}}{2\pi^{(n-j)/2}\Gamma\!\left(\frac{\Lambda+j+p}{p}\right)}
\int_0^\infty \theta^{(n-j)/2-1}S_n(\theta)\,d\theta.
\]
Since $S_n(\theta) = M_n(\theta) + R_n(\theta)$, Steps 3 and 4  give $\int_0^\infty \theta^{(n-j)/2-1} S_n(\theta)\,d\theta
\sim
\int_T^\infty \theta^{(n-j)/2-1}M_n(\theta)\,d\theta$ and therefore, by~\eqref{eq:int_T_infty_main_contrib_left_edge},
\begin{multline*}
\int_0^\infty \theta^{(n-j)/2-1} S_n(\theta)\,d\theta
\sim
\frac{2p-2}{p}\,
M_n^\#
\frac{
\Gamma\!\left(\frac{2p-1}{2p-2}\right)\,
\Gamma\!\left(\frac1{2p-2}\right)^j
\prod_{s=1}^r \Gamma\!\left(\frac{\lambda_s+p-1}{2p-2}\right)
}{
(p-1)^{j+r+1}
}
\\
\times\left(\frac{\sqrt\pi}{p-1}\right)^{N_n}
\left(
N_n\frac{\Gamma\!\left(\frac{2p-1}{2p-2}\right)}{\sqrt\pi}
\right)^{-\frac{\Lambda+j+p}{p}}
\Gamma\!\left(\frac{\Lambda+j+p}{p}\right).
\end{multline*}
Substituting this asymptotic and taking into account the relations $M_n^\# \sim n^{j+1}/j!$ and $N_n= n-r-j-1\sim n$ gives
\[
I_{n,j}(\lambda_1,\dots,\lambda_r)\sim
\frac{n^{j+1-\frac{\Lambda+j+p}{p}}}{j!}
\frac{
\Gamma\!\left(\frac{2p-1}{2p-2}\right)\,
\Gamma\!\left(\frac1{2p-2}\right)^j
\prod_{s=1}^r \Gamma\!\left(\frac{\lambda_s+p-1}{2p-2}\right)
}{
(p-1)^j\,\pi^{(r+1)/2}
}
\left(
\frac{\sqrt\pi}{\Gamma\!\left(\frac{2p-1}{2p-2}\right)}
\right)^{\frac{\Lambda+j+p}{p}}.
\]
This is exactly the claimed asymptotic formula.
\end{proof}

\subsection{Right-edge regime}In this section, we prove an asymptotic formula for  $I_{n,j(n)}(\lambda_1,\dots,\lambda_r)$  in the regime where $j(n) = n-m$ and $m \geq 1$ is a fixed integer.
\begin{proposition}[Laplace asymptotics in the right-edge regime]\label{prop:laplace_asymptotics_right_edge}
Fix $1<p<\infty$, an integer $r\ge 1$, and complex parameters
\begin{equation}\label{eq:lambdas_assumptions_asymptotics_right_edge}
\lambda_1,\dots,\lambda_r\in\mathbb C,
\qquad
\Re \lambda_j>\max\{-1,1-p\},\quad j=1,\dots,r.
\end{equation}
Put $\Lambda:=\lambda_1+\cdots+\lambda_r$. Fix an integer $m\geq 1$ and recall that for sufficiently large  integer $n$ we defined
\[
I_{n, n-m}(\lambda_1,\dots,\lambda_r)
:=
\int_{\partial \mathbb B_p^n}
\prod_{j=1}^r |x_j|^{\lambda_j}\,
d\Phi_{n-m}\bigl(\mathbb B_p^n,x\bigr).
\]
Then the following assertions hold:

\vspace{2mm}
\noindent
(i) Absolute asymptotic: As $n\to\infty$, one has
\[
I_{n, n-m}(\lambda_1,\dots,\lambda_r)
\sim
\frac{\Gamma\!\left(\frac m2\right)}{2(m-1)!}
\left(
\frac{p(p-1)\Gamma\!\left(1-\frac1p\right)}
{\pi\,\Gamma\!\left(\frac1p\right)}
\right)^{m/2}
\left(\prod_{\ell=1}^r
\frac{\Gamma\!\left(\frac{\lambda_\ell+1}{p}\right)}
{\Gamma\!\left(\frac1p\right)}
\right)
\frac{\left(\frac{2}{p}\Gamma\!\left(\frac1p\right)\right)^n\,n^{m/2}}
{\Gamma\!\left(\frac{n+\Lambda+p-m}{p}\right)}.
\]
In particular, setting $\lambda_1=\ldots= \lambda_r = 0$ gives
\[
I_{n, n-m}(0,\dots,0)
\sim
\frac{\Gamma\!\left(\frac m2\right)}{2(m-1)!}
\left(
\frac{p(p-1)\Gamma\!\left(1-\frac1p\right)}
{\pi\,\Gamma\!\left(\frac1p\right)}
\right)^{m/2}
\frac{\left(\frac{2}{p}\Gamma\!\left(\frac1p\right)\right)^n\,n^{m/2}}
{\Gamma\!\left(\frac{n+p-m}{p}\right)}.
\]

\vspace{2mm}
\noindent
(ii) Relative asymptotic: As $n\to\infty$, one has
\[
\frac{I_{n, n-m}(\lambda_1,\dots,\lambda_r)}{I_{n, n-m}(0,\dots,0)}
\sim
\frac{\Gamma\!\left(\frac{n+p-m}{p}\right)}
{\Gamma\!\left(\frac{n+\Lambda+p-m}{p}\right)}
\cdot
\prod_{\ell=1}^r
\frac{\Gamma\!\left(\frac{\lambda_\ell+1}{p}\right)}
{\Gamma\!\left(\frac1p\right)}
\sim
\left(\frac{p}{n}\right)^{\Lambda/p}
\cdot
\prod_{\ell=1}^r
\frac{\Gamma\!\left(\frac{\lambda_\ell+1}{p}\right)}
{\Gamma\!\left(\frac1p\right)}.
\]
\end{proposition}

\begin{remark}
In view of $V_{n-m}(\mathbb B_p^n) = I_{n,n-m}(0,\dots,0)$, Part (i) of Proposition~\ref{prop:laplace_asymptotics_right_edge} proves Part~(iii) of Theorem~\ref{theo:exact_asympt_V_j}.
\end{remark}

\begin{proof}[Proof of Proposition~\ref{prop:laplace_asymptotics_right_edge}]
Extend the parameters by $\lambda_{r+1}=\cdots=\lambda_n=0$.  Recall from Equations~\eqref{eq:I_n_lambda_1_ldots_lambda_n_exact_formula_conditions_lambdas}, \eqref{eq:I_n_lambda_1_ldots_lambda_n_exact_formula_setup}, \eqref{eq:S_n_def_setup_subsec} the formula
\begin{equation}\label{eq:I_n_n-m_as_int_of_S}
I_{n, n-m}(\lambda_1,\dots,\lambda_r)
=
\frac{p(p-1)^{m-1}}{2\pi^{m/2}\,
\Gamma\!\left(\frac{n+\Lambda+p-m}{p}\right)}
\int_0^\infty \theta^{\frac m2-1} S_n(\theta)\,d\theta,
\end{equation}
where
$$
S_n(\theta)=\sum_{i=1}^n\sum_{\substack{I\subseteq [n]\setminus\{i\}\\ |I|=m-1}}\mathcal F_p(\theta;\lambda_i+2p-2)
\prod_{j\in I}\mathcal F_p(\theta;\lambda_j+p-2)
\prod_{k\notin I\cup\{i\}}\mathcal F_p(\theta;\lambda_k).
$$

\medskip
\noindent
\textit{Step 1: Decomposition of $S_n(\theta)$.}
We split $S_n(\theta)$ into the main contribution $M_n(\theta)$ for which all $m$ special positions $I\cup \{i\}$ lie among
$\{r+1,\dots,n\}$, and the complementary contribution $R_n(\theta)$:
$$
S_n(\theta)=\sum_{i=r+1}^n
\sum_{\substack{I\subseteq \{r+1,\dots,n\}\setminus\{i\}\\ |I|=m-1}}\left(\cdots\right)+ \sum_{i=1}^n
\sum_{\substack{I\subseteq \{1,\ldots,n\}\setminus\{i\}\\ |I|=m-1\\ \{i\}\cup I\not\subseteq \{r+1,\dots,n\}}}\left(\cdots\right)=:M_n(\theta)+R_n(\theta),
$$
Thus, $R_n(\theta)$ is exactly the sum of those terms for which at least one of the $m$ special
positions $\{i\}\cup I$ falls in $\{1,\dots,r\}$.
Since $\lambda_{r+1}=\cdots=\lambda_n=0$, every summand in $M_n(\theta)$ is identical, and hence
\[
M_n(\theta)
=
(n-r)\binom{n-r-1}{m-1}\,
\mathcal I_p(\theta)^{\,n-r-m}\,G(\theta)
\quad \text{ with } G(\theta):=
\mathcal K_p(\theta)\,
\mathcal J_p(\theta)^{m-1}
\prod_{\ell=1}^r \mathcal F_p(\theta;\lambda_\ell).
\]
Fix $\delta >0$. In the next step, we shall analyze the integral $\int_0^\delta \theta^{\frac m2-1} M_n(\theta)\,d\theta$, which is the principal contribution.  Then we shall show that the integrals $\int_\delta^{\infty} \theta^{\frac m2-1} M_n(\theta)\,d\theta$, $\int_0^\delta \theta^{\frac m2-1} R_n(\theta)\,d\theta$,  and $\int_\delta^{\infty} \theta^{\frac m2-1} M_n(\theta)\,d\theta$ are negligible.

\medskip
\noindent
\textit{Step 2: Principal contribution: $\int_0^\delta \theta^{\frac m2-1} M_n(\theta)\,d\theta$.}
Observe that $M_n(\theta)$ contains the large-power factor $\mathcal I_p(\theta)^{\,n-r-m}$, and $\mathcal I_p'(\theta)=-\mathcal K_p(\theta)<0$ for all $\theta\geq 0$,
so $\mathcal I_p(\theta)$ is strictly decreasing and has its unique maximum at $\theta=0$. This suggests that the main contribution to the integral comes from a neighborhood of $\theta=0$.  To make this precise, we shall use the following standard result~\cite[Chapter~II, p.~58]{Wong2001}.
\begin{lemma}[Endpoint Laplace asymptotics]
\label{lem:endpoint_laplace_multiplicative}
Let $\delta>0$ and $\alpha>0$. Let $F:[0,\delta]\to (0,\infty)$ and
$H:[0,\delta]\to \mathbb C$ be continuous functions. Assume that
\begin{itemize}
\item $F(x)<F(0)$ for all $x\in (0,\delta]$;
\item the right derivative $F'(0)$ exists and satisfies $F'(0)<0$;
\item $H(0)\neq 0$.
\end{itemize}
Then, as $N\to\infty$,
\[
\int_0^\delta x^{\alpha-1}F(x)^N H(x)\,dx
\sim
H(0)\,\Gamma(\alpha)
\left(\frac{F(0)}{-F'(0)}\right)^\alpha
F(0)^N\,N^{-\alpha}.
\]
\end{lemma}

We apply the lemma with   $\alpha:=\frac m2$,  $N:=n-r-m$, and $F(\theta):=\mathcal I_p(\theta)$ for $0\le \theta\le \delta$.
Then $F$ is continuous, strictly positive and strictly decreasing on $[0,\delta]$, and has right derivative $I_p'(0)= -\mathcal K_p(0)<0$.
 Therefore, Lemma~\ref{lem:endpoint_laplace_multiplicative}
applies and yields
\begin{align}
\int_0^\delta \theta^{\frac m2-1}M_n(\theta)\,d\theta
&=
(n-r)\binom{n-r-1}{m-1}
\int_0^\delta
\theta^{\frac m2-1}\mathcal I_p(\theta)^{\,n-r-m}G(\theta)\,d\theta  \notag\\
&\sim
\frac{\Gamma\!\left(\frac m2\right)n^{m/2} }{(m-1)!}\,
\mathcal I_p(0)^{\,n-r-m/2}\,
\mathcal K_p(0)^{\,1-m/2}\,
\mathcal J_p(0)^{\,m-1}
\prod_{\ell=1}^r \mathcal F_p(0;\lambda_\ell).  \label{eq:right_edge_proof_int_0_delta_endpoint_laplace}
\end{align}

\medskip
\noindent
\textit{Step 3: $\int_\delta^{\infty} \theta^{\frac m2-1} M_n(\theta)\,d\theta$ and $\int_\delta^{\infty} \theta^{\frac m2-1} R_n(\theta)\,d\theta$ are negligible.}
Since $\mathcal I_p'(\theta)=-\mathcal K_p(\theta)<0$, the function $\mathcal I_p(\theta)$ is strictly decreasing on $[0,\infty)$, and
therefore
\[
q_\delta:=\sup_{\theta\ge \delta}\mathcal I_p(\theta)=\mathcal I_p(\delta)<\mathcal I_p(0).
\]
Moreover, for every $\nu\in \C$ with $\Re\nu>-1$, the change of variables
$u=\theta^{1/(2p-2)}y$ gives
\[
\mathcal F_p(\theta;\nu) =
\theta^{-\frac{\nu+1}{2p-2}}
\int_{\mathbb R}|x|^\nu e^{-|x|^{2p-2}}e^{-\theta^{-\frac p{2p-2}}|x|^p}\,dx = O\left(\theta^{-\frac{\Re\nu+1}{2p-2}}\right),
\qquad \theta\to\infty.
\]
Hence, one can choose an integer $L\ge 1$ such that
\[
\int_\delta^\infty
\theta^{\frac m2-1}\,
|G(\theta)|\,\mathcal I_p(\theta)^L\,d\theta<\infty.
\]
Since $\mathcal I_p(\theta)\le q_\delta$ for $\theta\ge \delta$, it follows that
\begin{align*}
\int_\delta^\infty \theta^{\frac m2-1}|M_n(\theta)|\,d\theta
&\le
(n-r)\binom{n-r-1}{m-1}
q_\delta^{\,n-r-m-L}
\int_\delta^\infty
\theta^{\frac m2-1}|G(\theta)|\mathcal I_p(\theta)^L\,d\theta \\
&=
O\!\left(n^m q_\delta^{\,n-r-m-L}\right).
\end{align*}
Since $q_\delta<\mathcal I_p(0)$, this is exponentially smaller than the main term obtained in~\eqref{eq:right_edge_proof_int_0_delta_endpoint_laplace}.

The same argument applies to $R_n(\theta)$: each summand of $R_n(\theta)$ contains at least
$n-r-m+1$ factors $\mathcal I_p(\theta)$, while all the remaining factors are of at most polynomial
growth as $\theta\to\infty$. Moreover, the number of summands is $O(n^{m-1})$. Hence,
\[
\int_\delta^\infty \theta^{\frac m2-1}|R_n(\theta)|\,d\theta
=
O\!\left(n^{m-1} q_\delta^{\,n-r-m+1-L}\right),
\]
which is likewise exponentially negligible compared to the main term obtained in~\eqref{eq:right_edge_proof_int_0_delta_endpoint_laplace}.

\medskip
\noindent
\textit{Step 3: $\int_0^\delta \theta^{\frac m2-1}R_n(\theta)\,d\theta$ is negligible.}
On the interval $[0,\delta]$, all the finitely many functions
$$
\mathcal F_p(\theta;\lambda_\ell),\qquad
\mathcal F_p(\theta;\lambda_\ell+p-2),\qquad
\mathcal F_p(\theta;\lambda_\ell+2p-2),
\quad \ell=1,\dots,r,
$$
as well as $\mathcal F_p(\theta;p-2)$ and  $\mathcal F_p(\theta;2p-2)$ are bounded. Every summand contributing to $R_n(\theta)$ contains at least
$n-r-m+1$ factors $\mathcal I_p(\theta)$. The number of such summands is $O(n^{m-1})$, since one of the $m$
special positions must be chosen from the fixed set $\{1,\dots,r\}$.
Therefore,
\[
|R_n(\theta)| \leq C n^{m-1}\mathcal I_p(\theta)^{n-r-m+1},
\qquad 0\le \theta\le \delta.
\]
Applying Lemma~\ref{lem:endpoint_laplace_multiplicative} once more, now with
$H\equiv 1$ and $N:=n-r-m+1\sim n$, we obtain
\[
\int_0^\delta \theta^{\frac m2-1}|R_n(\theta)|\,d\theta
\leq
C n^{m-1} \int_0^\delta \theta^{\frac m2-1}I_p(\theta)^{n-r-m+1} \,d\theta
=
O\!\left(n^{m-1}\mathcal I_p(0)^{n-r-m+1}n^{-m/2}\right).
\]
This is smaller by a factor $n^{-1}$ than the asymptotic~\eqref{eq:right_edge_proof_int_0_delta_endpoint_laplace} for the main part
$\int_0^\delta \theta^{m/2-1}M_n(\theta)\,d\theta$.

\medskip
\noindent
\textit{Step 4: Taking pieces together.}
Combining Steps 2 and 3 we obtain
\begin{align*}
\int_0^\infty \theta^{\frac m2-1}S_n(\theta)\,d\theta
&\sim
\frac{\Gamma\!\left(\frac m2\right)n^{m/2} }{(m-1)!}\,
\mathcal I_p(0)^{\,n-r-m/2}\,
\mathcal K_p(0)^{\,1-m/2}\,
\mathcal J_p(0)^{\,m-1}
\prod_{\ell=1}^r \mathcal F_p(0;\lambda_\ell).
\end{align*}
Substituting this into~\eqref{eq:I_n_n-m_as_int_of_S} yields
\begin{align*}
I_{n, n-m}(\lambda_1,\dots,\lambda_r)
&\sim
\frac{p(p-1)^{m-1}}{2\pi^{m/2}}
\frac{\Gamma\!\left(\frac m2\right)\,  n^{m/2}}{(m-1)!}
\frac{
\mathcal I_p(0)^{\,n-r-m/2}\,
\mathcal K_p(0)^{\,1-m/2}\,
\mathcal J_p(0)^{\,m-1}
\prod_{\ell=1}^r \mathcal F_p(0;\lambda_\ell)
}{
\Gamma\!\left(\frac{n+\Lambda+p-m}{p}\right)
}.
\end{align*}
Using
\[
\mathcal I_p(0)=\frac{2}{p}\Gamma\!\left(\frac1p\right),\quad
\mathcal J_p(0)=\frac{2}{p}\Gamma\!\left(1-\frac1p\right), \quad
\mathcal K_p(0)=\frac{2}{p}\Gamma\!\left(2-\frac1p\right),\quad
\mathcal F_p(0;\lambda_\ell)=\frac{2}{p}\Gamma\!\left(\frac{\lambda_\ell+1}{p}\right),
\]
and straightforward simplifications gives
\[
I_{n, n-m}(\lambda_1,\dots,\lambda_r)
\sim
\frac{\Gamma\!\left(\frac m2\right)}{2(m-1)!}
\left(
\frac{p(p-1)\Gamma\!\left(1-\frac1p\right)}
{\pi\,\Gamma\!\left(\frac1p\right)}
\right)^{m/2}
\left(\prod_{\ell=1}^r
\frac{\Gamma\!\left(\frac{\lambda_\ell+1}{p}\right)}
{\Gamma\!\left(\frac1p\right)}
\right)
\frac{\left(\frac{2}{p}\Gamma\!\left(\frac1p\right)\right)^n\,n^{m/2}}
{\Gamma\!\left(\frac{n+\Lambda+p-m}{p}\right)},
\]
which is precisely the formula stated in Part~(i) of Proposition~\ref{prop:laplace_asymptotics_right_edge}. The remaining assertions of Proposition~\ref{prop:laplace_asymptotics_right_edge} are straightforward.
\end{proof}

\subsection{Proof of Theorem~\ref{theo:maxwell_curvature}}
Fix $1<p<\infty$. For sufficiently large $n$,  let $j(n) \in \{0,\ldots, n-1\}$. Let  $X_n=(X_{1;n},\ldots, X_{n;n})$ be a random point on $\partial \mathbb B_p^n$ distributed according to the normalized curvature measure $\Phi_{j(n)}(\mathbb B_p^n, \cdot)/ V_{j(n)}(\mathbb B_p^n)$.

\medskip
\noindent
\textit{The bulk regime.}
Let  $j(n) /n \to\alpha \in (0,1)$ as $n\to\infty$. Since the distributions of $(X_{1;n}, \ldots, X_{r;n})$ and $(\xi_1^{(\alpha)},\ldots, \xi_r^{(\alpha)})$ are invariant w.r.t.\ sign flips of the arbitrary number of  coordinates, it suffices to show that
\begin{equation}\label{eq:maxwell_central_regime_proof_abs_vals}
(n^{1/p}|X_{1;n}|, \ldots, n^{1/p}|X_{r;n}|) \toweak ((p/\alpha)^{1/p} |\xi_1^{(\alpha)}|,\ldots, (p/\alpha)^{1/p} |\xi_r^{(\alpha)}|).
\end{equation}
We show the pointwise convergence of mixed moments (multivariate Mellin transforms) in a suitable half-plane, which implies weak convergence. Take some $\lambda_1,\ldots, \lambda_r \in \C$ with sufficiently large $\Re \lambda_\ell$ and let $\Lambda := \lambda_1+\ldots + \lambda_r$. Part~(ii) of  Proposition~\ref{prop:asymptotics_laplace_central} yields
\begin{align}
\E \left[\prod_{j=1}^r (n^{1/p}|X_{j;n}|)^{\lambda_j} \right]
&= n^{\Lambda/p} \frac{I_{n,j(n)}(\lambda_1,\dots,\lambda_r)}{I_{n,j(n)}(0,\ldots, 0)}
\notag\\
&\to
\left(\frac{p}{\alpha}\right)^{\Lambda/p}
\prod_{j=1}^r
\left(
\alpha\frac{\mathcal F_p(\theta_{p,\alpha};\lambda_j)}{\mathcal F_p(\theta_{p,\alpha};0)}
+
(1-\alpha)\frac{\mathcal F_p(\theta_{p,\alpha};\lambda_j+p-2)}{\mathcal F_p(\theta_{p,\alpha};p-2)}
\right),
\label{eq:method_moments_central_regime_1}
\end{align}
as $n\to\infty$. On the other hand, for every $j=1,\ldots, r$ the density of $\xi_j^{(\alpha)}$ is given by~\eqref{eq:f_p_alpha_density_in_maxwell_laws} and thus
$$
\E \left[|\xi_j^{(\alpha)}|^{\lambda_j}\right] =  \int_{-\infty}^\infty |u|^{\lambda_j} f_{p,\alpha}(u)\, du = \alpha\frac{\mathcal F_p(\theta_{p,\alpha};\lambda_j)}{\mathcal F_p(\theta_{p,\alpha};0)}
+
(1-\alpha)\frac{\mathcal F_p(\theta_{p,\alpha};\lambda_j+p-2)}{\mathcal F_p(\theta_{p,\alpha};p-2)}.
$$
By independence of $\xi_1^{(\alpha)},\dots, \xi_r^{(\alpha)}$,
\begin{align*}
\E \left[\prod_{j=1}^r ((p/\alpha)^{1/p} |\xi_j^{(\alpha)}|)^{\lambda_j}\right]
=
\left(\frac{p}{\alpha}\right)^{\Lambda/p}
\prod_{j=1}^r
\left(
\alpha\frac{\mathcal F_p(\theta_{p,\alpha};\lambda_j)}{\mathcal F_p(\theta_{p,\alpha};0)}
+
(1-\alpha)\frac{\mathcal F_p(\theta_{p,\alpha};\lambda_j+p-2)}{\mathcal F_p(\theta_{p,\alpha};p-2)}
\right),
\end{align*}
which is precisely the right-hand side of~\eqref{eq:method_moments_central_regime_1}. This proves~\eqref{eq:maxwell_central_regime_proof_abs_vals}.

\medskip
\noindent
\textit{The left-edge regime.}
Let $j(n) = j \in \N\cup\{0\}$ be fixed. By symmetry of the involved distributions w.r.t.\ sign flips, it suffices to show that for every $r\in \N$,
\begin{equation}\label{eq:maxwell_left_edge_regime_proof_abs_vals}
n^{1/p} (|X_{1;n}|, \ldots, |X_{r;n}|) \toweak  (p^{1/p}|\eta_1|,\ldots, p^{1/p}|\eta_r|),
\end{equation}
where the random variables $\eta_1,\ldots, \eta_r$ are independent and each $\eta_i$ has the density $g_p$ given by~\eqref{eq:g_p_u_density_maxwell}.

Fix $\lambda_1,\ldots, \lambda_r\in \mathbb C$ with sufficiently large $\Re \lambda_\ell$. Part~(ii) of Proposition~\ref{prop:laplace_asymptotics_left_edge} implies that
\begin{equation}\label{eq:method_moments_left_edge_regime_1}
\E \left[\prod_{\ell=1}^r (n^{1/p}|X_{\ell;n}|)^{\lambda_\ell} \right]
=
n^{\Lambda/p} \frac{I_{n,j}(\lambda_1,\dots,\lambda_r)}{I_{n,j}(0,\dots,0)}
\to
\left(\prod_{\ell=1}^r
\frac{\Gamma\!\left(\frac{\lambda_\ell+p-1}{2p-2}\right)}{\sqrt{\pi}}
\right)
\left(
\frac{\sqrt\pi}{\Gamma\!\left(\frac{2p-1}{2p-2}\right)}
\right)^{\Lambda/p},
\end{equation}
as $n\to\infty$. On the other hand, using independence, \eqref{eq:g_p_u_density_maxwell} and evaluating a standard gamma integral, we obtain
$$
\E \left[\prod_{\ell=1}^r (p^{1/p} |\eta_\ell|)^{\lambda_\ell}\right]
=
p^{\Lambda/p}
\cdot
\prod_{\ell=1}^r
\int_{-\infty}^\infty |u|^{\lambda_\ell} g_{p}(u)\, du
=
\left(\prod_{\ell=1}^r
\frac{\Gamma\!\left(\frac{\lambda_\ell+p-1}{2p-2}\right)}{\sqrt{\pi}}
\right)
\left(
\frac{\sqrt\pi}{\Gamma\!\left(\frac{2p-1}{2p-2}\right)}
\right)^{\Lambda/p},
$$
which coincides with the right-hand side of~\eqref{eq:method_moments_left_edge_regime_1}. This proves~\eqref{eq:maxwell_left_edge_regime_proof_abs_vals}.

\medskip
\noindent
\textit{The right-edge regime.}
Let $j(n) = n-m$, where $m\in \N$ is fixed. As before, it suffices to show that for every $r\in\N$,
\begin{equation}\label{eq:density_maxwell_right_edge_proof}
n^{1/p} (|X_{1;n}|, \ldots, |X_{r;n}|) \toweak (p^{1/p}|\xi_1|,\ldots,  p^{1/p}|\xi_r|),
\end{equation}
where the random variables $\xi_1,\ldots, \xi_r$ are independent and have the $p$-Gaussian density $f_p$ given by~\eqref{eq:p_gauss_density}.

Fix $\lambda_1,\ldots, \lambda_r\in \mathbb C$ with sufficiently large $\Re \lambda_\ell$. Part~(ii) of  Proposition~\ref{prop:laplace_asymptotics_right_edge} gives
\begin{align}
\E \left[\prod_{\ell=1}^r (n^{1/p}|X_{\ell;n}|)^{\lambda_\ell} \right]
= n^{\Lambda/p} \frac{I_{n,n-m}(\lambda_1,\dots,\lambda_r)}{I_{n,n-m}(0,\ldots, 0)}
\to
p^{\Lambda/p}
\cdot
\prod_{\ell=1}^r
\frac{\Gamma\!\left(\frac{\lambda_\ell+1}{p}\right)}
{\Gamma\!\left(\frac1p\right)},
\label{eq:method_moments_right_edge_regime_1}
\end{align}
as $n\to\infty$. On the other hand, it follows from independence that
$$
\E \left[\prod_{\ell=1}^r (p^{1/p} |\xi_\ell|)^{\lambda_\ell}\right]
=
p^{\Lambda/p}
\cdot
\prod_{\ell=1}^r
\int_{-\infty}^\infty |u|^{\lambda_\ell} f_{p}(u)\, du
=
p^{\Lambda/p}
\cdot
\prod_{\ell=1}^r
\frac{\Gamma\!\left(\frac{\lambda_\ell+1}{p}\right)}
{\Gamma\!\left(\frac1p\right)},
$$
which coincides with the right-hand side of~\eqref{eq:method_moments_right_edge_regime_1}.
This proves~\eqref{eq:density_maxwell_right_edge_proof}.
\hfill $\Box$

\section*{Acknowledgement}
ZK has been supported by the German Research Foundation under Germany's Excellence Strategy  EXC 2044/2 -- 390685587, \textit{Mathematics M\"unster: Dynamics - Geometry - Structure} and by the DFG priority program SPP 2265 \textit{Random Geometric Systems}. JP has been supported by the German Research Foundation under DFG project 516672205 \textit{Limit theorems for the volume of random
projections of $\ell_p$-balls}.
This paper benefited from many helpful and illuminating discussions with ChatGPT and ScholarGPT, which sharpened both the perspective and presentation; the work would have been impossible without these exchanges.

\bibliography{intrinsic_vol_l_p_balls_bib}

\begin{thebibliography}{50}
\providecommand{\natexlab}[1]{#1}
\providecommand{\url}[1]{\texttt{#1}}
\expandafter\ifx\csname urlstyle\endcsname\relax
  \providecommand{\doi}[1]{doi: #1}\else
  \providecommand{\doi}{doi: \begingroup \urlstyle{rm}\Url}\fi

\bibitem[Adamczak et~al.(2024)Adamczak, Pivovarov, and
  Simanjuntak]{AdamczakPivovarovSimanjuntak2024}
R.~Adamczak, P.~Pivovarov, and P.~Simanjuntak.
\newblock Limit theorems for the volumes of small codimensional random sections
  of {$\ell_p^n$}-balls.
\newblock \emph{The Annals of Probability}, 52\penalty0 (1):\penalty0 93--126,
  2024.
\newblock \doi{10.1214/23-AOP1646}.

\bibitem[Barthe and Wolff(2023)]{BartheWolff2023}
F.~Barthe and P.~Wolff.
\newblock Volume properties of high-dimensional {O}rlicz balls.
\newblock In Rados{\l}aw Adamczak, Nathael Gozlan, Karim Lounici, and Mokshay
  Madiman, editors, \emph{High Dimensional Probability IX}, volume~80 of
  \emph{Progress in Probability}, pages 75--95. Birkh{\"a}user, Cham, 2023.
\newblock \doi{10.1007/978-3-031-26979-0_2}.

\bibitem[Barthe et~al.(2005)Barthe, Gu{\'e}don, Mendelson, and
  Naor]{BartheGuedonMendelsonNaor2005}
F.~Barthe, O.~Gu{\'e}don, S.~Mendelson, and A.~Naor.
\newblock A probabilistic approach to the geometry of the $\ell_p^n$-ball.
\newblock \emph{The Annals of Probability}, 33\penalty0 (2):\penalty0 480--513,
  2005.
\newblock \doi{10.1214/009117904000000874}.

\bibitem[Berman(1980)]{Berman1980}
S.~M. Berman.
\newblock Stationarity, isotropy and sphericity in {$\ell_p$}.
\newblock \emph{Zeitschrift f{\"u}r Wahrscheinlichkeitstheorie und Verwandte
  Gebiete}, 54\penalty0 (1):\penalty0 21--23, 1980.
\newblock \doi{10.1007/BF00535348}.

\bibitem[Betke and Henk(1993)]{betke_henk}
U.~Betke and M.~Henk.
\newblock Intrinsic volumes and lattice points of crosspolytopes.
\newblock \emph{Monatsh. Math.}, 115\penalty0 (1-2):\penalty0 27--33, 1993.

\bibitem[Borovkov(1991)]{Borovkov1991}
K.~A. Borovkov.
\newblock On the convergence of projections of uniform distributions on balls.
\newblock \emph{Theory of Probability and Its Applications}, 35\penalty0
  (3):\penalty0 546--550, 1991.
\newblock Translated from \emph{Teoriya Veroyatnostei i ee Primeneniya} 35
  (1990), no. 3, 547--551.

\bibitem[Borwein and Bailey(2004)]{BorweinBailey2004}
J.~M. Borwein and D.~H. Bailey.
\newblock \emph{Mathematics by Experiment: Plausible Reasoning in the 21st
  Century}.
\newblock A K Peters, Natick, MA, 2004.
\newblock ISBN 9781568812113.

\bibitem[Diaconis and Freedman(1987)]{DiaconisFreedman1987}
P.~Diaconis and D.~Freedman.
\newblock A dozen de {F}inetti-style results in search of a theory.
\newblock \emph{Annales de l'IHP Probabilit{\'e}s et statistiques}, 23\penalty0
  (2, suppl.):\penalty0 397--423, 1987.

\bibitem[Federer(1959)]{Federer1959}
H.~Federer.
\newblock Curvature measures.
\newblock \emph{Transactions of the American Mathematical Society}, 93\penalty0
  (3):\penalty0 418--491, 1959.
\newblock \doi{10.1090/S0002-9947-1959-0110078-1}.

\bibitem[Fedoryuk(1977)]{Fedoryuk}
M.~V. Fedoryuk.
\newblock \emph{{Metod perevala. (In Russian)}}.
\newblock Izdat. ``Nauka'', Moscow, 1977.

\bibitem[Fr{\"u}hwirth and Prochno(2024)]{FruehwirthProchno2024}
L.~Fr{\"u}hwirth and J.~Prochno.
\newblock Sanov-type large deviations and conditional limit theorems for
  high-dimensional {O}rlicz balls.
\newblock \emph{Journal of Mathematical Analysis and Applications},
  536\penalty0 (1):\penalty0 128169, 2024.
\newblock \doi{10.1016/j.jmaa.2024.128169}.

\bibitem[Gantert et~al.(2017)Gantert, Kim, and Ramanan]{GantertKimRamanan2017}
N.~Gantert, S.~S. Kim, and K.~Ramanan.
\newblock Large deviations for random projections of {$\ell_p$} balls.
\newblock \emph{The Annals of Probability}, 45\penalty0 (6B):\penalty0
  4419--4476, 2017.
\newblock \doi{10.1214/16-AOP1169}.

\bibitem[Gao(2003)]{Gao2003}
F.~Gao.
\newblock The mean of a maximum likelihood estimator associated with the
  {B}rownian bridge.
\newblock \emph{Electronic Communications in Probability}, 8:\penalty0 1--5,
  2003.
\newblock \doi{10.1214/ECP.v8-1064}.

\bibitem[Gao(2013)]{Gao2013}
F.~Gao.
\newblock Volumes of generalized balls.
\newblock \emph{The American Mathematical Monthly}, 120\penalty0 (2):\penalty0
  130, 2013.
\newblock \doi{10.4169/amer.math.monthly.120.02.130}.

\bibitem[Gao and Vitale(2001)]{GV01}
F.~Gao and R.~A. Vitale.
\newblock Intrinsic volumes of the {B}rownian motion body.
\newblock \emph{Discrete Comput. Geom.}, 26\penalty0 (1):\penalty0 41--50,
  2001.

\bibitem[Gusakova et~al.(2025)Gusakova, Spodarev, and
  Zaporozhets]{GusakovaSpodarevZaporozhets2025}
A.~Gusakova, E.~Spodarev, and D.~Zaporozhets.
\newblock Intrinsic volumes of ellipsoids.
\newblock \emph{Journal of Mathematical Sciences}, 293\penalty0 (1):\penalty0
  83--96, 2025.
\newblock \doi{10.1007/s10958-025-07982-z}.

\bibitem[Hadwiger(1979)]{hadwiger}
H.~Hadwiger.
\newblock Gitterpunktanzahl im {S}implex und {W}ills'sche {V}ermutung.
\newblock \emph{Math. Ann.}, 239\penalty0 (3):\penalty0 271--288, 1979.

\bibitem[Henk and Hern{\'a}ndez~Cifre(2008)]{HenkHernandezCifre2008a}
M.~Henk and M.~A. Hern{\'a}ndez~Cifre.
\newblock Intrinsic volumes and successive radii.
\newblock \emph{Journal of Mathematical Analysis and Applications},
  343\penalty0 (2):\penalty0 733--742, 2008.
\newblock \doi{10.1016/j.jmaa.2008.01.091}.

\bibitem[Henk et~al.(2017)Henk, Richter-Gebert, and
  Ziegler]{HenkRichterGebertZiegler2017}
M.~Henk, J.~Richter-Gebert, and G.~M. Ziegler.
\newblock Basic properties of convex polytopes.
\newblock In Jacob~E. Goodman, Joseph O'Rourke, and Csaba~D. T{\'o}th, editors,
  \emph{Handbook of Discrete and Computational Geometry}, pages 383--413. CRC
  Press, 3 edition, 2017.
\newblock \doi{10.1201/9781315119601-15}.

\bibitem[Hug and Weil(2020)]{HugWeil2020}
D.~Hug and W.~Weil.
\newblock \emph{Lectures on convex geometry}, volume 286 of \emph{Graduate
  Texts in Mathematics}.
\newblock Springer, Cham, 2020.
\newblock \doi{10.1007/978-3-030-50180-8}.

\bibitem[Johnston and Prochno(2023)]{JohnstonProchno2023}
S.~G.~G. Johnston and J.~Prochno.
\newblock A {Maxwell} principle for generalized {O}rlicz balls.
\newblock \emph{Annales de l'Institut Henri Poincar{\'e}, Probabilit{\'e}s et
  Statistiques}, 59\penalty0 (3):\penalty0 1223--1247, 2023.
\newblock \doi{10.1214/22-AIHP1298}.

\bibitem[Jubin(2023)]{Jubin2023}
B.~Jubin.
\newblock Intrinsic volumes of sublevel sets.
\newblock \emph{Annales de la Facult{\'e} des sciences de Toulouse :
  Math{\'e}matiques}, 32\penalty0 (5):\penalty0 911--938, 2023.
\newblock \doi{10.5802/afst.1758}.
\newblock URL
  \url{https://afst.centre-mersenne.org/articles/10.5802/afst.1758/}.

\bibitem[Kabluchko and Prochno(2021)]{KabluchkoProchno2021}
Z.~Kabluchko and J.~Prochno.
\newblock The maximum entropy principle and volumetric properties of {O}rlicz
  balls.
\newblock \emph{Journal of Mathematical Analysis and Applications},
  495\penalty0 (1):\penalty0 124687, 2021.
\newblock \doi{10.1016/j.jmaa.2020.124687}.

\bibitem[Kabluchko and Schange(2025)]{KabluchkoSchange2025}
Z.~Kabluchko and P.~Schange.
\newblock Angles of orthocentric simplices, 2025.
\newblock URL \url{https://arxiv.org/abs/2505.05048}.

\bibitem[Kabluchko and Zaporozhets(2014)]{KabluchkoZaporozhets2014}
Z.~Kabluchko and D.~Zaporozhets.
\newblock Random determinants, mixed volumes of ellipsoids, and zeros of
  {G}aussian random fields.
\newblock \emph{Journal of Mathematical Sciences}, 199\penalty0 (2):\penalty0
  168--173, 2014.
\newblock \doi{10.1007/s10958-014-1844-9}.
\newblock Translated from \emph{Zapiski Nauchnykh Seminarov POMI} 408 (2012),
  187--196.

\bibitem[Kabluchko and Zaporozhets(2016)]{kabluchko_zaporozhets_sobolev}
Z.~Kabluchko and D.~Zaporozhets.
\newblock Intrinsic volumes of {S}obolev balls with applications to {B}rownian
  convex hulls.
\newblock \emph{Trans. Amer. Math. Soc.}, 368\penalty0 (12):\penalty0
  8873--8899, 2016.
\newblock ISSN 0002-9947,1088-6850.
\newblock \doi{10.1090/tran/6628}.
\newblock URL \url{https://doi.org/10.1090/tran/6628}.

\bibitem[Kabluchko and Zaporozhets(2019)]{KabluchkoZaporozhets2019}
Z.~Kabluchko and D.~Zaporozhets.
\newblock Expected volumes of {G}aussian polytopes, external angles, and
  multiple order statistics.
\newblock \emph{Transactions of the American Mathematical Society},
  372\penalty0 (3):\penalty0 1709--1733, 2019.
\newblock \doi{10.1090/tran/7708}.

\bibitem[Kempka and Vyb{\'\i}ral(2017)]{KempkaVybiral2017}
H.~Kempka and J.~Vyb{\'\i}ral.
\newblock Volumes of unit balls of mixed sequence spaces.
\newblock \emph{Mathematische Nachrichten}, 290\penalty0 (8--9):\penalty0
  1317--1327, 2017.
\newblock \doi{10.1002/mana.201500414}.

\bibitem[Kim and Ramanan(2018)]{KimRamanan2018}
S.~S. Kim and K.~Ramanan.
\newblock A conditional limit theorem for high-dimensional {$\ell_p$}-spheres.
\newblock \emph{Journal of Applied Probability}, 55\penalty0 (4):\penalty0
  1060--1077, 2018.
\newblock \doi{10.1017/jpr.2018.71}.

\bibitem[Liao and Ramanan(2024)]{LiaoRamanan2024}
Y.-T. Liao and K.~Ramanan.
\newblock Geometric sharp large deviations for random projections of
  {$\ell_p^n$} spheres and balls.
\newblock \emph{Electronic Journal of Probability}, 29:\penalty0 1--56, 2024.
\newblock \doi{10.1214/23-EJP1020}.

\bibitem[Lotz et~al.(2020)Lotz, McCoy, Nourdin, Peccati, and
  Tropp]{LotzMcCoyNourdinPeccatiTropp2020}
M.~Lotz, M.~B. McCoy, I.~Nourdin, G.~Peccati, and J.~A. Tropp.
\newblock Concentration of the intrinsic volumes of a convex body.
\newblock In Bo'az Klartag and Emanuel Milman, editors, \emph{Geometric Aspects
  of Functional Analysis}, volume 2266 of \emph{Lecture Notes in Mathematics},
  pages 139--167. Springer, Cham, 2020.
\newblock \doi{10.1007/978-3-030-46762-3_7}.

\bibitem[McCoy and Tropp(2014)]{McCoyTropp2014}
M.~B. McCoy and J.~A. Tropp.
\newblock From {S}teiner formulas for cones to concentration of intrinsic
  volumes.
\newblock \emph{Discrete \& Computational Geometry}, 51\penalty0 (4):\penalty0
  926--963, 2014.
\newblock \doi{10.1007/s00454-014-9595-4}.

\bibitem[McMullen(1991)]{McMullen1991}
P.~McMullen.
\newblock Inequalities between intrinsic volumes.
\newblock \emph{Monatshefte f{\"u}r Mathematik}, 111\penalty0 (1):\penalty0
  47--53, 1991.
\newblock \doi{10.1007/BF01299276}.

\bibitem[Mogul'ski{\u\i}(1991)]{Mogulskii1991English}
A.~A. Mogul'ski{\u\i}.
\newblock De {F}inetti-type results for {$\ell_p$}.
\newblock \emph{Siberian Mathematical Journal}, 32\penalty0 (4):\penalty0
  609--616, 1991.
\newblock \doi{10.1007/BF00972979}.

\bibitem[Naor(2007)]{Naor2007}
A.~Naor.
\newblock The surface measure and cone measure on the sphere of {$\ell_p^n$}.
\newblock \emph{Transactions of the American Mathematical Society},
  359\penalty0 (3):\penalty0 1045--1075, 2007.
\newblock \doi{10.1090/S0002-9947-06-03939-0}.

\bibitem[Naor and Romik(2003)]{NaorRomik2003}
A.~Naor and D.~Romik.
\newblock Projecting the surface measure of the sphere of {$\ell_p^n$}.
\newblock \emph{Annales de l'I.H.P. Probabilit{\'e}s et statistiques},
  39\penalty0 (2):\penalty0 241--261, 2003.
\newblock \doi{10.1016/S0246-0203(02)00008-0}.

\bibitem[Prochno(2026)]{P2026}
J.~Prochno.
\newblock The {L}arge and {M}oderate {D}eviations {A}pproach in {G}eometric
  {F}unctional {A}nalysis.
\newblock In \emph{High {D}imensional {P}robability {X}}, volume~82 of
  \emph{Progr. Probab.}, pages 347--421. Birkh\"auser/Springer, Cham, 2026.
\newblock ISBN 978-3-032-06056-3; 978-3-032-06057-0.
\newblock \doi{10.1007/978-3-032-06057-0\_12}.
\newblock URL \url{https://doi.org/10.1007/978-3-032-06057-0_12}.

\bibitem[Prochno et~al.(2018)Prochno, Th{\"a}le, and
  Turchi]{ProchnoThaleTurchi2018}
J.~Prochno, C.~Th{\"a}le, and N.~Turchi.
\newblock Geometry of {$\ell_p^n$}-balls: Classical results and recent
  developments, 2018.
\newblock URL \url{https://arxiv.org/abs/1808.10435}.

\bibitem[Prochno et~al.(2024)Prochno, Th{\"a}le, and
  Tuchel]{ProchnoThaleTuchel2024}
J.~Prochno, C.~Th{\"a}le, and P.~Tuchel.
\newblock Limit theorems for the volume of random projections and sections of
  {$\ell_p^N$}-balls, 2024.
\newblock URL \url{https://arxiv.org/abs/2412.16054}.

\bibitem[Rachev and R{\"u}schendorf(1991)]{RachevRuschendorf1991}
S.~T. Rachev and L.~R{\"u}schendorf.
\newblock Approximate independence of distributions on spheres and their
  stability properties.
\newblock \emph{The Annals of Probability}, 19\penalty0 (3):\penalty0
  1311--1337, 1991.
\newblock \doi{10.1214/aop/1176990346}.

\bibitem[Rivin(2007)]{Rivin2007}
I.~Rivin.
\newblock Surface area and other measures of ellipsoids.
\newblock \emph{Advances in Applied Mathematics}, 39\penalty0 (4):\penalty0
  409--427, 2007.
\newblock \doi{10.1016/j.aam.2006.08.009}.

\bibitem[Ruben(1960)]{ruben}
H.~Ruben.
\newblock On the geometrical moments of skew-regular simplices in
  hyperspherical space, with some applications in geometry and mathematical
  statistics.
\newblock \emph{Acta Math.}, 103:\penalty0 1--23, 1960.

\bibitem[Schechtman and Zinn(1990)]{SchechtmanZinn1990}
G.~Schechtman and J.~Zinn.
\newblock On the volume of the intersection of two {$\ell_p^n$} balls.
\newblock \emph{Proceedings of the American Mathematical Society}, 110\penalty0
  (1):\penalty0 217--224, 1990.
\newblock \doi{10.1090/S0002-9939-1990-1015684-0}.

\bibitem[Schneider(1993)]{Schneider1993}
R.~Schneider.
\newblock Convex surfaces, curvature and surface area measures.
\newblock In Peter~M. Gruber and J{\"o}rg~M. Wills, editors, \emph{Handbook of
  Convex Geometry, Vol. A}, pages 273--299. North-Holland, Amsterdam, 1993.

\bibitem[Schneider(2014)]{SchneiderBook}
R.~Schneider.
\newblock \emph{Convex {B}odies: the {B}runn-{M}inkowski {T}heory}, volume 151
  of \emph{Encyclopedia of Mathematics and its Applications}.
\newblock Cambridge University Press, Cambridge, expanded edition, 2014.

\bibitem[Schneider(2022)]{schneider_book_convex_cones_probab_geom}
R.~Schneider.
\newblock \emph{Convex cones~--~geometry and probability}, volume 2319 of
  \emph{Lecture Notes in Mathematics}.
\newblock Springer, Cham, 2022.
\newblock \doi{10.1007/978-3-031-15127-9}.
\newblock URL \url{https://doi.org/10.1007/978-3-031-15127-9}.

\bibitem[Schneider and Weil(2008)]{schneider_weil_book}
R.~Schneider and W.~Weil.
\newblock \emph{Stochastic and integral geometry}.
\newblock Probability and its Applications (New York). Springer-Verlag, Berlin,
  2008.
\newblock \doi{10.1007/978-3-540-78859-1}.
\newblock URL \url{https://doi.org/10.1007/978-3-540-78859-1}.

\bibitem[Tee(2005)]{Tee2005}
G.~J. Tee.
\newblock Surface area and capacity of ellipsoids in $n$ dimensions.
\newblock \emph{New Zealand Journal of Mathematics}, 34\penalty0 (2):\penalty0
  165--198, 2005.

\bibitem[Wang(2005)]{Wang2005}
X.~Wang.
\newblock Volumes of generalized unit balls.
\newblock \emph{Mathematics Magazine}, 78\penalty0 (5):\penalty0 390--395,
  2005.
\newblock \doi{10.2307/30044198}.

\bibitem[Wong(2001)]{Wong2001}
R.~Wong.
\newblock \emph{Asymptotic Approximations of Integrals}, volume~34 of
  \emph{Classics in Applied Mathematics}.
\newblock Society for Industrial and Applied Mathematics, Philadelphia, PA,
  2001.
\newblock \doi{10.1137/1.9780898719260}.

\end{thebibliography}
\bibliographystyle{plainnat}

\end{document}